\documentclass{article}
\usepackage{amssymb}
\usepackage{amsfonts}
\usepackage{amsmath}
\usepackage[margin=0.9in]{geometry}
\usepackage{cite}

\newtheorem{theorem}{Theorem}

\newtheorem{corollary}[theorem]{Corollary}

\newtheorem{definition}[theorem]{Definition}

\newtheorem{lemma}[theorem]{Lemma}
\newtheorem{notation}[theorem]{Notation}

\newtheorem{remark}[theorem]{Remark}

\newenvironment{proof}[1][Proof]{\noindent\textbf{#1.} }{\ \rule{0.5em}{0.5em}}
\input{tcilatex}
\begin{document}

\title{Rough differential equation in Banach space driven by\\
weak geometric $p$-rough path}
\author{Terry J. Lyons\thanks{%
University of Oxford and Oxford-Man Institute, Email:
terry.lyons@oxford-man.ox.ac.uk} \ \ \ \ Danyu Yang\thanks{%
University of Oxford and Oxford-Man Institute, Email: yangd@maths.ox.ac.uk}}
\maketitle

\begin{abstract}
By using an explicit ordinary differential equation to approximate the
exponential solution flow, we extend the universal limit theorem to rough
differential equation in Banach space driven by weak geometric rough path,
and give the quantitative dependence of solution in term of the initial
value, vector field and driving rough path.
\end{abstract}

\section{Introduction}

Rough paths theory gives meaning to systems driven by rough signals, and
provides a robust solution which is continuous with respect to the driving
signal in rough path metric. There are several formulations of rough paths
theory. In his original paper, Lyons \cite{Terry98} developed the theory of
rough paths. He works with rough differential equation driven by geometric $%
p $-rough path in Banach space, treats rough differential equation as a
special case of rough integral equation, and proves that the solution exists
uniquely and is robust. Gubinelli \cite{Gubinelli04,Gubinelli10} and Davie 
\cite{AMDavie} define a continuous path to be a solution if its increment on
small intervals is close to high order Euler expansion. Gubinelli introduces
the notions of controlled rough path \cite{Gubinelli04} and branched rough
path \cite{Gubinelli10}, and he can solve differential equation driven by
general (geometric \&\ non-geometric) rough paths. Davie works with general
rough paths when $p<3$, and gives sharp conditions on vector field for the
existence and uniqueness of solution with some impressive examples
delimiting the sharpness. Inspired by Davie's work, Friz and Victoir \cite%
{Friz Victoir} define solution of rough differential equation (RDE) as the
limit of solutions of ordinary differential equations (ODE), and extend the
formulation to general $p\geq 1$ by using geodesic approximation. There are
also alternative approaches by Feyel \& La Pradelle \cite{FeyelPradelle} and
Hu \& Nualart \cite{HuNualart}. For more systematic treatment of the theory,
see Lyons \& Qian \cite{LyonsQian}, Lejay \cite{Lejay}, Lyons, Caruana \& L%
\'{e}vy \cite{Lyonsnotes}, Friz \&\ Victoir \cite{FrizVictoirbook} and Friz
\& Hairer \cite{FrizHairer}.

For rough path theory, there is a considerable gap between lower $p$ and
arbitrary large $p$, and between finite dimensional space and infinite
dimensional space. (So far, only Lyons' approach can deal with rough paths
in infinite dimensional space with arbitrary roughness.) In comparison with
the case when signals are moderate oscillatory, as the roughness of system
increases, there should be some algebraic structure coming in to streamline
the otherwise complicated (if at all possible) calculation. This leads to a
more complete theory and provides an unified resolution. For finite
dimensional space, one can use Arzel\`{a}-Ascoli theorem to prove that the
solution exists when the vector field is $Lip\left( \gamma \right) $ for $%
\gamma >p-1$, but solution may not exist in general Banach space when $%
\gamma \in \left( p-1,p\right) $. Indeed, according to Shkarin \cite{Shkarin}%
, for a large family of Banach spaces (including $L_{p}\left[ 0,1\right] $, $%
1\leq p<\infty $, and $C\left[ 0,1\right] $) and for any $\alpha \in \left(
0,1\right) $, there exists an $\alpha $-H\"{o}lder continuous $f$ such that,
the ordinary differential equation $\dot{x}=f\left( x\right) $ has no
solution in any interval of the real line. In finite dimensional space, one
can use Jacobi matrix to prove continuity of solutions in initial value and
use Whitney's theorem to treat locally Lipschitz vector fields, but Jacobi
matrix and Whitney's theorem are not available in Banach space. Moreover, in
finite dimensional space, the set of signatures of continuous bounded
variation paths explore the truncated Lie group easily, and every weak
geometric $p$-rough path is a geometric $p^{\prime }$-rough path for any $%
p^{\prime }>p$. While for general Banach space, Rashevskii-Chow Theorem \cite%
{Rashevskii,Chow} no longer holds, and there is a part of the algebraic
structure which can not be reached by continuous bounded variation paths.

We work with RDE in Banach space driven by weak geometric rough paths, and
we use an explicit ODE to approximate the truncated exponential solution
flow. Chen \cite{Chen} prove that the logarithm of the signature of a
continuous bounded variation path is a Lie series. Castell and Gaines \cite%
{CastellGaines} use an ODE, whose vector field is a Lie polynomial, to
approximate the truncated exponential solution flow for stochastic
differential equations. Boutaib et al \cite{Youness} use similar ODE to
approximate the (first level) RDE solution in Banach space. We modify the
ODE in \cite{Youness} and use its solution to recover the truncated solution
flow on small intervals. The method of our analysis is based on Davie \cite%
{AMDavie} and Friz \& Victoir \cite{FrizVictoir,FrizVictoirbook}---basically
by comparing the increment of RDE solution on an interval with the solution
of an ODE and building up mathematical induction on the length of the
interval. Another independent work in this direction is Bailleul \cite%
{Bailleul1,Bailleul2}.

We prove that the solution of rough differential equation driven by weak
geometric $p$-rough path exists uniquely when the vector field is $Lip\left(
\gamma \right) $ for $\gamma >p$. Since being a weak geometric rough path is
easier to check (as the authors assumed) than being a geometric rough path,
this moderate extension of Lyons' original theorem could provide certain
convenience when one works in Banach space. As a consequence of our theorem,
the solution of rough differential equation in the sense of Lyons,
Gubinelli, Davie and Friz \& Victoir coincide where their settings overlap.
Moreover, for fixed time interval $\left[ s,t\right] $, there exists an
explicit ordinary differential equation on $\left[ 0,1\right] $ whose
solution at time $1$ is close to the increment on $\left[ s,t\right] $ of
solution to rough differential equation. Since the ordinary differential
equation is explicit, one could use it to simulate the solution of rough
differential equation (as Castell \& Gaines \cite{CastellGaines} did for
stochastic differential equation). The solution of the ordinary differential
equation, as we prove, takes value in nilpotent Lie group, and the error
(between ODE solution and RDE\ solution) is dimension-free and of the same
order as the error of high order Euler expansion. Finally, we prove the
quantitative version of universal limit theorem \cite{Terry98}, and give the
explicit dependence of solution in term of initial value, vector field and
driving rough path, extending Friz and Victoir's continuity result in \cite%
{FrizVictoirbook} to Banach space.

\section{Definitions and Notations}

\subsection{Algebraic Structure}

Let $\mathcal{V}$ be a Banach space. Based on Def 1.25 in \cite{Lyonsnotes},
we define admissible norm on tensor products.

\begin{definition}[admissible norm]
We say that the tensor product of $\mathcal{V}$ is endowed with an
admissible norm, if the following two conditions are satisfied:

1, For integer $n\geq 1$, the symmetric group $S_{n}$ acts by isometries on $%
\mathcal{V}^{\otimes n}$, that is%
\begin{equation}
\left\Vert \sigma v\right\Vert =\left\Vert v\right\Vert \text{, }\forall
\sigma \in S_{n}\text{, }\forall v\in \mathcal{V}^{\otimes n}\text{.}
\label{tensor norm is symmetric}
\end{equation}

2, The tensor product has norm $1$, that is,%
\begin{equation*}
\left\Vert u\otimes v\right\Vert \leq \left\Vert u\right\Vert \left\Vert
v\right\Vert \text{, }\forall u\in \mathcal{V}^{\otimes n}\text{, }\forall
v\in \mathcal{V}^{\otimes m}\text{, }\forall m,n\geq 1\text{.}
\end{equation*}
\end{definition}

For example, injective and projective tensor norms are admissible norms, see 
\cite{Ryan}.

\begin{definition}[{$\mathcal{V}^{\otimes n}$ and $\left[ \mathcal{V}\right]
^{n}$}]
\label{Definition of Vtensorn and Vbracketn}When $n=1$, $\mathcal{V}%
^{\otimes 1}:=\mathcal{V}$ and $\left[ \mathcal{V}\right] ^{1}:=\mathcal{V}$%
. For $n\geq 2$, we select an admissible norm and define $\mathcal{V}%
^{\otimes n}$ and $\left[ \mathcal{V}\right] ^{n}$ respectively as the
closure of (with $\left[ u,v\right] :=u\otimes v-v\otimes u$)%
\begin{gather*}
\left\{ \left. \sum_{k=1}^{m}v_{1}^{k}\otimes \cdots \otimes
v_{n-1}^{k}\otimes v_{n}^{k}\right\vert \left\{ v_{i}^{k}\right\} \subset 
\mathcal{V}\text{, }m\geq 1\right\} \text{,} \\
\left\{ \left. \sum_{k=1}^{m}\left[ v_{1}^{k},\cdots \left[
v_{n-1}^{k},v_{n}^{k}\right] \right] \right\vert \left\{ v_{i}^{k}\right\}
\subset \mathcal{V}\text{, }m\geq 1\right\} \text{,}
\end{gather*}%
w.r.t. the norm selected.
\end{definition}

\begin{notation}
For integers $n\geq k\geq 1$, $\pi _{k}$ denotes the projection of $%
\mathbb{R}
\oplus \mathcal{V}\oplus \cdots \oplus \mathcal{V}^{\otimes n}$ to $\mathcal{%
V}^{\otimes k}$.
\end{notation}

\begin{definition}[$L^{n}\left( \mathcal{V}\right) $]
\label{Definition of Ln(V)}For integer $n\geq 1$, $L^{n}\left( \mathcal{V}%
\right) $ denotes the Banach space 
\begin{equation*}
L^{n}\left( \mathcal{V}\right) :=%
\mathbb{R}
\oplus \mathcal{V}\oplus \cdots \oplus \mathcal{V}^{\otimes n}\text{,}
\end{equation*}%
equipped with the norm%
\begin{equation*}
\left\Vert l\right\Vert :=\sum_{k=0}^{n}\left\Vert \pi _{k}\left( l\right)
\right\Vert \text{, }\forall l\in L^{n}\left( \mathcal{V}\right) \text{.}
\end{equation*}
\end{definition}

\begin{definition}[$T^{n}\left( \mathcal{V}\right) $]
\label{Definition of Tn(V)}For integer $n\geq 1$, define%
\begin{equation*}
T^{n}\left( \mathcal{V}\right) :=1\oplus \mathcal{V}\oplus \cdots \oplus 
\mathcal{V}^{\otimes n}\text{.}
\end{equation*}%
For $g,h\in T^{n}\left( \mathcal{V}\right) $, define $g\otimes h$ and $%
g^{-1} $ by%
\begin{equation*}
g\otimes h:=\sum_{k=0}^{n}\sum_{j=0}^{k}\pi _{j}\left( g\right) \otimes \pi
_{k-j}\left( h\right) \text{, }g^{-1}:=1+\sum_{k=1}^{n}\pi _{k}\left(
\sum_{j=1}^{n}\left( -1\right) ^{j}\left( g-1\right) ^{\otimes j}\right) 
\text{,}
\end{equation*}%
and equip $T^{n}\left( \mathcal{V}\right) $ with $\left\vert \!\left\vert
\!\left\vert \cdot \right\vert \!\right\vert \!\right\vert $\ defined by%
\begin{equation}
\left\vert \!\left\vert \!\left\vert t\right\vert \!\right\vert
\!\right\vert :=\sum_{k=1}^{n}\left\Vert \pi _{k}\left( t\right) \right\Vert
^{\frac{1}{k}}\text{, }\forall t\in T^{n}\left( \mathcal{V}\right) \text{.}
\label{Definition of norm on group}
\end{equation}%
Then $T^{n}\left( \mathcal{V}\right) $ is a nilpotent topological group.
\end{definition}

\begin{definition}
\label{Definition dilation operator}For $\lambda >0$ and integer $n\geq 1$,
define the dilation operator $\delta _{\lambda }:T^{n}\left( \mathcal{V}%
\right) \rightarrow T^{n}\left( \mathcal{V}\right) $ by 
\begin{equation*}
\delta _{\lambda }g:=\sum_{k=0}^{n}\lambda ^{k}\pi _{k}\left( g\right) \text{%
, }\forall g\in 1\oplus \mathcal{V}\oplus \cdots \oplus \mathcal{V}^{\otimes
n}\text{.}
\end{equation*}
\end{definition}

$T^{n}\left( \mathcal{V}\right) $ is nilpotent because $\left[ t^{n},\cdots %
\left[ t^{2},t^{1}\right] \right] =0$, $\forall \left\{ t^{i}\right\}
_{i=1}^{n}\subset T^{n}\left( \mathcal{V}\right) $. $\left\vert \!\left\vert
\!\left\vert \cdot \right\vert \!\right\vert \!\right\vert $ defined at $%
\left( \ref{Definition of norm on group}\right) $ is homogeneous w.r.t.
dilation, but is not a norm because it is not sub-additive. While $%
\left\vert \!\left\vert \!\left\vert \cdot \right\vert \!\right\vert
\!\right\vert $ is equivalent to a norm up to a constant depending on $n$
(see Exercise 7.38 \cite{FrizVictoirbook} where the equivalency extends
naturally to Banach spaces).

\begin{definition}
Define $\exp :\mathcal{V}\oplus \cdots \oplus \mathcal{V}^{\otimes
n}\rightarrow 1\oplus \mathcal{V}\oplus \cdots \oplus \mathcal{V}^{\otimes
n} $ by%
\begin{equation*}
\exp \left( a\right) :=1+\sum_{k=1}^{n}\pi _{k}\left( \sum_{j=1}^{n}\frac{%
a^{\otimes j}}{j!}\right) \text{, }\forall a\in \mathcal{V}\oplus \cdots
\oplus \mathcal{V}^{\otimes n}\text{.}
\end{equation*}%
Define $\log :1\oplus \mathcal{V}\oplus \cdots \oplus \mathcal{V}^{\otimes
n}\rightarrow \mathcal{V}\oplus \cdots \oplus \mathcal{V}^{\otimes n}$ by%
\begin{equation}
\log \left( g\right) :=\sum_{k=1}^{n}\pi _{k}\left( \sum_{j=1}^{n}\frac{%
\left( -1\right) ^{j+1}}{j}\left( g-1\right) ^{\otimes j}\right) \text{, }%
\forall g\in 1\oplus \mathcal{V}\oplus \cdots \oplus \mathcal{V}^{\otimes n}%
\text{.}  \label{definition of logarithm}
\end{equation}
\end{definition}

Then it can be checked that $\log \left( \exp \left( t-1\right) \right) =t-1$
and $\exp \left( \log \left( t\right) \right) =t$, $\forall t\in T^{n}\left( 
\mathcal{V}\right) $.

\begin{definition}[$G^{n}\left( \mathcal{V}\right) $]
For integer $n\geq 1$, (with $\left[ \mathcal{V}\right] ^{k}$ in Definition %
\ref{Definition of Vtensorn and Vbracketn}) we define 
\begin{equation*}
G^{n}\left( \mathcal{V}\right) :=\left\{ \exp \left( a\right) |a\in \left[ 
\mathcal{V}\right] ^{1}\oplus \left[ \mathcal{V}\right] ^{2}\oplus \cdots
\oplus \left[ \mathcal{V}\right] ^{n}\right\} \text{.}
\end{equation*}%
Then $G^{n}\left( \mathcal{V}\right) $ is a subgroup of $T^{n}\left( 
\mathcal{V}\right) $ (based on Baker--Campbell--Hausdorff formula), called
the step-$n$ nilpotent Lie group of degree $n$.
\end{definition}

For more about nilpotent Lie group, please refer to e.g. \cite{ReutenauerC}.

\subsection{Vector Field and Differential Operator}

Let $\mathcal{U}$, $\mathcal{V}$ and $\mathcal{W}$ be Banach spaces.

\begin{definition}
\label{Definition Cgamma}For $\gamma >0$, we say $r:\mathcal{V}\rightarrow 
\mathcal{U}$ is $Lip\left( \gamma \right) $ and denote $r\in C^{\gamma
}\left( \mathcal{V},\mathcal{U}\right) $, if $r$ is $\lfloor \gamma \rfloor $%
-times Fr\'{e}chet differentiable ($\lfloor \gamma \rfloor $ denotes the
largest integer which is strictly less than $\gamma $), and%
\begin{equation*}
\left\vert r\right\vert _{\gamma }:=\left( \max_{k=0,1,\dots ,\lfloor \gamma
\rfloor }\left\Vert D^{k}r\right\Vert _{\infty }\right) \vee \left\Vert
D^{\lfloor \gamma \rfloor }r\right\Vert _{\left( \gamma -\lfloor \gamma
\rfloor \right) -H\ddot{o}l}<\infty \text{,}
\end{equation*}%
where $\left\Vert \cdot \right\Vert _{\infty }$ denotes the uniform norm and 
$\left\Vert \cdot \right\Vert _{\left( \gamma -\lfloor \gamma \rfloor
\right) -H\ddot{o}l}$ denotes the $\left( \gamma -\lfloor \gamma \rfloor
\right) $-H\"{o}lder norm. \noindent

\noindent Denote by \noindent $C^{0}\left( \mathcal{V},\mathcal{U}\right) $
the space of bounded measurable mappings from $\mathcal{V}$ to $\mathcal{U}$.

\noindent Denote by $C^{\gamma ,loc}\left( \mathcal{V},\mathcal{U}\right) $
the space of locally $Lip\left( \gamma \right) $ mappings.
\end{definition}

\begin{definition}
\label{Definition LV Cgamma}$L\left( \mathcal{W},C^{\gamma }\left( \mathcal{V%
},\mathcal{U}\right) \right) $ denotes the space of linear mappings from $%
\mathcal{W}$ to $C^{\gamma }\left( \mathcal{V},\mathcal{U}\right) $. Define%
\begin{equation*}
\left\vert f\right\vert _{\gamma }:=\sup_{w\in \mathcal{W},\left\Vert
w\right\Vert =1}\left\vert f\left( w\right) \right\vert _{\gamma }\text{, }%
\forall f\in L\left( \mathcal{W},C^{\gamma }\left( \mathcal{V},\mathcal{U}%
\right) \right) \text{.}
\end{equation*}

\noindent Similarly, $L\left( \mathcal{W},C^{\gamma ,loc}\left( \mathcal{V},%
\mathcal{U}\right) \right) $ denotes the space of linear mappings from $%
\mathcal{W}$ to $C^{\gamma ,loc}\left( \mathcal{V},\mathcal{U}\right) $.
\end{definition}

For $r\in C^{k,loc}\left( \mathcal{U},\mathcal{U}\right) $ and $j=0,\dots ,k$%
, $D^{j}r\in L\left( \mathcal{U}^{\otimes j},C^{k-j,loc}\left( \mathcal{U},%
\mathcal{U}\right) \right) $.

\begin{notation}[$\mathcal{D}^{k}\left( \mathcal{U}\right) $]
For integer $k\geq 0$, denote by $\mathcal{D}^{k}\left( \mathcal{U}\right) $
the set of locally bounded $k$th order differential operators (on $%
C^{k,loc}\left( \mathcal{U},\mathcal{U}\right) $). More specifically, $p\in 
\mathcal{D}^{k}\left( \mathcal{U}\right) $ if and only if $p:C^{k,loc}\left( 
\mathcal{U},\mathcal{U}\right) \rightarrow C^{0,loc}\left( \mathcal{U},%
\mathcal{U}\right) $ and there exist locally bounded $p_{j}\in
C^{0,loc}\left( \mathcal{U},\mathcal{U}^{\otimes j}\right) $, $j=0,1,\dots
,k $, with $p_{k}\not\equiv 0$, such that%
\begin{equation*}
p\left( r\right) \left( u\right) =\sum_{j=0}^{k}\left( D^{j}r\right) \left(
p_{j}\left( u\right) \right) \left( u\right) \text{, }\forall u\in \mathcal{U%
}\text{, }\forall r\in C^{k,loc}\left( \mathcal{U},\mathcal{U}\right) \text{.%
}
\end{equation*}%
We define the norm $\left\vert \cdot \right\vert _{k}$ on $\mathcal{D}%
^{k}\left( \mathcal{U}\right) $ by%
\begin{equation*}
\left\vert p\right\vert _{k}:=\max_{j=0,1,\dots ,k}\sum_{n=1}^{\infty }\frac{%
\left( \sup_{\left\Vert u\right\Vert \leq n}\left\Vert p_{j}\left( u\right)
\right\Vert \right) \wedge 1}{2^{n}}\text{, }\forall p\in \mathcal{D}%
^{k}\left( \mathcal{U}\right) \text{.}
\end{equation*}%
Then $\mathcal{D}^{k}\left( \mathcal{U}\right) $ is a Banach space (with the
natural addition and scalar multiplication).
\end{notation}

\begin{definition}[composition]
\label{Definition differential operator f}Let $p^{1}\in \mathcal{D}%
^{j_{1}}\left( \mathcal{U}\right) $ and $p^{2}\in \mathcal{D}^{j_{2}}\left( 
\mathcal{U}\right) $ for integers $j_{1}\geq 0$, $j_{2}\geq 0$. When the
components of $p^{2}$ are locally $Lip\left( j_{1}\right) $, we define the
composition of $p^{1}$ and $p^{2}$, $p^{1}\circ p^{2}\in \mathcal{D}%
^{j_{1}+j_{2}}\left( \mathcal{U}\right) $, by%
\begin{equation*}
\left( p^{1}\circ p^{2}\right) \left( r\right) :=p^{1}\left( p^{2}\left(
r\right) \right) \text{, }\forall r\in C^{j_{1}+j_{2},loc}\left( \mathcal{U},%
\mathcal{U}\right) \text{.}
\end{equation*}%
For $p\in \mathcal{D}^{j}\left( \mathcal{U}\right) $, $j\geq 0$, when the
components of $p$ are locally $Lip\left( \left( k-1\right) \times j\right) $
for integer $k\geq 1$, we define the differential operator $p^{\circ k}\in 
\mathcal{D}^{k\times j}\left( \mathcal{U}\right) $ by 
\begin{equation}
p^{\circ 1}:=p\text{ \ and \ }p^{\circ k}:=p\circ p^{\circ \left( k-1\right)
}\text{, }k\geq 2\text{.}  \label{Notation of composition of f with itself}
\end{equation}
\end{definition}

Compositions of differential operators are associative, i.e. $\left(
p^{1}\circ p^{2}\right) \circ p^{3}=p^{1}\circ \left( p^{2}\circ
p^{3}\right) $.

\begin{definition}[$F^{\circ k}$]
\label{Definition of f circ k}Suppose $F\in L\left( \mathcal{V},C^{\gamma
,loc}\left( \mathcal{U},\mathcal{U}\right) \right) $ for some $\gamma \geq 0$%
. Then for any $v\in \mathcal{V}$, we define $F^{\circ 1}\left( v\right) \in 
\mathcal{D}^{1}\left( \mathcal{U}\right) $ by%
\begin{equation}
F^{\circ 1}\left( v\right) \left( r\right) \left( u\right) :=\left(
Dr\right) \left( F\left( v\right) \left( u\right) \right) \left( u\right) 
\text{, }\forall u\in \mathcal{U}\text{, }\forall r\in C^{1,loc}\left( 
\mathcal{U},\mathcal{U}\right) \text{.}  \label{definition of Fcirc1}
\end{equation}%
For integer $k\in 1,2,\dots ,\lfloor \gamma \rfloor +1$ and $\left\{
v_{j}\right\} _{j=1}^{k}\subset \mathcal{V}$, we define $F^{\circ k}\left(
v_{1}\otimes \cdots \otimes v_{k}\right) \in \mathcal{D}^{k}\left( \mathcal{U%
}\right) $ by 
\begin{equation}
F^{\circ k}\left( v_{1}\otimes \cdots \otimes v_{k}\right) :=\left( F^{\circ
1}\left( v_{1}\right) \right) \circ \left( F^{\circ 1}\left( v_{2}\right)
\right) \circ \cdots \circ \left( F^{\circ 1}\left( v_{k}\right) \right) .
\label{definition of fcirc k on elements}
\end{equation}%
Then we denote by $F^{\circ k}\in L\left( \mathcal{V}^{\otimes k},\mathcal{D}%
^{k}\left( \mathcal{U}\right) \right) $ the unique continuous linear
operator which satisfies $\left( \ref{definition of fcirc k on elements}%
\right) $.
\end{definition}

\subsection{Rough Differential Equation}

\label{Section of RDE}Recall $\left\vert \!\left\vert \!\left\vert \cdot
\right\vert \!\right\vert \!\right\vert :=\sum_{k=1}^{\left[ p\right]
}\left\Vert \pi _{k}\left( \cdot \right) \right\Vert ^{\frac{1}{k}}$ defined
at $\left( \ref{Definition of norm on group}\right) $. $\left[ p\right] $
denotes the largest integer which is less or equal to $p$.

\begin{definition}
For $p\geq 1$, suppose $X:\left[ 0,T\right] \rightarrow G^{\left[ p\right]
}\left( \mathcal{V}\right) $ is continuous. Define $p$-variation of $X$ by ($%
X_{s,t}:=X_{s}^{-1}\otimes X_{t}$) 
\begin{equation*}
\left\Vert X\right\Vert _{p-var,\left[ 0,T\right] }:=\sup_{D\subset \left[
0,T\right] }\left( \tsum\nolimits_{j,t_{j}\in D}\left\vert \!\left\vert
\!\left\vert X_{t_{j},t_{j+1}}\right\vert \!\right\vert \!\right\vert
^{p}\right) ^{\frac{1}{p}}\text{,}
\end{equation*}%
where the supremum is taken over all finite partitions $D=\left\{
t_{j}\right\} _{j=0}^{n}$, $0=t_{0}<t_{1}<\cdots <t_{n}=T$. Let $%
C^{p-var}\left( \left[ 0,T\right] ,G^{\left[ p\right] }\left( \mathcal{V}%
\right) \right) $ denote the set of continuous paths with finite $p$%
-variation.
\end{definition}

\begin{definition}[weak geometric rough path]
For $p\geq 1$, $X:\left[ 0,T\right] \rightarrow G^{\left[ p\right] }\left( 
\mathcal{V}\right) $ is called a weak geometric $p$-rough path if $X\in
C^{p-var}\left( \left[ 0,T\right] ,G^{\left[ p\right] }\left( \mathcal{V}%
\right) \right) $.
\end{definition}

Recall Banach space $L^{n}\left( \mathcal{U}\right) =%
\mathbb{R}
\oplus \mathcal{U}\oplus \cdots \oplus \mathcal{U}^{\otimes n}$ defined in
Definition \ref{Definition of Ln(V)}. For $l\in L^{n}\left( \mathcal{U}%
\right) $ and $u\in \mathcal{U}$, $l\otimes u\in L^{n}\left( \mathcal{U}%
\right) $ is defined by $\pi _{0}\left( l\otimes u\right) =0$ and $\pi
_{k}\left( l\otimes u\right) =\pi _{k-1}\left( l\right) \otimes u$, $%
k=1,2,\dots ,n$.

\begin{notation}
For $\gamma \geq 0$, suppose $f\in L\left( \mathcal{V},C^{\gamma }\left( 
\mathcal{U},\mathcal{U}\right) \right) $ and $\eta \in \mathcal{U}$. Denote $%
f\left( \cdot +\eta \right) \in L\left( \mathcal{V},C^{\gamma }\left( 
\mathcal{U},\mathcal{U}\right) \right) $ by%
\begin{equation}
f\left( \cdot +\eta \right) \left( v\right) \left( u\right) :=f\left(
v\right) \left( u+\eta \right) \text{, }\forall v\in \mathcal{V}\text{, }%
\forall u\in \mathcal{U}\text{.}  \label{Notation of f(cdot+xi)}
\end{equation}%
Denote $F\left( f\right) \in L\left( \mathcal{V},C^{\gamma ,loc}\left(
L^{n}\left( \mathcal{U}\right) ,L^{n}\left( \mathcal{U}\right) \right)
\right) $ by 
\begin{equation}
F\left( f\right) \left( v\right) \left( l\right) :=l\otimes \left( f\left(
v\right) \left( \pi _{1}\left( l\right) \right) \right) \text{, }\forall
v\in \mathcal{V}\text{, }\forall l\in L^{n}\left( \mathcal{U}\right) \text{.}
\label{Notation of F(f)}
\end{equation}
\end{notation}

Gubinelli \cite{Gubinelli04,Gubinelli10} and Davie \cite{AMDavie} define a
continuous path $Y$ to be a solution, if the increment of $Y$ on small
interval is comparable to high order Euler expansion. (Gubinelli's
formulation is more algebraic, but his solution could be stated in this way.)

For $f\in L\left( \mathcal{V},C^{\gamma }\left( \mathcal{U},\mathcal{U}%
\right) \right) $ and $\eta \in \mathcal{U}$, denote $f\left( \cdot +\eta
\right) \in L\left( \mathcal{V},C^{\gamma }\left( \mathcal{U},\mathcal{U}%
\right) \right) $ as at $\left( \ref{Notation of f(cdot+xi)}\right) $,
denote $F\left( f\left( \cdot +\eta \right) \right) \in L\left( \mathcal{V}%
,C^{\gamma ,loc}\left( L^{n}\left( \mathcal{U}\right) ,L^{n}\left( \mathcal{U%
}\right) \right) \right) $ as at $\left( \ref{Notation of F(f)}\right) $ and
define $F\left( f\left( \cdot +\eta \right) \right) ^{\circ k}\in L\left( 
\mathcal{V}^{\otimes k},\mathcal{D}^{k}\left( L^{n}\left( \mathcal{U}\right)
\right) \right) $ as in Definition \ref{Definition of f circ k}. Denote $%
Id_{L^{n}\left( \mathcal{U}\right) }$ the identity function on $L^{n}\left( 
\mathcal{U}\right) $, i.e. $Id_{L^{n}\left( \mathcal{U}\right) }\left(
l\right) =l$, $\forall l\in L^{n}\left( \mathcal{U}\right) $.

\begin{definition}[Gubinelli/Davie]
\label{Definition of solution of RDE}For $\gamma >p\geq 1$, suppose $X\in
C^{p-var}\left( \left[ 0,T\right] ,G^{\left[ p\right] }\left( \mathcal{V}%
\right) \right) $, $f\in L\left( \mathcal{V},C^{\gamma }\left( \mathcal{U},%
\mathcal{U}\right) \right) $ and $\xi \in G^{\left[ p\right] }\left( 
\mathcal{U}\right) $. Then $Y:\left[ 0,T\right] \rightarrow T^{\left[ p%
\right] }\left( \mathcal{U}\right) $ is said to be a solution of the rough
differential equation%
\begin{equation*}
dY=f\left( Y\right) dX\text{, \ }Y_{0}=\xi \text{,}
\end{equation*}%
if there exists a function $\theta :\left\{ 0\leq s\leq t\leq T\right\}
\rightarrow \overline{%
\mathbb{R}
^{+}}$ satisfying%
\begin{equation*}
\lim_{D\subset \left[ 0,T\right] ,\left\vert D\right\vert \rightarrow
0}\sum_{j,t_{j}\in D}\theta \left( t_{j},t_{j+1}\right) =0\text{,}
\end{equation*}%
such that, for all sufficiently small $\left[ s,t\right] \subseteq \left[ 0,T%
\right] $, (with $Y_{s,t}:=Y_{s}^{-1}\otimes Y_{t}$ and $1\in L^{\left[ p%
\right] }\left( \mathcal{U}\right) $)%
\begin{equation*}
\left\Vert Y_{s,t}-\tsum\nolimits_{k=1}^{\left[ p\right] }F\left( f\left(
\cdot +\pi _{1}\left( Y_{s}\right) \right) \right) ^{\circ k}\pi _{k}\left(
X_{s,t}\right) \left( Id_{L^{\left[ p\right] }\left( \mathcal{U}\right)
}\right) \left( 1\right) \right\Vert \leq \theta \left( s,t\right) \text{.}
\end{equation*}
\end{definition}

\noindent As will be apparent in the proofs, the shuffle product (used in 
\cite{Terry98}) is hidden in $F\left( f\left( \cdot +\pi _{1}\left(
Y_{s}\right) \right) \right) ^{\circ k}\pi _{k}\left( X_{s,t}\right) $.

\section{Main Result}

\begin{definition}
$\omega :\left\{ 0\leq s\leq t\leq T\right\} \rightarrow \overline{%
\mathbb{R}
^{+}}$ is called a control, if $\omega $ is continuous, vanishing on the
diagonal, and is sub-additive i.e. 
\begin{equation*}
\omega \left( s,u\right) +\omega \left( u,t\right) \leq \omega \left(
s,t\right) ,\forall 0\leq s\leq u\leq t\leq T\text{.}
\end{equation*}
\end{definition}

\begin{theorem}
\label{Theorem existence and uniqueness}For $\gamma >p\geq 1$, suppose $X\in
C^{p-var}\left( \left[ 0,T\right] ,G^{\left[ p\right] }\left( \mathcal{V}%
\right) \right) $, $f\in L\left( \mathcal{V},C^{\gamma }\left( \mathcal{U},%
\mathcal{U}\right) \right) $ and $\xi \in G^{\left[ p\right] }\left( 
\mathcal{U}\right) $. Then the rough differential equation%
\begin{equation}
dY=f\left( Y\right) dX\text{, }Y_{0}=\xi \text{,}  \label{RDE}
\end{equation}%
has a unique solution (denoted as $Y$) in the sense of Definition \ref%
{Definition of solution of RDE}, which is a continuous path taking values in 
$G^{\left[ p\right] }\left( \mathcal{U}\right) $. If define control $\omega $
by%
\begin{equation*}
\omega \left( s,t\right) :=\left\vert f\right\vert _{\gamma }^{p}\left\Vert
X\right\Vert _{p-var,\left[ s,t\right] }^{p}\text{, }\forall 0\leq s\leq
t\leq T\text{,}
\end{equation*}%
then there exists a constant $C_{p}$ such that, for any $0\leq s\leq t\leq T$%
,%
\begin{equation}
\left\Vert Y\right\Vert _{p-var,\left[ s,t\right] }\leq C_{p}\left( \omega
\left( s,t\right) ^{\frac{1}{p}}\vee \omega \left( s,t\right) \right) \text{.%
}  \label{bound on RDE solution in main theorem}
\end{equation}%
Moreover, for $0\leq s\leq t\leq T$, if let $y^{s,t}:\left[ 0,1\right]
\rightarrow L^{\left[ p\right] }\left( \mathcal{U}\right) $ denote the
solution of the ordinary differential equation%
\begin{equation}
dy_{u}^{s,t}=\tsum\nolimits_{k=1}^{\left[ p\right] }F\left( f\left( \cdot
+\pi _{1}\left( Y_{s}\right) \right) \right) ^{\circ k}\pi _{k}\left( \log
X_{s,t}\right) \left( Id_{L^{\left[ p\right] }\left( \mathcal{U}\right)
}\right) \left( y_{u}^{s,t}\right) du\text{, }u\in \left[ 0,1\right] \text{, 
}y_{0}^{s,t}=1\text{,}  \label{definition of ys,t in main theorem}
\end{equation}%
then $y^{s,t}$ takes value in $G^{\left[ p\right] }\left( \mathcal{U}\right) 
$, and there exists a constant $C_{p}$, such that, ($Y_{s,t}:=Y_{s}^{-1}%
\otimes Y_{t}$)%
\begin{gather}
\left\Vert Y_{s,t}-y_{1}^{s,t}\right\Vert \leq C_{p}\left( \omega \left(
s,t\right) ^{\frac{\left[ p\right] +1}{p}}\vee \omega \left( s,t\right) ^{%
\left[ p\right] }\right) \text{,}
\label{estimate of difference in main theorem} \\
\left\Vert Y_{s,t}-\tsum\nolimits_{k=1}^{\left[ p\right] }F\left( f\left(
\cdot +\pi _{1}\left( Y_{s}\right) \right) \right) ^{\circ k}\pi _{k}\left(
X_{s,t}\right) \left( Id_{L^{\left[ p\right] }\left( \mathcal{U}\right)
}\right) \left( 1\right) \right\Vert \leq C_{p}\left( \omega \left(
s,t\right) ^{\frac{\left[ p\right] +1}{p}}\vee \omega \left( s,t\right) ^{%
\left[ p\right] }\right) \text{.}  \notag
\end{gather}
\end{theorem}

The proof of Theorem \ref{Theorem existence and uniqueness} starts from p%
\pageref{Proof of Theorem existence and uniqueness}.

\begin{remark}
The solution of $\left( \ref{RDE}\right) $ is defined in Gubinelli/Davie's
sense. Based on Universal Limit Theorem and Theorem \ref{Theorem continuity}
below, when the vector field is $Lip\left( \gamma \right) $ for $\gamma >p$,
the solutions in Lyons \cite{Terry98} and in Friz \&\ Victoir \cite%
{FrizVictoirbook} coincide with our solution.
\end{remark}

\begin{remark}
Based on Euler expansion of solution of ODE ($\left( \ref{equ 1 in Lemma of
ODE}\right) $ in Lemma \ref{Lemma Euler expansion of solution of ODE and
difference between one step and two steps} below) and the definition of RDE
solution (Definition \ref{Definition of solution of RDE}), the solution of
the ODE $\left( \ref{definition of ys,t in main theorem}\right) $ coincides
with the solution of the RDE: 
\begin{equation*}
dY=f\left( Y\right) dX^{s,t}\text{, }Y_{0}=\xi \text{,}
\end{equation*}%
with $X^{s,t}\in C^{p-var}\left( \left[ 0,1\right] ,G^{\left[ p\right]
}\left( \mathcal{V}\right) \right) $ defined by $X_{u}^{s,t}=\exp \left(
u\log X_{s,t}\right) $, $u\in \left[ 0,1\right] $.
\end{remark}

\begin{notation}
Suppose $\omega :\left\{ 0\leq s\leq t\leq T\right\} \rightarrow \overline{%
\mathbb{R}
^{+}}$ is a control. For $\alpha \in (0,1]$, define control 
\begin{equation}
\omega ^{\alpha }\left( s,t\right) :=\sup_{D\subset \left[ s,t\right]
,\omega \left( t_{j},t_{j+1}\right) \leq \alpha }\sum_{j,t_{j}\in D}\omega
\left( t_{j},t_{j+1}\right) \text{, }\forall 0\leq s\leq t\leq T\text{.}
\label{Definition of omega alpha}
\end{equation}%
Suppose $X^{i}\in C^{p-var}\left( \left[ 0,T\right] ,G^{\left[ p\right]
}\left( \mathcal{V}\right) \right) $, $i=1,2$. For control $\omega $,
integer $n=1,2,\dots ,\left[ p\right] $, $\left[ s,t\right] \subseteq \left[
0,T\right] $ and $\alpha \in (0,1]$, denote%
\begin{gather}
d_{p,\left[ s,t\right] }^{n}\left( X^{1},X^{2}\right) :=\left(
\sup_{D\subset \left[ s,t\right] }\sum_{j,t_{j}\in D}\left\Vert \pi
_{n}\left( X_{t_{j},t_{j+1}}^{1}\right) -\pi _{n}\left(
X_{t_{j},t_{j+1}}^{2}\right) \right\Vert ^{\frac{p}{n}}\right) ^{\frac{n}{p}}%
\text{,}  \label{Definition of dpn} \\
d_{p,\left[ s,t\right] }^{n,\alpha }\left( X^{1},X^{2}\right) :=\left(
\sup_{D\subset \left[ s,t\right] ,\omega \left( t_{j},t_{j+1}\right) \leq
\alpha }\sum_{j,t_{j}\in D}\left\Vert \pi _{n}\left(
X_{t_{j},t_{j+1}}^{1}\right) -\pi _{n}\left( X_{t_{j},t_{j+1}}^{2}\right)
\right\Vert ^{\frac{p}{n}}\right) ^{\frac{n}{p}}\text{.}
\label{Definition of dpn^alpha}
\end{gather}
\end{notation}

\begin{theorem}
\label{Theorem continuity}For $i=1,2$ and $\gamma >p\geq 1$, suppose $%
X^{i}\in C^{p-var}\left( \left[ 0,T\right] ,G^{\left[ p\right] }\left( 
\mathcal{V}\right) \right) $, $f^{i}\in L\left( \mathcal{V},C^{\gamma
}\left( \mathcal{U},\mathcal{U}\right) \right) $ and $\xi ^{i}\in G^{\left[ p%
\right] }\left( U\right) $. Let $Y^{i}:\left[ 0,T\right] \rightarrow G^{%
\left[ p\right] }\left( \mathcal{U}\right) $ be the solution of the rough
differential equation%
\begin{equation*}
dY^{i}=f^{i}\left( Y^{i}\right) dX^{i}\text{, }Y_{0}^{i}=\xi ^{i}\text{.}
\end{equation*}%
Define control $\omega \ $by%
\begin{equation*}
\omega \left( s,t\right) =\left\vert f^{1}\right\vert _{\gamma
}^{p}\left\Vert X^{1}\right\Vert _{p-var,\left[ s,t\right] }^{p}+\left\vert
f^{2}\right\vert _{\gamma }^{p}\left\Vert X^{2}\right\Vert _{p-var,\left[ s,t%
\right] }^{p}\text{, }\forall 0\leq s\leq t\leq T\text{.}
\end{equation*}%
For $\alpha \in (0,1]$, define $\omega ^{\alpha }$ and $d_{p}^{n,\alpha }$
based on $\omega $ as at $\left( \ref{Definition of omega alpha}\right) $
and $\left( \ref{Definition of dpn^alpha}\right) $. Then there exists $%
C_{p,\gamma }$ (which only depends on $p$ and $\gamma $) such that, for $%
\alpha \in (0,1]$, $\left[ s,t\right] \subseteq \left[ 0,T\right] $ and $%
k=1,2,\dots ,\left[ p\right] $,%
\begin{eqnarray}
&&d_{p,\left[ s,t\right] }^{k}\left( Y^{1},Y^{2}\right)
\label{formula of continuity in main theorem} \\
&\leq &C_{p,\gamma }\exp \left( C_{p,\gamma }\alpha ^{-1}\omega ^{\alpha
}\left( s,t\right) \right)  \notag \\
&&\times \left( \omega ^{\alpha }\left( s,t\right) ^{\frac{k}{p}}\left(
\left\Vert \pi _{1}\left( Y_{s}^{1}\right) -\pi _{1}\left( Y_{s}^{2}\right)
\right\Vert +\left\vert \frac{f^{1}}{\left\vert f^{1}\right\vert _{\gamma }}-%
\frac{f^{2}}{\left\vert f^{2}\right\vert _{\gamma }}\right\vert _{\gamma
-1}\right) +\sum_{n=1}^{\left[ p\right] }\omega ^{\alpha }\left( s,t\right)
^{\frac{\left( k-n\right) \vee 0}{p}}d_{p,\left[ s,t\right] }^{n,\alpha
}\left( \delta _{\left\vert f^{1}\right\vert _{\gamma }}X^{1},\delta
_{\left\vert f^{2}\right\vert _{\gamma }}X^{2}\right) \right) \text{.} 
\notag
\end{eqnarray}
\end{theorem}

The proof of Theorem \ref{Theorem continuity} starts from p\pageref{Proof of
Theorem continuity}.

Based on Lemma \ref{Lemma log into exp} below and sub-additivity of a
control, $\left( \ref{formula of continuity in main theorem}\right) $ holds
with $X^{i}$ replaced by $\log X^{i}$, $i=1,2$.

According to Cass, Litterer \& Lyons \cite{CassLittererLyons}, for a large
family of Gaussian processes (including fractional Brownian motion when $%
H>4^{-1}$) and any $\alpha \in (0,1]$, $\exp \left( C_{p,\gamma }\alpha
^{-1}\omega ^{\alpha }\left( s,t\right) \right) $ has finite moments of all
orders.

\begin{corollary}
\label{Corollary error of ODE approximation}For $\gamma >p\geq 1$, suppose $%
X\in C^{p-var}\left( \left[ 0,T\right] ,G^{\left[ p\right] }\left( \mathcal{V%
}\right) \right) $, $f\in L\left( \mathcal{V},C^{\gamma }\left( \mathcal{U},%
\mathcal{U}\right) \right) $ and $\xi \in G^{\left[ p\right] }\left( 
\mathcal{U}\right) $. Let $Y:\left[ 0,T\right] \rightarrow G^{\left[ p\right]
}\left( \mathcal{U}\right) $ be the solution of the rough differential
equation%
\begin{equation*}
dY=f\left( Y\right) dX\text{, }Y_{0}=\xi \text{.}
\end{equation*}%
For finite partition $D=\left\{ t_{j}\right\} _{j=0}^{n}$ of $\left[ 0,T%
\right] $, let $y^{D}:\left[ 0,T\right] \rightarrow G^{\left[ p\right]
}\left( \mathcal{V}\right) $ be the solution of the ordinary differential
equation%
\begin{equation}
dy_{u}^{D}=\tsum\nolimits_{k=1}^{\left[ p\right] }F\left( f\right) ^{\circ
k}\pi _{k}\left( \log X_{t_{j},t_{j+1}}\right) \left( Id_{L^{\left[ p\right]
}\left( \mathcal{U}\right) }\right) \left( y_{u}^{D}\right) \frac{du}{%
t_{j+1}-t_{j}}\text{, }u\in \left[ t_{j},t_{j+1}\right] \text{, }%
y_{0}^{D}=\xi \text{.}  \label{expression of concatenated ODEs}
\end{equation}%
Denote control $\omega $ by $\omega \left( s,t\right) :=\left\vert
f\right\vert _{\gamma }^{p}\left\Vert X\right\Vert _{p-var,\left[ s,t\right]
}^{p}$ and denote $\omega ^{\alpha }$ based on $\omega $ as at $\left( \ref%
{Definition of omega alpha}\right) $. Then $y^{D}$ takes value in $G^{\left[
p\right] }\left( \mathcal{U}\right) $, and there exists $C_{p,\gamma }$ such
that, for any $\alpha \in (0,1]$, (with $\alpha _{0}:=\max_{t_{j}\in
D}\omega \left( t_{j},t_{j+1}\right) $)%
\begin{equation}
\left\Vert Y_{T}-y_{T}^{D}\right\Vert \leq C_{p,\gamma }\left\Vert \xi
\right\Vert \exp \left( C_{p,\gamma }\left( \omega ^{\alpha _{0}}\left(
0,T\right) +\alpha ^{-1}\omega ^{\alpha }\left( 0,T\right) \right) \right)
\left( \tsum\nolimits_{j=0}^{n-1}\omega \left( t_{j},t_{j+1}\right) ^{\frac{%
\left[ p\right] +1}{p}}\vee \omega \left( t_{j},t_{j+1}\right) ^{\left[ p%
\right] }\right) \text{.}
\label{estimate of error of concatenated ODE approximation}
\end{equation}
\end{corollary}

The proof of Corollary \ref{Corollary error of ODE approximation} starts
from p\pageref{Proof of Corollary error of ODE approximation}.

By similar arguments, $\left( \ref{estimate of error of concatenated ODE
approximation}\right) $ holds if $y^{D}$ is replaced by concatenated Euler
approximation. If we only consider $\left\Vert \pi _{1}\left( Y_{T}\right)
-\pi _{1}\left( y_{T}^{D}\right) \right\Vert $, then we can drop $\omega
\left( t_{j},t_{j+1}\right) ^{\left[ p\right] }$ in $\left( \ref{estimate of
error of concatenated ODE approximation}\right) $.

\section{Proofs}

We specify the dependence of coefficients (e.g. $C_{p,\gamma }$), but their
exact values may change from line to line.

\noindent For $\gamma >0$, let $\lfloor \gamma \rfloor $ denote the largest
integer which is strictly less than $\gamma $, and denote $\left\{ \gamma
\right\} :=\gamma -\lfloor \gamma \rfloor $.

\subsection{Preparation}

For Banach space $\mathcal{U}$, denote $Id_{\mathcal{U}}$ as the identity
function on $\mathcal{U}$, i.e. $Id_{\mathcal{U}}\left( u\right) =u$, $%
\forall u\in \mathcal{U}$.

\noindent We define ordered shuffle as in \cite{Lyonsnotes} (p73-74).

\begin{definition}[ordered shuffle]
For integer $k\geq 1$, denote by $S_{k}$ the symmetric group of order $k$.
For $j_{1}+\cdots +j_{n}=k$, $j_{i}\geq 1$, define ordered shuffle $OS\left(
j_{1},\dots ,j_{n}\right) $ to be the set of $\sigma \in S_{k}$ which satisfy%
\begin{gather*}
\sigma \left( 1\right) <\sigma \left( 2\right) <\cdots <\sigma \left(
j_{1}\right) \text{, \ }\sigma \left( j_{1}+1\right) <\cdots <\sigma \left(
j_{1}+j_{2}\right) \text{,} \\
\sigma \left( j_{1}+\cdots +j_{n-1}+1\right) <\cdots <\sigma \left(
j_{1}+\cdots +j_{n}\right) \text{, }\sigma \left( j_{1}\right) <\sigma
\left( j_{1}+j_{2}\right) <\cdots <\sigma \left( j_{1}+\cdots +j_{n}\right) 
\text{.}
\end{gather*}
\end{definition}

\begin{notation}
For $f\in L\left( \mathcal{V},C^{\gamma }\left( \mathcal{U},\mathcal{U}%
\right) \right) $ and $j_{1}+\cdots +j_{n}=k$, $j_{i}=1,\dots ,\lfloor
\gamma \rfloor $, denote by $f^{\circ j_{n}}\otimes \cdots \otimes f^{\circ
j_{1}}\in L\left( \mathcal{V}^{\otimes k},C\left( C^{\max_{i}j_{i}-1}\left( 
\mathcal{U},\mathcal{U}\right) ,C^{0}\left( \mathcal{U},\mathcal{U}^{\otimes
n}\right) \right) \right) $ the unique continuous linear operator which
satisfies that, $\forall \left\{ v_{j}\right\} _{j=1}^{k}\subset \mathcal{V}$%
, $\forall r\in C^{\max_{i}j_{i}-1}\left( \mathcal{U},\mathcal{U}\right) $, $%
\forall u\in \mathcal{U}$.%
\begin{eqnarray*}
&&\left( f^{\circ j_{n}}\otimes \cdots \otimes f^{\circ j_{1}}\right) \left(
v_{k}\otimes \cdots \otimes v_{1}\right) \left( r\right) \left( u\right) \\
&=&\left( f^{\circ j_{n}}\left( v_{k}\otimes \cdots \otimes
v_{k-j_{n}+1}\right) \left( r\right) \left( u\right) \right) \otimes \cdots
\otimes \left( f^{\circ j_{1}}\left( v_{j_{1}}\otimes \cdots \otimes
v_{1}\right) \left( r\right) \left( u\right) \right) \text{.}
\end{eqnarray*}
\end{notation}

Recall the Banach space $L^{n}\left( \mathcal{U}\right) :=%
\mathbb{R}
\oplus \mathcal{U\oplus }\cdots \oplus \mathcal{U}^{\otimes n}$ in
Definition \ref{Definition of Ln(V)} on p\pageref{Definition of Ln(V)}, $%
F\left( f\right) \left( y\right) :=y\otimes f\left( \pi _{1}\left( y\right)
\right) $ as denoted at $\left( \ref{Notation of F(f)}\right) $ on p\pageref%
{Notation of F(f)} and $F\left( f\right) ^{\circ k}$ defined in Definition %
\ref{Definition of f circ k} on p\pageref{Definition of f circ k}. For $%
\sigma \in S_{k}$, denote by $\sigma :\mathcal{V}^{\otimes k}\rightarrow 
\mathcal{V}^{\otimes k}$ the unique continuous linear operator which
satisfies%
\begin{equation*}
\sigma \left( v_{k}\otimes \cdots \otimes v_{1}\right) =v_{\sigma \left(
k\right) }\otimes \cdots \otimes v_{\sigma \left( 1\right) }\text{, }\forall
\left\{ v_{j}\right\} _{j=1}^{k}\subset \mathcal{V}\text{.}
\end{equation*}

\begin{lemma}
\label{Lemma explicit form of Fcirck}Suppose $f\in L\left( \mathcal{V}%
,C^{\gamma }\left( \mathcal{U},\mathcal{U}\right) \right) $. Then for $%
k=1,\dots ,\lfloor \gamma \rfloor +1$ and any $v\in \mathcal{V}^{\otimes k}$%
, 
\begin{eqnarray}
&&F\left( f\right) ^{\circ k}\left( v\right) \left( Id_{L^{n}\left( \mathcal{%
U}\right) }\right) \left( y\right)  \label{explicit form of Fcirck} \\
&=&y\otimes \left( \sum_{j_{1}+\cdots +j_{n}=k,j_{i}\geq 1}\,\sum_{\sigma
\in OS\left( j_{1},\dots ,j_{n}\right) }\left( f^{\circ j_{n}}\otimes \cdots
\otimes f^{\circ j_{1}}\right) \left( \sigma \left( v\right) \right) \right)
\left( Id_{\mathcal{U}}\right) \left( \pi _{1}\left( y\right) \right) \text{.%
}  \notag
\end{eqnarray}%
In particular, for $g\in G^{\left[ p\right] }\left( \mathcal{V}\right) $, if
let $y:\left[ 0,1\right] \rightarrow L^{\left[ p\right] }\left( \mathcal{U}%
\right) $ denote the solution to the ODE%
\begin{equation*}
dy_{u}=\sum_{k=1}^{\left[ p\right] }F\left( f\right) ^{\circ k}\pi
_{k}\left( \log g\right) \left( Id_{L^{\left[ p\right] }\left( \mathcal{U}%
\right) }\right) \left( y_{u}\right) du\text{, }u\in \left[ 0,1\right] \text{%
, }y_{0}=1\text{,}
\end{equation*}%
then, with $y^{k}:=\pi _{k}\left( y\right) $, $k=0,1,\dots ,\left[ p\right] $%
, we have%
\begin{eqnarray}
y_{t}^{0} &\equiv &1\text{,}  \label{explicit form of levels of ODE} \\
y_{t}^{k} &=&\sum_{j=1}^{k}\int_{0}^{t}y_{u}^{k-j}\otimes \left(
\sum_{l=j}^{ \left[ p\right] }\sum_{i_{1}+\cdots +i_{j}=l,i_{s}\geq
1}\sum_{\sigma \in OS\left( i_{1},\dots ,i_{j}\right) }\left( \left(
f^{\circ i_{j}}\otimes \cdots \otimes f^{\circ i_{1}}\right) \sigma \left(
\pi _{l}\left( \log g\right) \right) \right) \right) \left( Id_{\mathcal{U}%
}\right) \left( y_{u}^{1}\right) du\text{.}  \notag
\end{eqnarray}
\end{lemma}

\begin{proof}
$\left( \ref{explicit form of Fcirck}\right) $ can be proved by using
mathematical induction when $v=v_{k}\otimes \cdots \otimes v_{1}$ for $%
\left\{ v_{j}\right\} _{j=1}^{k}\subset \mathcal{V}$. Then by using
linearity and continuity (in $\mathcal{V}^{\otimes k}$), $\left( \ref%
{explicit form of Fcirck}\right) $ holds for any $v\in \mathcal{V}^{\otimes
k}$. $\left( \ref{explicit form of levels of ODE}\right) $ follows from $%
\left( \ref{explicit form of Fcirck}\right) $.
\end{proof}

\begin{lemma}
\label{Lemma explicit Euler expansion}For $\gamma >p\geq 1$, suppose $f\in
L\left( \mathcal{V},C^{\gamma }\left( \mathcal{U},\mathcal{U}\right) \right) 
$, $g,h\in G^{\left[ p\right] }\left( \mathcal{V}\right) $ and $\xi \in L^{%
\left[ p\right] }\left( \mathcal{V}\right) $. Let $y:\left[ 0,2\right]
\rightarrow L^{\left[ p\right] }\left( \mathcal{U}\right) $ be the solution
to the ordinary differential equation:%
\begin{equation}
dy_{u}=\left\{ 
\begin{array}{cc}
\sum_{k=1}^{\left[ p\right] }F\left( f\right) ^{\circ k}\pi _{k}\left( \log
g\right) \left( Id_{L^{n}\left( \mathcal{U}\right) }\right) \left(
y_{u}\right) du, & u\in \left[ 0,1\right] \\ 
\sum_{k=1}^{\left[ p\right] }F\left( f\right) ^{\circ k}\pi _{k}\left( \log
h\right) \left( Id_{L^{n}\left( \mathcal{U}\right) }\right) \left(
y_{u}\right) du, & u\in \left[ 1,2\right]%
\end{array}%
\right. \text{, }y_{0}=\xi \text{.}
\label{ODE in lemma of explicit euler expansion}
\end{equation}%
For $j_{1}+\cdots +j_{n}\leq \left[ p\right] $, $j_{i}\geq 1$, denote two
mappings in $C^{1+\left\{ \gamma \right\} ,loc}\left( L^{\left[ p\right]
}\left( \mathcal{U}\right) ,L^{\left[ p\right] }\left( \mathcal{U}\right)
\right) $ by 
\begin{gather*}
F\left( \log g\right) ^{\left( j_{1},\cdots ,j_{n}\right) }:=F\left(
f\right) ^{\circ \left( j_{1}+\cdots +j_{n}\right) }\left( \pi
_{j_{1}}\left( \log g\right) \otimes \cdots \otimes \pi _{j_{n}}\left( \log
g\right) \right) \left( Id_{L^{n}\left( \mathcal{U}\right) }\right) \text{,}
\\
F\left( \left( \log g\right) ^{\left( j_{1},\cdots ,j_{n-1}\right)
}h^{\left( j_{n}\right) }\right) :=F\left( f\right) ^{\circ \left(
j_{1}+\cdots +j_{n}\right) }\left( \pi _{j_{1}}\left( \log g\right) \otimes
\cdots \otimes \pi _{j_{n-1}}\left( \log g\right) \otimes \pi _{j_{n}}\left(
h\right) \right) \left( Id_{L^{n}\left( \mathcal{U}\right) }\right) \text{.}
\end{gather*}%
Then we have 
\begin{eqnarray*}
&&y_{2}-\xi -\sum_{k=1}^{\left[ p\right] }F\left( f\right) ^{\circ k}\pi
_{k}\left( g\otimes h\right) \left( Id_{L^{n}\left( \mathcal{U}\right)
}\right) \left( \xi \right) \\
&=&\sum_{\substack{ j_{1}+\cdots +j_{n}=\left[ p\right]  \\ j_{i}\geq 1}}%
\,\,\idotsint\limits_{1<u_{1}<\cdots <u_{n}<2}\left( F\left( \log h\right)
^{\left( j_{1},\cdots ,j_{n}\right) }\left( y_{u_{1}}\right) -F\left( \log
h\right) ^{\left( j_{1},\cdots ,j_{n}\right) }\left( y_{1}\right) \right)
du_{1}\cdots du_{n} \\
&&+\sum_{\substack{ j_{1}+\cdots +j_{n}=\left[ p\right]  \\ j_{i}\geq 1}}%
\,\,\idotsint\limits_{0<u_{1}<\cdots <u_{n}<1}\left( F\left( \log g\right)
^{\left( j_{1},\cdots ,j_{n}\right) }\left( y_{u_{1}}\right) -F\left( \log
g\right) ^{\left( j_{1},\cdots ,j_{n}\right) }\left( \xi \right) \right)
du_{1}\cdots du_{n} \\
&&+\sum_{\substack{ j_{1}+\cdots +j_{n}=\left[ p\right]  \\ j_{i}\geq
1,n\geq 1}}\,\,\idotsint\limits_{0<u_{1}<\cdots <u_{n}<1}\left( F\left(
\left( \log g\right) ^{\left( j_{1},\cdots ,j_{n-1}\right) }h^{\left(
j_{n}\right) }\right) \left( y_{u_{1}}\right) -F\left( \left( \log g\right)
^{\left( j_{1},\cdots ,j_{n-1}\right) }h^{\left( j_{n}\right) }\right)
\left( \xi \right) \right) du_{1}\cdots du_{n}
\end{eqnarray*}%
\begin{eqnarray*}
&&+\sum_{\substack{ j_{1}+\cdots +j_{n}\leq \left[ p\right] -1  \\ %
j_{1}+\cdots +j_{n+1}\geq \left[ p\right] +1  \\ j_{i}=1,\dots ,\left[ p%
\right] ,n\geq 1}}\,\,\idotsint\limits_{1<u_{1}<\cdots <u_{n}<2}D\left(
F\left( \log h\right) ^{\left( j_{1},\cdots ,j_{n}\right) }\right) \left(
y_{u_{1}}\right) F\left( \log h\right) ^{\left( j_{n+1}\right) }\left(
y_{u_{1}}\right) du_{1}\cdots du_{n+1} \\
&&+\sum_{\substack{ j_{1}+\cdots +j_{n}\leq \left[ p\right] -1  \\ %
j_{1}+\cdots +j_{n+1}\geq \left[ p\right] +1  \\ j_{i}=1,\dots ,\left[ p%
\right] ,n\geq 1}}\,\,\idotsint\limits_{0<u_{1}<\cdots <u_{n}<1}D\left(
F\left( \log g\right) ^{\left( j_{1},\cdots ,j_{n}\right) }\right) \left(
y_{u_{1}}\right) F\left( \log g\right) ^{\left( j_{n+1}\right) }\left(
y_{u_{1}}\right) du_{1}\cdots du_{n+1} \\
&&+\sum_{\substack{ j_{1}+\cdots +j_{n}\leq \left[ p\right] -1  \\ %
j_{1}+\cdots +j_{n+1}\geq \left[ p\right] +1  \\ j_{i}=1,\dots ,\left[ p%
\right] ,n\geq 1}}\,\,\idotsint\limits_{0<u_{1}<\cdots <u_{n}<1}D\left(
F\left( \left( \log g\right) ^{\left( j_{1},\cdots ,j_{n-1}\right)
}h^{\left( j_{n}\right) }\right) \right) \left( y_{u_{1}}\right) F\left(
\log g\right) ^{\left( j_{n+1}\right) }\left( y_{u_{1}}\right) du_{1}\cdots
du_{n+1}
\end{eqnarray*}
\end{lemma}

\begin{proof}
Based on Lemma $21$ in \cite{Youness} (whose proof applies to locally
Lipschitz vector fields), $\left\{ F\left( f\right) ^{\circ k}\pi _{k}\left(
\log h\right) \right\} _{k=1}^{\left[ p\right] }$ are first order
differential operators, and for integer $n\leq \left[ p\right] -k$ and any $%
v\in \mathcal{V}^{\otimes n}$,%
\begin{equation*}
D\left( F\left( f\right) ^{\circ n}\left( v\right) \left( Id_{L^{n}\left( 
\mathcal{U}\right) }\right) \right) \left( F\left( f\right) ^{\circ k}\pi
_{k}\left( \log h\right) \left( Id_{L^{n}\left( \mathcal{U}\right) }\right)
\right) =F\left( f\right) ^{\circ \left( k+n\right) }\left( \pi _{k}\left(
\log h\right) \otimes v\right) \left( Id_{L^{n}\left( \mathcal{U}\right)
}\right) \text{.}
\end{equation*}%
Then by subtraction and using the fact that $y$ is the solution to the ODE $%
\left( \ref{ODE in lemma of explicit euler expansion}\right) $ (or see Lemma 
$22$ in \cite{Youness}), we get%
\begin{eqnarray*}
&&y_{2}-y_{1}-\sum_{k=1}^{\left[ p\right] }F\left( f\right) ^{\circ k}\pi
_{k}\left( h\right) \left( Id_{L^{\left[ p\right] }\left( \mathcal{U}\right)
}\right) \left( y_{1}\right) \\
&=&\sum_{\substack{ j_{1}+\cdots +j_{n}=\left[ p\right]  \\ j_{i}\geq 1}}%
\,\,\idotsint\limits_{1<u_{1}<\cdots <u_{n}<2}\left( F\left( \log h\right)
^{\left( j_{1},\cdots ,j_{n}\right) }\left( y_{u_{1}}\right) -F\left( \log
h\right) ^{\left( j_{1},\cdots ,j_{n}\right) }\left( y_{1}\right) \right)
du_{1}\cdots du_{n} \\
&&+\sum_{\substack{ j_{1}+\cdots +j_{n}\leq \left[ p\right] -1  \\ %
j_{1}+\cdots +j_{n+1}\geq \left[ p\right] +1  \\ j_{i}=1,\dots ,\left[ p%
\right] ,n\geq 1}}\,\idotsint\limits_{1<u_{1}<\cdots <u_{n+1}<2}D\left(
F\left( \log h\right) ^{\left( j_{1},\cdots ,j_{n}\right) }\right) \left(
y_{u_{1}}\right) F\left( \log h\right) ^{\left( j_{n+1}\right) }\left(
y_{u_{1}}\right) du_{1}\cdots du_{n+1}\text{.}
\end{eqnarray*}%
Similar estimate applies to $y_{1}$ and $F\left( f\right) ^{\circ k}\pi
_{k}\left( h\right) \left( Id_{L^{\left[ p\right] }\left( \mathcal{U}\right)
}\right) \left( y_{1}\right) $, $k=1,2,\dots ,\left[ p\right] $.
\end{proof}

\begin{lemma}
\label{Lemma simple property of solution of ODE}For $\gamma >p\geq 1$,
suppose $f\in L\left( \mathcal{V},C^{\gamma }\left( \mathcal{U},\mathcal{U}%
\right) \right) $ and $g\in G^{\left[ p\right] }\left( \mathcal{V}\right) $.
Let $y:\left[ 0,1\right] \rightarrow L^{\left[ p\right] }\left( \mathcal{U}%
\right) $ be the solution to the ordinary differential equation%
\begin{equation}
dy_{u}=\sum_{k=1}^{\left[ p\right] }F\left( f\right) ^{\circ k}\pi
_{k}\left( \log g\right) \left( Id_{L^{\left[ p\right] }\left( \mathcal{U}%
\right) }\right) \left( y_{u}\right) du\text{, }u\in \left[ 0,1\right] \text{%
, }y_{0}=1\text{.}  \label{ODE solution in Lemma}
\end{equation}%
Then $y$ takes value in $G^{\left[ p\right] }\left( \mathcal{U}\right) $,
and when $\left\vert f\right\vert _{\gamma }=1$ and $\left\vert \!\left\vert
\!\left\vert g\right\vert \!\right\vert \!\right\vert \leq 1$, we have%
\begin{equation}
\left\Vert \pi _{k}\left( y\right) \right\Vert _{1-var,\left[ 0,1\right]
}\leq C_{p}\left\vert \!\left\vert \!\left\vert g\right\vert \!\right\vert
\!\right\vert ^{k}\text{, }k=1,2,\dots ,\left[ p\right] \text{.}
\label{bound on y on small interval}
\end{equation}
\end{lemma}

\begin{proof}
Based on $\left( \ref{explicit form of levels of ODE}\right) $ in Lemma \ref%
{Lemma explicit form of Fcirck}, if we denote $y^{k}:=\pi _{k}\left(
y\right) $, then it can be proved inductively that%
\begin{equation}
\sup_{t\in \left[ 0,1\right] }\left\Vert y_{t}^{k}\right\Vert \leq
C_{p}\left\vert \!\left\vert \!\left\vert g\right\vert \!\right\vert
\!\right\vert ^{k}\text{, }k=0,1,\dots ,\left[ p\right] \text{.}
\label{inner bound on y}
\end{equation}%
Indeed, $\left( \ref{inner bound on y}\right) $ holds clearly when $k=0$.
Then by using $\left\Vert \sigma \left( \pi _{l}\left( \log g\right) \right)
\right\Vert =\left\Vert \pi _{l}\left( \log g\right) \right\Vert $ (tensor
norm is symmetric as at $\left( \ref{tensor norm is symmetric}\right) $),
for $k=1,\dots ,\left[ p\right] $, (since $\left\vert \!\left\vert
\!\left\vert g\right\vert \!\right\vert \!\right\vert \leq 1$, $\left\vert
f\right\vert _{\gamma }=1$)%
\begin{equation*}
\sup_{t\in \left[ 0,1\right] }\left\Vert y_{t}^{k}\right\Vert \leq
C_{p}\sum_{j=1}^{k}\sup_{u\in \left[ 0,1\right] }\left\Vert
y_{u}^{k-j}\right\Vert \left( \sum_{l=j}^{\left[ p\right] }\left\vert
\!\left\vert \!\left\vert g\right\vert \!\right\vert \!\right\vert
^{l}\right) \leq C_{p}\left\vert \!\left\vert \!\left\vert g\right\vert
\!\right\vert \!\right\vert ^{k}\text{.}
\end{equation*}

Then we prove that $y$ takes value in $G^{\left[ p\right] }\left( \mathcal{U}%
\right) $. For $i=1,2$, let $\mathcal{V}^{i}$ be Banach spaces, and $%
F^{i}\in L\left( \mathcal{V}^{i},\mathcal{D}^{k_{i}}\left( L^{\left[ p\right]
}\left( \mathcal{U}\right) \right) \right) $. Denote by $\left[ F^{2},F^{1}%
\right] \in L\left( \mathcal{V}^{2}\otimes \mathcal{V}^{1},\mathcal{D}%
^{1}\left( L^{\left[ p\right] }\left( \mathcal{U}\right) \right) \right) $
the unique continuous linear operator which satisfies%
\begin{gather}
\left[ F^{2},F^{1}\right] \left( v^{2}\otimes v^{1}\right) \left( r\right)
:=\left( Dr\right) \left( F^{2}\left( v_{2}\right) \circ F^{1}\left(
v_{1}\right) -F^{1}\left( v_{1}\right) \circ F^{2}\left( v_{2}\right)
\right) \left( Id_{L^{\left[ p\right] }\left( \mathcal{U}\right) }\right) 
\text{, }  \label{inner definition of [F1,F2]} \\
\forall v^{1}\in \mathcal{V}^{1}\text{, }\forall v^{2}\in \mathcal{V}^{2}%
\text{, }\forall r\in C^{1,loc}\left( L^{\left[ p\right] }\left( \mathcal{U}%
\right) ,L^{\left[ p\right] }\left( \mathcal{U}\right) \right) \text{.} 
\notag
\end{gather}%
For integer $k=1,\dots ,\left[ p\right] $, with $F\left( f\right) ^{\circ
1}\in L\left( \mathcal{V},\mathcal{D}^{1}\left( L^{\left[ p\right] }\left( 
\mathcal{U}\right) \right) \right) $ defined at $\left( \ref{definition of
Fcirc1}\right) $ (on p\pageref{definition of Fcirc1}), we define 
\begin{equation*}
\left[ F\left( f\right) \right] ^{\circ 1}:=F\left( f\right) ^{\circ 1}\text{
and }\left[ F\left( f\right) \right] ^{\circ \left( k+1\right) }:=\left[
F\left( f\right) ^{\circ 1},\left[ F\left( f\right) \right] ^{\circ k}\right]
\text{.}
\end{equation*}%
Then based on Lemma 21 in \cite{Youness} (whose proof applies to locally
Lipschitz vector fields), for $k=1,\dots ,\left[ p\right] $, $\left\{
v_{i}\right\} _{i=1}^{k}\subset \mathcal{V}$ and $r\in C^{1,loc}\left( L^{%
\left[ p\right] }\left( \mathcal{U}\right) ,L^{\left[ p\right] }\left( 
\mathcal{U}\right) \right) $,%
\begin{eqnarray}
F\left( f\right) ^{\circ k}\left[ v_{k},\dots ,\left[ v_{2},v_{1}\right]
\dots \right] \left( r\right) &=&\left[ F\left( f\right) \right] ^{\circ
k}\left( v_{k}\otimes \cdots \otimes v_{1}\right) \left( r\right)
\label{inner relationship between F(f) and [F(f)]} \\
&=&\left( Dr\right) \left( \left[ F\left( f\right) \right] ^{\circ k}\left(
v_{k}\otimes \cdots \otimes v_{1}\right) \left( Id_{L^{\left[ p\right]
}\left( \mathcal{U}\right) }\right) \right) \text{.}  \notag
\end{eqnarray}%
We want to prove that, for $k=2,\dots ,\left[ p\right] $, there exist $%
\{G_{i}^{s,j,k}\}\subset C^{1}\left( \mathcal{U},\mathcal{U}\right) $, such
that, for any $y\in L^{\left[ p\right] }\left( \mathcal{U}\right) $, 
\begin{eqnarray}
&&\left[ F\left( f\right) \right] ^{\circ k}\left( v_{k}\otimes \cdots
\otimes v_{1}\right) \left( Id_{L^{\left[ p\right] }\left( \mathcal{U}%
\right) }\right) \left( y\right)
\label{inner inductive hypothesis of being Lie} \\
&=&y\otimes \left( \sum_{j=2}^{k}\sum_{s=1}^{l_{j}}\left[ G_{j}^{s,j,k},%
\dots ,\left[ G_{2}^{s,j,k},G_{1}^{s,j,k}\right] \dots \right] \left( \pi
_{1}\left( y\right) \right) +f^{\circ k}\left( \left[ v_{k},\dots ,\left[
v_{2},v_{1}\right] \dots \right] \right) \left( Id_{\mathcal{U}}\right)
\left( \pi _{1}\left( y\right) \right) \right) \text{.}  \notag
\end{eqnarray}%
When $k=2$, by using the definition of $\left[ F\left( f\right) \right]
^{\circ 2}$, we have%
\begin{eqnarray*}
&&\left[ F\left( f\right) \right] ^{\circ 2}\left( v_{2}\otimes v_{1}\right)
\left( Id_{L^{\left[ p\right] }\left( \mathcal{U}\right) }\right) \left(
y\right) \\
&=&\left( \left( F\left( f\right) ^{\circ 1}\left( v_{2}\right) \right)
\circ \left( F\left( f\right) ^{\circ 1}\left( v_{1}\right) \right) -\left(
F\left( f\right) ^{\circ 1}\left( v_{1}\right) \right) \circ \left( F\left(
f\right) ^{\circ 1}\left( v_{2}\right) \right) \right) \left( Id_{L^{\left[ p%
\right] }\left( \mathcal{U}\right) }\right) \left( y\right) \\
&=&y\otimes \left( \left( f\left( v_{2}\right) \left( \pi _{1}\left(
y\right) \right) \right) \otimes \left( f\left( v_{1}\right) \left( \pi
_{1}\left( y\right) \right) \right) -\left( f\left( v_{1}\right) \left( \pi
_{1}\left( y\right) \right) \right) \otimes \left( f\left( v_{2}\right)
\left( \pi _{1}\left( y\right) \right) \right) \right) \\
&&+y\otimes \left( f^{\circ 2}\left( v_{2}\otimes v_{1}\right) \left( Id_{%
\mathcal{U}}\right) -f^{\circ 2}\left( v_{1}\otimes v_{2}\right) \left( Id_{%
\mathcal{U}}\right) \right) \\
&=&y\otimes \left( \left[ \left( f^{\circ 1}\left( v_{2}\right) \right)
\left( Id_{\mathcal{U}}\right) ,\left( f^{\circ 1}\left( v_{1}\right)
\right) \left( Id_{\mathcal{U}}\right) \right] \left( \pi _{1}\left(
y\right) \right) +f^{\circ 2}\left( \left[ v_{2},v_{1}\right] \right) \left(
Id_{\mathcal{U}}\right) \left( \pi _{1}\left( y\right) \right) \right) \text{%
.}
\end{eqnarray*}%
Then $\left( \ref{inner inductive hypothesis of being Lie}\right) $ holds
when $k=2$ with $l_{2}=1$ and $G_{i}^{1,2,2}=\left( f^{\circ 1}\left(
v_{i}\right) \right) \left( Id_{\mathcal{U}}\right) $, $i=1,2$. Suppose $%
\left( \ref{inner inductive hypothesis of being Lie}\right) $ holds for $k$.
Then for $k+1$, by using $\left( \ref{inner definition of [F1,F2]}\right) $,
the second equality in $\left( \ref{inner relationship between F(f) and
[F(f)]}\right) $ and inductive hypothesis $\left( \ref{inner inductive
hypothesis of being Lie}\right) $, we have%
\begin{eqnarray*}
&&\left[ F\left( f\right) \right] ^{\circ \left( k+1\right) }\left(
v_{k+1}\otimes v_{k}\otimes \cdots \otimes v_{1}\right) \left( Id_{L^{\left[
p\right] }\left( \mathcal{U}\right) }\right) \left( y\right) \\
&=&\left( \left( F\left( f\right) ^{\circ 1}v_{k+1}\right) \circ \left[
F\left( f\right) \right] ^{\circ k}\left( v_{k}\otimes \cdots \otimes
v_{1}\right) -\left[ F\left( f\right) \right] ^{\circ k}\left( v_{k}\otimes
\cdots \otimes v_{1}\right) \circ \left( F\left( f\right) ^{\circ
1}v_{k+1}\right) \right) \left( Id_{L^{\left[ p\right] }\left( \mathcal{U}%
\right) }\right) \left( y\right) \\
&=&D\left( \left[ F\left( f\right) \right] ^{\circ k}\left( v_{k}\otimes
\cdots \otimes v_{1}\right) \left( Id_{L^{\left[ p\right] }\left( \mathcal{U}%
\right) }\right) \right) \left( F\left( f\right) ^{\circ 1}v_{k+1}\right)
\left( Id_{L^{\left[ p\right] }\left( \mathcal{U}\right) }\right) \left(
y\right) \\
&&-D\left( \left( F\left( f\right) ^{\circ 1}v_{k+1}\right) \left( Id_{L^{ 
\left[ p\right] }\left( \mathcal{U}\right) }\right) \right) \left[ F\left(
f\right) \right] ^{\circ k}\left( v_{k}\otimes \cdots \otimes v_{1}\right)
\left( Id_{L^{\left[ p\right] }\left( \mathcal{U}\right) }\right) \left(
y\right) \\
&=&y\otimes \left( \sum_{j=2}^{k}\sum_{s=1}^{l_{j}}\left[ \left( f^{\circ
1}v_{k+1}\right) \left( Id_{\mathcal{U}}\right) ,\left[ G_{j}^{s,j,k},\dots ,%
\left[ G_{2}^{s,j,k},G_{1}^{s,j,k}\right] \dots \right] \right] \left( \pi
_{1}\left( y\right) \right) \right) \\
&&+y\otimes \left( \sum_{j=2}^{k}\sum_{s=1}^{l_{j}}\sum_{i=1}^{j}\left[
G_{j}^{s,j,k},\dots \left[ \left( DG_{i}^{s,j,k}\right) \left( \left(
f^{\circ 1}v_{k+1}\right) \left( Id_{\mathcal{U}}\right) \right) ,\dots %
\left[ G_{2}^{s,j,k},G_{1}^{s,j,k}\right] \right] \dots \right] \left( \pi
_{1}\left( y\right) \right) \right) \\
&&+y\otimes \left[ \left( f^{\circ 1}v_{k+1}\right) \left( Id_{\mathcal{U}%
}\right) ,f^{\circ k}\left( \left[ v_{k},\dots ,\left[ v_{2},v_{1}\right] %
\right] \right) \left( Id_{\mathcal{U}}\right) \right] \left( \pi _{1}\left(
y\right) \right) \\
&&+y\otimes f^{\circ \left( k+1\right) }\left( \left[ v_{k+1},\left[
v_{k}\dots ,\left[ v_{2},v_{1}\right] \right] \right] \right) \left( Id_{%
\mathcal{U}}\right) \left( \pi _{1}\left( y\right) \right) \text{.}
\end{eqnarray*}%
As a result, by choosing $\{G_{i}^{s,j,k+1}\}$ properly, $\left( \ref{inner
inductive hypothesis of being Lie}\right) $ holds for $k+1$.

Then based on $\left( \ref{inner relationship between F(f) and [F(f)]}%
\right) $ and $\left( \ref{inner inductive hypothesis of being Lie}\right) $%
, we have that, there exists a function $L$ on $\mathcal{U}$ taking values
in Lie polynomials $\left[ \mathcal{U}\right] ^{1}\oplus \cdots \oplus \left[
\mathcal{U}\right] ^{\left[ p\right] }$ (with $\left[ \mathcal{U}\right]
^{n} $ in Definition \ref{Definition of Vtensorn and Vbracketn} on p\pageref%
{Definition of Vtensorn and Vbracketn}), such that the ODE $\left( \ref{ODE
solution in Lemma}\right) $ can be re-written as 
\begin{equation*}
dy_{u}=y_{u}\otimes \left( L\left( \pi _{1}\left( y_{u}\right) \right)
\right) du,\text{ }u\in \left[ 0,1\right] \text{.}
\end{equation*}%
As a result, if we denote 
\begin{equation*}
\gamma _{t}:=\int_{0}^{t}L\left( \pi _{1}\left( y_{u}\right) \right) du\text{%
,}
\end{equation*}%
Then $\gamma $ is differentiable, taking value in Lie polynomials of degree $%
\left[ p\right] $, and%
\begin{equation*}
dy_{u}=y_{u}\otimes d\gamma _{u}\text{, }u\in \left[ 0,1\right] \text{.}
\end{equation*}%
Then it can be checked that $y$ takes values in $G^{\left[ p\right] }\left( 
\mathcal{U}\right) $.
\end{proof}

\begin{lemma}
\label{Lemma Euler expansion of solution of ODE and difference between one
step and two steps}For $\gamma >p\geq 1$, suppose $f\in L\left( \mathcal{V}%
,C^{\gamma }\left( \mathcal{U},\mathcal{U}\right) \right) $ with $\left\vert
f\right\vert _{\gamma }=1$, and $g,h\in G^{\left[ p\right] }\left( \mathcal{U%
}\right) $ satisfying $\left\vert \!\left\vert \!\left\vert g\right\vert
\!\right\vert \!\right\vert \vee \left\vert \!\left\vert \!\left\vert
h\right\vert \!\right\vert \!\right\vert \vee \left\vert \!\left\vert
\!\left\vert g\otimes h\right\vert \!\right\vert \!\right\vert \leq 1$. Let $%
y^{g}:\left[ 0,1\right] \rightarrow G^{\left[ p\right] }\left( \mathcal{U}%
\right) $ and $y^{g,h}:\left[ 0,2\right] \rightarrow G^{\left[ p\right]
}\left( \mathcal{U}\right) $ be the solution to the ODE:%
\begin{equation*}
dy^{g}=\tsum\nolimits_{k=1}^{\left[ p\right] }F\left( f\right) ^{\circ k}\pi
_{k}\left( \log g\right) \left( Id_{L^{\left[ p\right] }\left( \mathcal{U}%
\right) }\right) \left( y^{g}\right) du\text{, }u\in \left[ 0,1\right] \text{%
, \ }y_{0}^{g}=1\text{.}
\end{equation*}%
\begin{equation*}
dy_{u}^{g,h}=\left\{ 
\begin{array}{cc}
\sum_{k=1}^{\left[ p\right] }F\left( f\right) ^{\circ k}\pi _{k}\left( \log
g\right) \left( Id_{L^{\left[ p\right] }\left( \mathcal{U}\right) }\right)
\left( y_{u}^{g,h}\right) du, & u\in \left[ 0,1\right] \\ 
\sum_{k=1}^{\left[ p\right] }F\left( f\right) ^{\circ k}\pi _{k}\left( \log
h\right) \left( Id_{L^{\left[ p\right] }\left( \mathcal{U}\right) }\right)
\left( y_{u}^{g,h}\right) du, & u\in \left[ 1,2\right]%
\end{array}%
\right. \text{, }y_{0}^{g,h}=1\text{.}
\end{equation*}%
Then%
\begin{equation}
\left\Vert y_{1}^{g}-1-\tsum\nolimits_{k=1}^{\left[ p\right] }F\left(
f\right) ^{\circ k}\pi _{k}\left( g\right) \left( Id_{L^{\left[ p\right]
}\left( \mathcal{U}\right) }\right) \left( 1\right) \right\Vert \leq
C_{p}\left\vert \!\left\vert \!\left\vert g\right\vert \!\right\vert
\!\right\vert ^{\left[ p\right] +1}\text{,}  \label{equ 1 in Lemma of ODE}
\end{equation}%
and%
\begin{equation}
\left\Vert y_{2}^{g,h}-y_{1}^{g\otimes h}\right\Vert \leq C_{p}\left(
\left\vert \!\left\vert \!\left\vert g\right\vert \!\right\vert
\!\right\vert \vee \left\vert \!\left\vert \!\left\vert h\right\vert
\!\right\vert \!\right\vert \vee \left\vert \!\left\vert \!\left\vert
g\otimes h\right\vert \!\right\vert \!\right\vert \right) ^{\left[ p\right]
+1}\text{.}  \label{equ 2 in Lemma of ODE}
\end{equation}
\end{lemma}

\begin{proof}
Based on explicit Euler expansion of $y_{1}^{g}$ in Lemma \ref{Lemma
explicit Euler expansion} and $\sup_{u\in \left[ 0,1\right] }\!\!\left\Vert
\pi _{k}\left( y_{u}^{g}\right) \right\Vert \leq C_{p}\left\vert
\!\left\vert \!\left\vert g\right\vert \!\right\vert \!\right\vert ^{k}$ at $%
\left( \ref{bound on y on small interval}\right) $ in Lemma \ref{Lemma
simple property of solution of ODE}, $\left( \ref{equ 1 in Lemma of ODE}%
\right) $ holds; again based on Lemma \ref{Lemma explicit Euler expansion}
and using $\left( \ref{equ 1 in Lemma of ODE}\right) $, $\left( \ref{equ 2
in Lemma of ODE}\right) $ holds.
\end{proof}

\begin{lemma}
\label{Lemma log into exp}For $i=1,2$, suppose $g_{i}\in G^{\left[ p\right]
}\left( \mathcal{V}\right) $ satisfying $\delta :=\left\vert \!\left\vert
\!\left\vert g_{1}\right\vert \!\right\vert \!\right\vert \vee \left\vert
\!\left\vert \!\left\vert g_{2}\right\vert \!\right\vert \!\right\vert \leq
1 $. Then for $\beta \geq 0$,%
\begin{equation*}
\tsum\nolimits_{n=1}^{\left[ p\right] }\delta ^{\left( \beta -n\right) \vee
0}\left\Vert \pi _{n}\left( \log g_{1}\right) -\pi _{n}\left( \log
g_{2}\right) \right\Vert \text{ and }\tsum\nolimits_{n=1}^{\left[ p\right]
}\delta ^{\left( \beta -n\right) \vee 0}\left\Vert \pi _{n}\left(
g_{1}\right) -\pi _{n}\left( g_{2}\right) \right\Vert \text{,}
\end{equation*}%
are equivalent up to a constant $C_{p}$.
\end{lemma}

\begin{proof}
For $n=1,2,\dots ,\left[ p\right] $,%
\begin{eqnarray*}
&&\delta ^{\left( \beta -n\right) \vee 0}\left\Vert \pi _{n}\left( \log
g_{1}\right) -\pi _{n}\left( \log g_{2}\right) \right\Vert \\
&\leq &C_{p}\delta ^{\left( \beta -n\right) \vee
0}\tsum\nolimits_{j_{1}+\cdots +j_{l}=n,j_{i}\geq 1}\left\Vert \pi
_{j_{1}}\left( g_{1}\right) \otimes \cdots \otimes \pi _{j_{l}}\left(
g_{1}\right) -\pi _{j_{1}}\left( g_{2}\right) \otimes \cdots \otimes \pi
_{j_{l}}\left( g_{2}\right) \right\Vert \text{.}
\end{eqnarray*}%
Then by using that 
\begin{equation*}
a_{1}\otimes \cdots \otimes a_{l}-b_{1}\otimes \cdots \otimes
b_{l}=\tsum\nolimits_{i=0}^{l-1}a_{1}\otimes \cdots a_{i}\otimes \left(
a_{i+1}-b_{i+1}\right) \otimes b_{i+2}\otimes \cdots \otimes b_{l}\text{,}
\end{equation*}%
we have ($\delta :=\left\vert \!\left\vert \!\left\vert g_{1}\right\vert
\!\right\vert \!\right\vert \vee \left\vert \!\left\vert \!\left\vert
g_{2}\right\vert \!\right\vert \!\right\vert $)%
\begin{equation*}
\tsum\nolimits_{j_{1}+\cdots +j_{l}=n,j_{i}\geq 1}\left\Vert \pi
_{j_{1}}\left( g_{1}\right) \otimes \cdots \otimes \pi _{j_{l}}\left(
g_{1}\right) -\pi _{j_{1}}\left( g_{2}\right) \otimes \cdots \otimes \pi
_{j_{l}}\left( g_{2}\right) \right\Vert \leq
C_{p}\tsum\nolimits_{j=1}^{n}\delta ^{n-j}\left\Vert \pi _{j}\left(
g_{1}\right) -\pi _{j}\left( g_{2}\right) \right\Vert \text{,}
\end{equation*}%
and%
\begin{eqnarray*}
\delta ^{\left( \beta -n\right) \vee 0}\left\Vert \pi _{n}\left( \log
g_{1}\right) -\pi _{n}\left( \log g_{2}\right) \right\Vert &\leq
&C_{p}\tsum\nolimits_{j=1}^{n}\delta ^{\left( \beta -n\right) \vee
0+n-j}\left\Vert \pi _{j}\left( g_{1}\right) -\pi _{j}\left( g_{2}\right)
\right\Vert \\
&\leq &C_{p}\tsum\nolimits_{j=1}^{n}\delta ^{\left( \beta -j\right) \vee
0}\left\Vert \pi _{j}\left( g_{1}\right) -\pi _{j}\left( g_{2}\right)
\right\Vert \text{.}
\end{eqnarray*}%
The proof for the other direction is similar.
\end{proof}

\begin{lemma}
\label{Lemma two ODE with different vector fields}For $i=1,2$ and $\gamma
>p\geq 1$, suppose $f^{i}\in L\left( \mathcal{V},C^{\gamma }\left( \mathcal{U%
},\mathcal{U}\right) \right) $, $X^{i}\in C^{p-var}\left( \left[ 0,T\right]
,G^{\left[ p\right] }\left( \mathcal{V}\right) \right) $ and $g^{i}\in G^{%
\left[ p\right] }\left( \mathcal{V}\right) $. Let $y^{i}:\left[ 0,1\right]
\rightarrow G^{\left[ p\right] }\left( \mathcal{U}\right) $ be the solution
to the ODE%
\begin{equation*}
dy_{u}^{i}=\tsum\nolimits_{k=1}^{\left[ p\right] }F\left( f^{i}\right)
^{\circ k}\pi _{k}\left( \log g^{i}\right) \left( Id_{L^{\left[ p\right]
}\left( \mathcal{U}\right) }\right) \left( y_{u}^{i}\right) du,\text{ }u\in %
\left[ 0,1\right] \text{, }y_{0}^{i}=1\text{.}
\end{equation*}%
We further assume that $\left\vert f^{i}\right\vert _{\gamma }=1$, $i=1,2$,
and $\delta :=\left\vert \!\left\vert \!\left\vert g^{1}\right\vert
\!\right\vert \!\right\vert \vee \left\vert \!\left\vert \!\left\vert
g^{2}\right\vert \!\right\vert \!\right\vert \leq 1$. Then for $k=1,\dots ,%
\left[ p\right] $, 
\begin{equation}
\left\Vert \pi _{k}\left( y_{1}^{1}\right) -\pi _{k}\left( y_{1}^{2}\right)
\right\Vert \leq C_{p}\left( \delta ^{k}\left\vert f^{1}-f^{2}\right\vert _{ 
\left[ p\right] -1}+\delta ^{\left( k-n\right) \vee 0}\tsum\nolimits_{n=1}^{%
\left[ p\right] }\left\Vert \pi _{n}\left( g^{1}\right) -\pi _{n}\left(
g^{2}\right) \right\Vert \right) \text{.}
\label{estimate of two ODEs with different vector field homogeneous initial condition}
\end{equation}
\end{lemma}

\begin{proof}
For $i=1,2$ and $k=0,1,\dots ,\left[ p\right] $, denote $y^{i,k}:=\pi
_{k}\left( y^{i}\right) $. Based on $\left( \ref{explicit form of levels of
ODE}\right) $ in Lemma \ref{Lemma explicit form of Fcirck} on p\pageref%
{explicit form of levels of ODE}, we have 
\begin{equation*}
y_{t}^{1,1}-y_{t}^{2,1}=\sum_{n=1}^{\left[ p\right] }\int_{0}^{t}\left(
\left( f^{1}\right) ^{\circ n}\pi _{n}\left( \log g^{1}\right) \left( Id_{%
\mathcal{U}}\right) \left( y_{u}^{1,1}\right) -\left( f^{2}\right) ^{\circ
n}\pi _{n}\left( \log g^{2}\right) \left( Id_{\mathcal{U}}\right) \left(
y_{u}^{2,1}\right) \right) du\text{.}
\end{equation*}%
Then ($\delta \leq 1$)%
\begin{equation*}
\left\Vert y_{t}^{1,1}-y_{t}^{2,1}\right\Vert \leq C_{p}\left( \delta
\left\vert f^{1}-f^{2}\right\vert _{\left[ p\right] -1}+\sum_{n=1}^{\left[ p%
\right] }\left\Vert \pi _{n}\left( \log g^{1}\right) -\pi _{n}\left( \log
g^{2}\right) \right\Vert +\delta \int_{0}^{t}\left\Vert
y_{u}^{1,1}-y_{u}^{2,1}\right\Vert du\right) \text{.}
\end{equation*}%
By using Gronwall's inequality, we have 
\begin{equation*}
\sup_{t\in \left[ 0,1\right] }\left\Vert y_{t}^{1,1}-y_{t}^{2,1}\right\Vert
\leq C_{p}\left( \delta \left\vert f^{1}-f^{2}\right\vert _{\left[ p\right]
-1}+\sum_{n=1}^{\left[ p\right] }\left\Vert \pi _{n}\left( \log g^{1}\right)
-\pi _{n}\left( \log g^{2}\right) \right\Vert \right) \text{.}
\end{equation*}

Suppose for $j=1,\dots ,k-1$,%
\begin{equation}
\sup_{t\in \left[ 0,1\right] }\left\Vert y_{t}^{1,j}-y_{t}^{2,j}\right\Vert
\leq C_{p}\left( \delta ^{j}\left\vert f^{1}-f^{2}\right\vert _{\left[ p%
\right] -1}+\delta ^{\left( j-n\right) \vee 0}\sum_{n=1}^{\left[ p\right]
}\left\Vert \pi _{n}\left( \log g^{1}\right) -\pi _{n}\left( \log
g^{2}\right) \right\Vert \right) \text{.}
\label{inner inductive hypothesis in Lemma of two ODEs with different vector field}
\end{equation}%
Then for $k=2,\dots ,\left[ p\right] $, based on $\left( \ref{explicit form
of levels of ODE}\right) $ in Lemma \ref{Lemma explicit form of Fcirck} on p%
\pageref{explicit form of levels of ODE}, we have, ($\delta \leq 1$) 
\begin{eqnarray*}
&&\left\Vert y_{t}^{1,k}-y_{t}^{2,k}\right\Vert \\
&\leq &C_{p}\sum_{j=1}^{k}\delta ^{j}\sup_{t\in \left[ 0,1\right]
}\left\Vert y_{t}^{1,k-j}-y_{t}^{2,k-j}\right\Vert \\
&&+C_{p}\sum_{j=1}^{k}\sup_{t\in \left[ 0,1\right] }\left\Vert
y_{t}^{2,k-j}\right\Vert \left( \delta ^{j}\left\vert f^{1}-f^{2}\right\vert
_{\left[ p\right] -1}+\sum_{l=j}^{\left[ p\right] }\left\Vert \pi _{l}\left(
\log g_{1}\right) -\pi _{l}\left( \log g_{2}\right) \right\Vert +\delta
^{j}\left\Vert y_{t}^{1,1}-y_{t}^{2,1}\right\Vert \right) \text{.}
\end{eqnarray*}%
Then by using $\sup_{t\in \left[ 0,1\right] }\left\Vert
y_{t}^{2,k-j}\right\Vert \leq C_{p}\delta ^{k-j}$ as at $\left( \ref{bound
on y on small interval}\right) $ in Lemma \ref{Lemma simple property of
solution of ODE} and using inductive hypothesis $\left( \ref{inner inductive
hypothesis in Lemma of two ODEs with different vector field}\right) $, 
\begin{equation*}
\left\Vert y_{t}^{1,k}-y_{t}^{2,k}\right\Vert \leq C_{p}\left( \delta
^{k}\left\vert f^{1}-f^{2}\right\vert _{\left[ p\right] -1}+\sum_{n=1}^{%
\left[ p\right] }\delta ^{\left( k-n\right) \vee 0}\left\Vert \pi _{n}\left(
\log g_{1}\right) -\pi _{n}\left( \log g_{2}\right) \right\Vert \right) 
\text{.}
\end{equation*}%
Then combined with Lemma \ref{Lemma log into exp}, one can replace $\log
g_{i}$ by $g_{i}$ (up to a constant depending on $p$).
\end{proof}

Lemma \ref{Lemma simple but important} follows from Lemma 3.5 \cite{AMDavie}
or Lemma 10.22 \cite{FrizVictoirbook}.

\begin{lemma}
\label{Lemma simple but important}Suppose $f^{1}$ and $f^{2}$ are $Lip\left(
\beta \right) \,$\ for some $\beta \in (1,2]$. Then%
\begin{eqnarray*}
&&\left\Vert f^{1}\left( u_{1}\right) -f^{1}\left( u_{2}\right) -\left(
f^{2}\left( v_{1}\right) -f^{2}\left( v_{2}\right) \right) \right\Vert \\
&\leq &\left\vert f^{1}\right\vert _{\beta }\left\Vert u_{1}-u_{2}-\left(
v_{1}-v_{2}\right) \right\Vert +\left( \left\Vert u_{1}-u_{2}\right\Vert
+\left\Vert v_{1}-v_{2}\right\Vert \right) ^{\beta -1}\left( \left\vert
f^{1}\right\vert _{\beta }\left\Vert u_{2}-v_{2}\right\Vert +\left\vert
f^{1}-f^{2}\right\vert _{\beta -1}\right)
\end{eqnarray*}
\end{lemma}

\begin{lemma}
\label{Lemma B}Suppose $\left[ p\right] +1\geq \gamma >p\geq 1$. For $i=1,2$%
, let $X^{i}\in C^{p-var}\left( \left[ 0,T\right] ,G^{\left[ p\right]
}\left( \mathcal{V}\right) \right) $, $f^{i}\in L\left( \mathcal{V}%
,C^{\gamma }\left( \mathcal{U},\mathcal{U}\right) \right) $ and $\xi ^{i}\in
G^{\left[ p\right] }\left( \mathcal{U}\right) $. Assume $\left\vert
f^{i}\right\vert _{\gamma }\leq 1$ and $\left\Vert \xi ^{i}\right\Vert \leq
1 $. Define control $\omega :\left\{ \left( s,t\right) |0\leq s\leq t\leq
T\right\} \rightarrow \overline{%
\mathbb{R}
^{+}}$ by%
\begin{equation*}
\omega \left( s,t\right) :=\left\Vert X^{1}\right\Vert _{p-var,\left[ s,t%
\right] }^{p}+\left\Vert X^{2}\right\Vert _{p-var,\left[ s,t\right] }^{p}%
\text{.}
\end{equation*}%
For $\left[ s,t\right] \subset \left[ 0,T\right] $ satisfying $\omega \left(
s,t\right) \leq 1$, let $y^{i,s,t}:\left[ 0,1\right] \rightarrow G^{\left[ p%
\right] }\left( \mathcal{U}\right) $ be the solution of the ODE%
\begin{equation*}
dy_{r}^{i,s,t}=\sum_{k=1}^{\left[ p\right] }F\left( f^{i}\right) ^{\circ
k}\pi _{k}\left( \log X_{s,t}^{i}\right) \left( Id_{L^{\left[ p\right]
}\left( \mathcal{U}\right) }\right) \left( y_{r}^{i,s,t}\right) dr\text{, }%
r\in \left[ 0,1\right] \text{, }y_{0}^{i,s,t}=\xi ^{i}\text{.}
\end{equation*}%
For $u\in \left[ s,t\right] $, let $y^{i,s,u,t}:\left[ 0,2\right]
\rightarrow G^{\left[ p\right] }\left( \mathcal{U}\right) $ be the solution
of the ODE%
\begin{equation*}
dy_{r}^{i,s,u,t}=\left\{ 
\begin{array}{cc}
\sum_{k=1}^{\left[ p\right] }F\left( f^{i}\right) ^{\circ k}\pi _{k}\left(
\log X_{s,u}^{i}\right) \left( Id_{L^{\left[ p\right] }\left( \mathcal{U}%
\right) }\right) \left( y_{r}^{i,s,u,t}\right) dr, & r\in \left[ 0,1\right]
\\ 
\sum_{k=1}^{\left[ p\right] }F\left( f^{i}\right) ^{\circ k}\pi _{k}\left(
\log X_{u,t}^{i}\right) \left( Id_{L^{\left[ p\right] }\left( \mathcal{U}%
\right) }\right) \left( y_{r}^{i,s,u,t}\right) dr, & r\in \left[ 1,2\right]%
\end{array}%
\right. \text{, }y_{0}^{i,s,u,t}=\xi ^{i}\text{.}
\end{equation*}%
Then there exists a constant $C_{p}$ such that, for any $\left[ s,t\right]
\subseteq \left[ 0,T\right] $ satisfying $\omega \left( s,t\right) \leq 1$
and any $u\in \left[ s,t\right] $, (with $d_{p,\left[ s,t\right] }^{n}\left(
X^{1},X^{2}\right) $ defined at $\left( \ref{Definition of dpn}\right) $ on p%
\pageref{Definition of dpn}) 
\begin{eqnarray*}
&&\left\Vert y_{2}^{1,s,u,t}-y_{2}^{2,s,u,t}-\left(
y_{1}^{1,s,t}-y_{1}^{2,s,t}\right) \right\Vert \\
&\leq &C_{p}\left( \omega \left( s,t\right) ^{\frac{\gamma }{p}}\left(
\left\vert f^{1}-f^{2}\right\vert _{\gamma -1}+\left\Vert \xi ^{1}-\xi
^{2}\right\Vert \right) +\sum_{n=1}^{\left[ p\right] }\omega \left(
s,t\right) ^{\frac{\gamma -n}{p}}d_{p,\left[ s,t\right] }^{n}\left(
X^{1},X^{2}\right) \right) \text{.}
\end{eqnarray*}
\end{lemma}

\begin{proof}
Fix $\left[ s,t\right] \subseteq \left[ 0,T\right] $ satisfying $\omega
\left( s,t\right) \leq 1$. Since $\left\Vert \xi ^{i}\right\Vert \leq 1$ and 
$\omega \left( s,t\right) \leq 1$, (based on $\left( \ref{bound on y on
small interval}\right) $ in Lemma \ref{Lemma simple property of solution of
ODE} on p\pageref{bound on y on small interval}) there exists $C_{p}$ such
that 
\begin{equation}
\max_{i=1,2}\left\Vert y^{i,s,u,t}\right\Vert _{\infty }\vee \left\Vert
y^{i,s,t}\right\Vert _{\infty }\leq C_{p}\text{.}
\label{inner first estimate of uniform norm}
\end{equation}

Based on Lemma \ref{Lemma explicit Euler expansion} (explicit remainder of
Euler expansion of ODE) and Lemma \ref{Lemma simple but important}, and
using $\left( \ref{inner first estimate of uniform norm}\right) $ and Lemma %
\ref{Lemma log into exp} (replacing log-signature by signature), we have%
\begin{equation}
\left\Vert y_{0,2}^{1,s,u,t}-y_{0,2}^{2,s,u,t}-\left(
y_{0,1}^{1,s,t}-y_{0,1}^{2,s,t}\right) \right\Vert \leq C_{p}\left(
I+I\!I+I\!I\!I+I\!V\right) \text{,}  \label{inner eqn 1 Lemma B}
\end{equation}%
where ($\delta :=\omega \left( s,t\right) \leq 1$)%
\begin{eqnarray*}
I &=&\delta ^{\left[ p\right] }\left( \sup_{r\in \left[ 1,2\right]
}\left\Vert
y_{r}^{1,s,u,t}-y_{1}^{1,s,u,t}-y_{r}^{2,s,u,t}+y_{1}^{2,s,u,t}\right\Vert
\right) \\
&&+\left( \sup_{r\in \left[ 1,2\right] }\left\Vert
y_{r}^{1,s,u,t}-y_{1}^{1,s,u,t}\right\Vert +\sup_{r\in \left[ 1,2\right]
}\left\Vert y_{r}^{2,s,u,t}-y_{1}^{2,s,u,t}\right\Vert \right) ^{\left\{
\gamma \right\} } \\
&&\times \left( \delta ^{\left[ p\right] }\left\Vert
y_{1}^{1,s,u,t}-y_{1}^{2,s,u,t}\right\Vert +\delta ^{\left[ p\right]
}\left\vert f^{1}-f^{2}\right\vert _{\gamma -1}+\tsum\nolimits_{n=1}^{\left[
p\right] }\delta ^{\left[ p\right] -n}d_{p,\left[ s,t\right] }^{n}\left(
X^{1},X^{2}\right) \right) \text{,}
\end{eqnarray*}%
\begin{eqnarray*}
I\!I &=&\delta ^{\left[ p\right] }\left( \sup_{r\in \left[ 0,1\right]
}\left\Vert y_{r}^{1,s,u,t}-\xi ^{1}-y_{r}^{2,s,u,t}+\xi ^{2}\right\Vert
\right) \\
&&+\left( \sup_{r\in \left[ 0,1\right] }\left\Vert y_{r}^{1,s,u,t}-\xi
^{1}\right\Vert +\left\Vert y_{r}^{2,s,u,t}-\xi ^{2}\right\Vert \right)
^{\left\{ \gamma \right\} } \\
&&\times \left( \delta ^{\left[ p\right] }\left\Vert \xi ^{1}-\xi
^{2}\right\Vert +\delta ^{\left[ p\right] }\left\vert f^{1}-f^{2}\right\vert
_{\gamma -1}+\tsum\nolimits_{n=1}^{\left[ p\right] }\delta ^{\left[ p\right]
-n}d_{p,\left[ s,t\right] }^{n}\left( X^{1},X^{2}\right) \right) \text{,}
\end{eqnarray*}%
\begin{eqnarray*}
I\!I\!I &=&\delta ^{\left[ p\right] }\left( \sup_{r\in \left[ 0,1\right]
}\left\Vert y_{r}^{1,s,t}-\xi ^{1}-y_{r}^{2,s,t}+\xi ^{2}\right\Vert \right)
\\
&&+\left( \sup_{r\in \left[ 0,1\right] }\left\Vert y_{r}^{1,s,t}-\xi
^{1}\right\Vert +\left\Vert y_{r}^{2,s,t}-\xi ^{2}\right\Vert \right)
^{\left\{ \gamma \right\} } \\
&&\times \left( \delta ^{\left[ p\right] }\left\Vert \xi ^{1}-\xi
^{2}\right\Vert +\delta ^{\left[ p\right] }\left\vert f^{1}-f^{2}\right\vert
_{\gamma -1}+\tsum\nolimits_{n=1}^{\left[ p\right] }\delta ^{\left[ p\right]
-n}d_{p,\left[ s,t\right] }^{n}\left( X^{1},X^{2}\right) \right) \text{,}
\end{eqnarray*}%
\begin{equation*}
I\!V=\delta ^{\left[ p\right] +1}\left( \sup_{r\in \left[ 0,2\right]
}\left\Vert y_{r}^{1,s,u,t}-y_{r}^{2,s,u,t}\right\Vert +\sup_{r\in \left[ 0,1%
\right] }\left\Vert y_{r}^{1,s,t}-y_{r}^{2,s,t}\right\Vert +\left\vert
f^{1}-f^{2}\right\vert _{\gamma -1}\right) +\sum_{n=1}^{\left[ p\right]
}\delta ^{\left[ p\right] +1-n}d_{p,\left[ s,t\right] }^{n}\left(
X^{1},X^{2}\right) \text{.}
\end{equation*}

Based on Lemma \ref{Lemma two ODE with different vector fields} (continuous
dependence of ODE solution), we have%
\begin{equation*}
\left\Vert y^{1,s,u,t}-y^{2,s,u,t}\right\Vert _{\infty ,\left[ 0,2\right]
}\vee \left\Vert y^{1,s,t}-y^{2,s,t}\right\Vert _{\infty ,\left[ 0,1\right]
}\leq C_{p}\left( \left\Vert \xi ^{1}-\xi ^{2}\right\Vert +\left\vert
f^{1}-f^{2}\right\vert _{\gamma -1}+\tsum\nolimits_{n=1}^{\left[ p\right]
}d_{p,\left[ s,t\right] }^{n}\left( X^{1},X^{2}\right) \right) \text{,}
\end{equation*}%
and%
\begin{eqnarray*}
&&\sup_{r\in \left[ 0,1\right] }\left( \left\Vert
y_{r+1}^{1,s,u,t}-y_{1}^{1,s,u,t}-y_{r+1}^{2,s,u,t}+y_{1}^{2,s,u,t}\right%
\Vert \vee \left\Vert y_{r}^{1,s,u,t}-\xi ^{1}-y_{r}^{2,s,u,t}+\xi
^{2}\right\Vert \vee \left\Vert y_{r}^{1,s,t}-\xi ^{1}-y_{r}^{2,s,t}+\xi
^{2}\right\Vert \right) \\
&\leq &C_{p}\left( \delta \left( \left\Vert \xi ^{1}-\xi ^{2}\right\Vert
+\left\vert f^{1}-f^{2}\right\vert _{\gamma -1}\right)
+\tsum\nolimits_{n=1}^{\left[ p\right] }d_{p,\left[ s,t\right] }^{n}\left(
X^{1},X^{2}\right) \right) \text{,}
\end{eqnarray*}%
Based on Lemma \ref{Lemma Euler expansion of solution of ODE and difference
between one step and two steps} (error of high order Euler expansion), we
have%
\begin{equation*}
\sup_{r\in \left[ 0,1\right] }\left( \left\Vert
y_{r+1}^{i,s,u,t}-y_{1}^{i,s,u,t}\right\Vert \vee \left\Vert
y_{r}^{i,s,u,t}-\xi ^{i}\right\Vert \vee \left\Vert y_{r}^{i,s,t}-\xi
^{i}\right\Vert \right) \leq C_{p}\delta \text{.}
\end{equation*}%
By substituting these estimates into $\left( \ref{inner eqn 1 Lemma B}%
\right) $, we have 
\begin{equation*}
\left\Vert y_{2}^{1,s,u,t}-y_{2}^{2,s,u,t}-\left(
y_{1}^{1,s,t}-y_{1}^{2,s,t}\right) \right\Vert \leq C_{p}\left( \delta
^{\gamma }\left( \left\Vert \xi ^{1}-\xi ^{2}\right\Vert +\left\vert
f^{1}-f^{2}\right\vert _{\gamma -1}\right) +\tsum\nolimits_{n=1}^{\left[ p%
\right] }\delta ^{\gamma -n}d_{p,\left[ s,t\right] }^{n}\left(
X^{1},X^{2}\right) \right) \text{.}
\end{equation*}
\end{proof}

\subsection{RDE driven by weak geometric rough path in Banach space}

\begin{notation}
\label{Notation dyadic partition of a control}Suppose $\omega :\left\{
\left( s,t\right) |0\leq s\leq t\leq T\right\} \rightarrow \overline{%
\mathbb{R}
^{+}}$ is a control. For integer $n\geq 0$, let $D^{n}=\left\{
t_{j}^{n}\right\} _{j=0}^{2^{n}}$ be a sequence of nested finite partitions
of $\left[ 0,T\right] $, defined recursively as 
\begin{equation*}
t_{0}^{0}=0\text{, }t_{1}^{0}:=T\text{, }t_{2j}^{n+1}:=t_{j}^{n}\text{, }%
j=0,1,\dots ,2^{n}\text{, }n\geq 0\text{,}
\end{equation*}%
and $t_{2j+1}^{n+1}\in \left( t_{j}^{n},t_{j+1}^{n}\right) $ satisfying 
\begin{equation*}
\omega \left( t_{2j}^{n+1},t_{2j+1}^{n+1}\right) =\omega \left(
t_{2j+1}^{n+1},t_{2j+2}^{n+1}\right) \leq \frac{1}{2}\omega \left(
t_{j}^{n},t_{j+1}^{n}\right) \text{, }j=0,\dots ,2^{n}-1\text{, }n\geq 0%
\text{. }
\end{equation*}%
Denote%
\begin{equation*}
\Lambda \left( n\right) :=\left\{ t_{j}^{n}|j=0,\dots ,2^{n}\right\} \text{, 
}n\geq 0\text{.}
\end{equation*}%
We call $\left[ t_{j}^{n},t_{j+1}^{n}\right] $ a dyadic interval of level $n$
and call $t_{j}^{n}$ a dyadic point of level $n$.
\end{notation}

As a result, when the level of dyadic intervals increases, their "length"
decreases. Then we decompose an interval as union of dyadic intervals. (The
decomposition is in the same spirit as 4.1.1 in \cite{LyonsQian}\ or Lemma
28 in \cite{LyonsYang}.)

\begin{lemma}
\label{Lemma decompose an interval as union of dyadic intervals}For integer $%
n\geq 0$ and $\left\{ s,t\right\} \subseteq \Lambda \left( n\right) $,
denote by $n_{0}$ the level of biggest dyadic interval in $\left[ s,t\right] 
$. Then we can decompose $\left[ s,t\right] $ as union of dyadic intervals
in such a way that, there exists a dyadic point $p\in \left[ s,t\right] $ of
level $n_{0}-1$, such that the level of dyadic intervals to the left/right
of $p$ is strictly increasing.
\end{lemma}

\begin{proof}
We recursively cut out the biggest dyadic interval in $\left[ s,t\right] $,
and decompose $\left[ s,t\right] $ as union of dyadic intervals. Denote the
level of the biggest dyadic interval in $\left[ s,t\right] $ by $n_{0}$.
Then $n_{0}\leq n$, and there could be one level $n_{0}$ dyadic interval or
two adjacent level $n_{0}$ dyadic intervals in $\left[ s,t\right] $, but
there can not be more than two of them. Indeed, if there are more than two
level $n_{0}$ dyadic intervals, then (since $\left[ s,t\right] $ is
connected) two of them will compose a level $n_{0}-1$ dyadic interval, which
contradicts with our assumption that the biggest dyadic interval is of level 
$n_{0}$. Let $I_{l}/I_{r}$ denote the interval on the left/right side of
level $n_{0}$ dyadic interval(s) in $\left[ s,t\right] $. Since we cut out
the level $n_{0}$ dyadic interval(s) in $\left[ s,t\right] $, $I_{l}/I_{r}$
is strictly contained in a level $n_{0}$ dyadic interval, with its
right/left boundary point a level $n_{0}$ dyadic point. Thus, by recursively
cutting out the biggest dyadic interval in $I_{l}/I_{r}$, we decompose $%
I_{l}/I_{r}$ as the union of dyadic intervals which are strictly monotone in
their level. \ In this way, we decompose $\left[ s,t\right] $ as the union
of dyadic intervals. If there are two level $n_{0}$ dyadic intervals in $%
\left[ s,t\right] $ (denoted as $I^{1}$ and $I^{2}$), we select $p$ as the
point between $I^{1}$ and $I^{2}$, so $p$ is a level $n_{0}-1$ dyadic point.
If there is only one level $n_{0}$ dyadic interval (denoted as $I$), we
select $p$ as the boundary point of $I$ which is of level $n_{0}-1$. Based
on our construction, the level of dyadic interval(s) to the left/right of $p$
is strictly increasing.
\end{proof}

Lemma \ref{Lemma dyadic to non-dyadic} and Lemma \ref{Lemma from dyadic
difference to non-dyadic differene} extend estimates on dyadic intervals to
general intervals.

\begin{lemma}
\label{Lemma dyadic to non-dyadic}Suppose $\omega $ is a control with dyadic
partition $\Lambda \left( n\right) =\{t_{j}^{n}\}_{j}$ as in Notation \ref%
{Notation dyadic partition of a control}. Suppose $\gamma :\left[ 0,T\right]
\rightarrow \mathcal{U}$ is a continuous path, and for some $\theta >0$ and
some integer $n\geq 1$,%
\begin{equation*}
\left\Vert \gamma _{t}-\gamma _{s}\right\Vert \leq \omega \left( s,t\right)
^{\theta }\text{, }\forall \left[ s,t\right] =\left[ t_{j}^{l},t_{j+1}^{l}%
\right] \subseteq \left[ 0,T\right] \text{, }l=0,1,\dots ,n\text{.}
\end{equation*}%
Then there exists $C_{\theta }$ such that%
\begin{equation*}
\left\Vert \gamma _{t}-\gamma _{s}\right\Vert \leq C_{\theta }\omega \left(
s,t\right) ^{\theta }\text{, }\forall \left[ s,t\right] \subseteq \left[ 0,T%
\right] \text{, }\left\{ s,t\right\} \subseteq \Lambda \left( n\right) \text{%
.}
\end{equation*}
\end{lemma}

\begin{proof}
Fix $\left\{ s,t\right\} \subseteq \Lambda \left( n\right) $. Denote by $%
n_{0}$ the level of biggest dyadic interval in $\left[ s,t\right] $. Then we
decompose $\left[ s,t\right] $ as union of dyadic intervals as in Lemma \ref%
{Lemma decompose an interval as union of dyadic intervals}, so there exists
a level $n_{0}-1$ dyadic point $u\in \left[ s,t\right] $ such that the level
of dyadic intervals to the left/right of $u$ is strictly increasing. We
estimate $\left[ u,t\right] $ as an example. The estimation of $\left[ s,u%
\right] $ is similar. Suppose the dyadic decomposition of $\left[ u,t\right] 
$ is $\left[ t_{0},t_{1}\right] \cup \cdots \cup \left[ t_{l-1},t_{l}\right] 
$. Since the level of $\left[ t_{j},t_{j+1}\right] $ is strictly increasing
as $j$ increases and $u$ is a dyadic point of level $n_{0}-1$, we have, the
level of $t_{j}$ is strictly lower than the level of $\left[ t_{j},t_{j+1}%
\right] $, $j=0,1,\dots ,l-1$. Then,%
\begin{equation*}
\omega \left( t_{j},t_{j+1}\right) \leq \omega \left( t_{j},t_{l}\right)
\leq \frac{1}{2}\omega \left( t_{j-1},t_{l}\right) \leq \cdots \leq \frac{1}{%
2^{j}}\omega \left( t_{0},t_{l}\right) =\frac{1}{2^{j}}\omega \left(
u,t\right) \text{.}
\end{equation*}%
Thus, 
\begin{equation*}
\left\Vert \gamma _{t}-\gamma _{u}\right\Vert \leq
\sum_{j=0}^{l-1}\left\Vert \gamma _{t_{j+1}}-\gamma _{t_{j}}\right\Vert \leq
\sum_{j=0}^{l-1}\omega \left( t_{j},t_{j+1}\right) ^{\theta }\leq \left(
\sum_{j=0}^{l-1}\frac{1}{2^{j\theta }}\right) \omega \left( u,t\right)
^{\theta }\leq C_{\theta }\omega \left( u,t\right) ^{\theta }\text{.}
\end{equation*}%
Similarly, we get%
\begin{equation*}
\left\Vert \gamma _{u}-\gamma _{s}\right\Vert \leq C_{\theta }\omega \left(
s,u\right) ^{\theta }\text{.}
\end{equation*}%
Then for $\left\{ s,t\right\} \subseteq \Lambda \left( n\right) $, $\left[
s,t\right] \subseteq \left[ 0,T\right] $, we have%
\begin{equation*}
\left\Vert \gamma _{t}-\gamma _{s}\right\Vert \leq \left\Vert \gamma
_{u}-\gamma _{s}\right\Vert +\left\Vert \gamma _{t}-\gamma _{u}\right\Vert
\leq C_{\theta }\omega \left( s,u\right) ^{\theta }+C_{\theta }\omega \left(
u,t\right) ^{\theta }\leq C_{\theta }\omega \left( s,t\right) ^{\theta }%
\text{.}
\end{equation*}
\end{proof}

\begin{lemma}
\label{Lemma from dyadic difference to non-dyadic differene}For $\gamma
>p\geq 1$, suppose $X\in C^{p-var}\left( \left[ 0,T\right] ,G^{\left[ p%
\right] }\left( \mathcal{V}\right) \right) $ and $f\in L\left( \mathcal{V}%
,C^{\gamma }\left( \mathcal{U},\mathcal{U}\right) \right) $. Define control%
\begin{equation*}
\omega \left( s,t\right) :=\left\vert f\right\vert _{\gamma }^{p}\left\Vert
X\right\Vert _{p-var,\left[ s,t\right] }^{p}\text{, }\forall 0\leq s\leq
t\leq T\text{.}
\end{equation*}%
For some $C_{p}>0$ and $\left\{ s_{0},t_{0}\right\} \subseteq \Lambda \left(
n\right) $ satisfying $\omega \left( s_{0},t_{0}\right) \leq 1$, suppose $y:%
\left[ s_{0},t_{0}\right] \rightarrow G^{\left[ p\right] }\left( \mathcal{U}%
\right) $ is continuous and satisfies 
\begin{equation*}
\sup_{r\in \left[ s_{0},t_{0}\right] }\left\Vert y_{r}\right\Vert \leq C_{p}%
\text{.}
\end{equation*}%
For $\left\{ s,t\right\} \subseteq \Lambda \left( n\right) $, $\left[ s,t%
\right] \subseteq \left[ s_{0},t_{0}\right] $, let $y^{s,t}:\left[ 0,1\right]
\rightarrow G^{\left[ p\right] }\left( \mathcal{U}\right) $ be the solution
of the ODE%
\begin{equation*}
dy_{r}^{s,t}=\sum_{k=1}^{\left[ p\right] }F\left( f\right) ^{\circ k}\left(
\pi _{k}\left( \log X_{s,t}\right) \right) \left( Id_{L^{\left[ p\right]
}\left( \mathcal{U}\right) }\right) \left( y_{r}^{s,t}\right) dr\text{, }%
r\in \left[ 0,1\right] \text{, }y_{0}^{s,t}=y_{s}\text{.}
\end{equation*}%
Suppose for some $C_{p}>0$ and $\theta >0$, we have%
\begin{equation}
\left\Vert y_{t}-y_{1}^{s,t}\right\Vert \leq C_{p}\omega \left( s,t\right)
^{\theta }\text{, }\forall \left[ s,t\right] =\left[ t_{j}^{l},t_{j+1}^{l}%
\right] \subseteq \left[ s_{0},t_{0}\right] \text{, }l=0,1,\dots ,n\text{.}
\label{inner estimate on dyadic interval in Lemma from dyadic to non-dyadic}
\end{equation}%
Then for any $\left\{ s,t\right\} \in \Lambda \left( n\right) $, $\left[ s,t%
\right] \subseteq \left[ s_{0},t_{0}\right] $, we have%
\begin{equation*}
\left\Vert y_{t}-y_{1}^{s,t}\right\Vert \leq C_{p,\theta }\omega \left(
s,t\right) ^{\theta \wedge \left( \frac{\left[ p\right] +1}{p}\right) }\text{%
.}
\end{equation*}
\end{lemma}

\begin{proof}
When $\left[ s,t\right] \subseteq \left[ s_{0},t_{0}\right] $ is non-dyadic,
we decompose $\left[ s,t\right] $ as the union of dyadic intervals as in
Lemma \ref{Lemma decompose an interval as union of dyadic intervals}. Denote
the level of biggest dyadic interval in $\left[ s,t\right] $ by $n_{0}$.
Then based on Lemma \ref{Lemma decompose an interval as union of dyadic
intervals}, there exists a level $n_{0}-1$ dyadic point $u\in \left[ s,t%
\right] $, such that the level of dyadic intervals to the left/right of $u$
is strictly increasing. Then%
\begin{equation}
\left\Vert y_{t}-y_{1}^{s,t}\right\Vert \leq \left\Vert
y_{t}-y_{1}^{u,t}\right\Vert +\left\Vert y_{u}-y_{1}^{s,u}\right\Vert
+\left\Vert \left( y_{1}^{s,t}-y_{s}\right) -\left( y_{1}^{u,t}-y_{u}\right)
-\left( y_{1}^{s,u}-y_{s}\right) \right\Vert \text{.}
\label{inner decompose s,t as s,u u,t}
\end{equation}%
We take $\left\Vert y_{t}-y_{1}^{u,t}\right\Vert $ as an example. The
estimation of $\left\Vert y_{u}-y_{1}^{s,u}\right\Vert $ is similar.

Denote the dyadic decomposition of $\left[ u,t\right] $ (as in Lemma \ref%
{Lemma decompose an interval as union of dyadic intervals}) by $\left[
t_{0},t_{1}\right] \cup \cdots \cup \left[ t_{l-1},t_{l}\right] $. For $%
j=1,\dots ,l-1$, let $y^{t_{j-1},t_{j},t_{l}}$ denote the solution of the
ODE: 
\begin{equation}
dy_{r}^{t_{j-1},t_{j},t_{l}}=\left\{ 
\begin{array}{cc}
\sum_{k=1}^{\left[ p\right] }F\left( f\right) ^{\circ k}\pi _{k}\left( \log
X_{t_{j-1},t_{j}}\right) \left( Id_{L^{\left[ p\right] }\left( \mathcal{U}%
\right) }\right) \left( y_{r}^{t_{j-1},t_{j},t_{l}}\right) dr, & r\in \left[
0,1\right] \\ 
\sum_{k=1}^{\left[ p\right] }F\left( f\right) ^{\circ k}\pi _{k}\left( \log
X_{t_{j},t_{l}}\right) \left( Id_{L^{\left[ p\right] }\left( \mathcal{U}%
\right) }\right) \left( y_{r}^{t_{j-1},t_{j},t_{l}}\right) dr, & r\in \left[
1,2\right]%
\end{array}%
\right. \text{, }y_{0}^{t_{j-1},t_{j},t_{l}}=y_{t_{j-1}}\text{.}
\label{inner definition of two steps ode}
\end{equation}%
Then 
\begin{eqnarray}
&&\left\Vert y_{t}-y_{1}^{u,t}\right\Vert
\label{inner estimate of difference on u,t} \\
&=&\left\Vert y_{t}-y_{u}-\left( y_{1}^{u,t}-y_{u}\right) \right\Vert  \notag
\\
&\leq &\sum_{j=0}^{l-1}\left\Vert y_{t_{j+1}}-y_{t_{j}}-\left(
y_{1}^{t_{j},t_{j+1}}-y_{t_{j}}\right) \right\Vert  \notag \\
&&+\sum_{j=1}^{l-1}\left\Vert \left(
y_{1}^{t_{j-1},t_{j}}-y_{t_{j-1}}\right) +\left(
y_{1}^{t_{j},t_{l}}-y_{t_{j}}\right) -\left(
y_{1}^{t_{j-1},t_{l}}-y_{t_{j-1}}\right) \right\Vert  \notag \\
&\leq &\sum_{j=0}^{l-1}\left\Vert y_{t_{j+1}}-y_{t_{j}}-\left(
y_{1}^{t_{j},t_{j+1}}-y_{t_{j}}\right) \right\Vert  \notag \\
&&+\sum_{j=1}^{l-1}\left( \left\Vert \left(
y_{2}^{t_{j-1},t_{j},t_{l}}-y_{1}^{t_{j-1},t_{j}}\right) -\left(
y_{1}^{t_{j},t_{l}}-y_{t_{j}}\right) \right\Vert +\left\Vert
y_{2}^{t_{j-1},t_{j},t_{l}}-y_{1}^{t_{j-1},t_{l}}\right\Vert \right) \text{.}
\notag
\end{eqnarray}

Since the level of $t_{j}$ is strictly lower than the level of $\left[
t_{j},t_{j+1}\right] $ and the level of $\left[ t_{j},t_{j+1}\right] $ is
strictly increasing as $j$ increases, we have, for $j=0,\dots ,l-1$,%
\begin{equation}
\omega \left( t_{j},t_{j+1}\right) \leq \omega \left( t_{j},t_{l}\right)
\leq \frac{1}{2}\omega \left( t_{j-1},t_{l}\right) \leq \cdots \leq \frac{1}{%
2^{j}}\omega \left( u,t\right) \text{.}
\label{inner estimate of non-dyadic interval}
\end{equation}%
Then since $\left\{ \left[ t_{j},t_{j+1}\right] \right\} _{j}$ are dyadic,
by using assumption $\left( \ref{inner estimate on dyadic interval in Lemma
from dyadic to non-dyadic}\right) $, we have 
\begin{eqnarray}
\sum_{j=0}^{l-1}\left\Vert y_{t_{j+1}}-y_{t_{j}}-\left(
y_{1}^{t_{j},t_{j+1}}-y_{t_{j}}\right) \right\Vert &\leq
&C_{p}\sum_{j=0}^{l-1}\omega \left( t_{j},t_{j+1}\right) ^{\theta }
\label{inner estimate of difference on u,t 1} \\
&\leq &C_{p}\left( \sum_{j=0}^{\infty }\left( \frac{1}{2}\right) ^{j\theta
}\right) \omega \left( u,t\right) ^{\theta }\leq C_{p,\theta }\omega \left(
u,t\right) ^{\theta }\text{.}  \notag
\end{eqnarray}

On the other hand, since $\sup_{r\in \left[ s_{0},t_{0}\right] }\left\Vert
y_{r}\right\Vert \leq C_{p}$, by using Lemma \ref{Lemma two ODE with
different vector fields} on p\pageref{Lemma two ODE with different vector
fields} (continuous dependence on initial value) and using the assumption on
dyadic interval $\left[ t_{j-1},t_{j}\right] $ at $\left( \ref{inner
estimate on dyadic interval in Lemma from dyadic to non-dyadic}\right) $, we
have%
\begin{equation}
\left\Vert \left( y_{2}^{t_{j-1},t_{j},t_{l}}-y_{1}^{t_{j-1},t_{j}}\right)
-\left( y_{1}^{t_{j},t_{l}}-y_{t_{j}}\right) \right\Vert \leq C_{p}\omega
\left( t_{j},t_{l}\right) ^{\frac{1}{p}}\left\Vert
y_{1}^{t_{j-1},t_{j}}-y_{t_{j}}\right\Vert \leq C_{p}\omega \left(
t_{j-1},t_{l}\right) ^{\theta +\frac{1}{p}}\text{.}
\label{inner estimate of non-dyadic interval lemma a}
\end{equation}%
Based on $\left( \ref{equ 2 in Lemma of ODE}\right) $ in Lemma \ref{Lemma
Euler expansion of solution of ODE and difference between one step and two
steps} on p\pageref{equ 2 in Lemma of ODE} (error between two-steps ODE and
one-step ODE) and $\sup_{r\in \left[ s_{0},t_{0}\right] }\left\Vert
y_{r}\right\Vert \leq C_{p}$, 
\begin{equation}
\left\Vert y_{2}^{t_{j-1},t_{j},t_{l}}-y_{1}^{t_{j-1},t_{l}}\right\Vert \leq
C_{p}\omega \left( t_{j-1},t_{l}\right) ^{\frac{\left[ p\right] +1}{p}}\text{%
.}  \label{inner estimate of non-dyadic interval lemma b}
\end{equation}%
Then, combining $\left( \ref{inner estimate of non-dyadic interval lemma a}%
\right) $, $\left( \ref{inner estimate of non-dyadic interval lemma b}%
\right) $ and $\left( \ref{inner estimate of non-dyadic interval}\right) $,
we have ($\omega \left( u,t\right) \leq \omega \left( s_{0},t_{0}\right)
\leq 1$)%
\begin{eqnarray}
&&\sum_{j=1}^{l-1}\left( \left\Vert \left(
y_{2}^{t_{j-1},t_{j},t_{l}}-y_{1}^{t_{j-1},t_{j}}\right) -\left(
y_{1}^{t_{j},t_{l}}-y_{t_{j}}\right) \right\Vert +\left\Vert
y_{2}^{t_{j-1},t_{j},t_{l}}-y_{1}^{t_{j-1},t_{l}}\right\Vert \right)
\label{inner estimate of difference on u,t 2} \\
&\leq &C_{p}\sum_{j=1}^{l-1}\omega \left( t_{j-1},t_{l}\right) ^{\left(
\theta +\frac{1}{p}\right) \wedge \left( \frac{\left[ p\right] +1}{p}\right)
}\leq C_{p}\left( \sum_{j=1}^{l-1}\left( \frac{1}{2^{j-1}}\right) ^{\left(
\theta +\frac{1}{p}\right) \wedge \left( \frac{\left[ p\right] +1}{p}\right)
}\right) \omega \left( u,t\right) ^{\left( \theta +\frac{1}{p}\right) \wedge
\left( \frac{\left[ p\right] +1}{p}\right) }  \notag \\
&\leq &C_{p,\theta }\omega \left( u,t\right) ^{\left( \theta +\frac{1}{p}%
\right) \wedge \left( \frac{\left[ p\right] +1}{p}\right) }  \notag
\end{eqnarray}

As a result, by combining $\left( \ref{inner estimate of difference on u,t}%
\right) $, $\left( \ref{inner estimate of difference on u,t 1}\right) $ and $%
\left( \ref{inner estimate of difference on u,t 2}\right) $, we have%
\begin{equation*}
\left\Vert y_{t}-y_{1}^{u,t}\right\Vert \leq C_{p,\theta }\omega \left(
u,t\right) ^{\theta \wedge \left( \frac{\left[ p\right] +1}{p}\right) }\text{%
.}
\end{equation*}%
Similarly, we have%
\begin{equation}
\left\Vert y_{u}-y_{1}^{s,u}\right\Vert \leq C_{p,\theta }\omega \left(
s,u\right) ^{\theta \wedge \left( \frac{\left[ p\right] +1}{p}\right) }\text{%
.}  \label{inner estimate on [s,u]}
\end{equation}%
On the other hand, with $y^{s,u,t}$ defined at $\left( \ref{inner definition
of two steps ode}\right) $ and by using $\left( \ref{inner estimate on [s,u]}%
\right) $, similar estimate as at $\left( \ref{inner estimate of non-dyadic
interval lemma a}\right) $ and $\left( \ref{inner estimate of non-dyadic
interval lemma b}\right) $ lead to 
\begin{eqnarray*}
&&\left\Vert \left( y_{1}^{s,t}-y_{s}\right) -\left(
y_{1}^{u,t}-y_{u}\right) -\left( y_{1}^{s,u}-y_{s}\right) \right\Vert \\
&\leq &\left\Vert y_{2}^{s,u,t}-y_{1}^{s,u}-\left( y_{1}^{u,t}-y_{u}\right)
\right\Vert +\left\Vert y_{2}^{s,u,t}-y_{1}^{s,t}\right\Vert \\
&\leq &C_{p}\omega \left( u,t\right) ^{\frac{1}{p}}\left\Vert
y_{1}^{s,u}-y_{u}\right\Vert +C_{p}\omega \left( s,t\right) ^{\frac{\left[ p%
\right] +1}{p}} \\
&\leq &C_{p,\theta }\omega \left( s,t\right) ^{\theta \wedge \left( \frac{%
\left[ p\right] +1}{p}\right) }\text{,}
\end{eqnarray*}%
As a result, combined with $\left( \ref{inner decompose s,t as s,u u,t}%
\right) $, we have, for any $\left\{ s,t\right\} \in \Lambda \left( n\right) 
$, $\left[ s,t\right] \subseteq \left[ s_{0},t_{0}\right] $,%
\begin{equation}
\left\Vert y_{t}-y_{1}^{s,t}\right\Vert \leq C_{p,\theta }\omega \left(
s,t\right) ^{\theta \wedge \left( \frac{\left[ p\right] +1}{p}\right) }\text{%
.}  \label{inner estimate of difference}
\end{equation}
\end{proof}

\begin{lemma}
\label{Lemma dyadic paths are uniformly bounded}For $\gamma >p\geq 1$,
suppose $X\in C^{p-var}\left( \left[ 0,T\right] ,G^{\left[ p\right] }\left( 
\mathcal{V}\right) \right) $, $f\in L\left( \mathcal{V},C^{\gamma }\left( 
\mathcal{U},\mathcal{U}\right) \right) $ and $\xi \in G^{\left[ p\right]
}\left( \mathcal{U}\right) $. Define control 
\begin{equation*}
\omega \left( s,t\right) :=\left\vert f\right\vert _{\gamma }^{p}\left\Vert
X\right\Vert _{p-var,\left[ s,t\right] }^{p}\text{, }\forall 0\leq s\leq
t\leq T\text{.}
\end{equation*}%
Based on $\omega $, define dyadic partitions $\Lambda \left( n\right)
:=\left\{ t_{j}^{n}\right\} $, $n\geq 0$, as in Notation \ref{Notation
dyadic partition of a control}. For $n\geq 0$, let $y^{n}:\left[ 0,T\right]
\rightarrow G^{\left[ p\right] }\left( \mathcal{U}\right) $ be the solution
of the ordinary differential equation 
\begin{eqnarray}
y_{0}^{n} &=&\xi ,  \label{Definition of yn} \\
dy_{t}^{n} &=&\sum_{k=1}^{\left[ p\right] }F\left( f\right) ^{\circ k}\pi
_{k}\left( \log X_{t_{j}^{n},t_{j+1}^{n}}\right) \left( Id_{L^{\left[ p%
\right] }\left( \mathcal{U}\right) }\right) \left( y_{t}^{n}\right) \frac{dt%
}{t_{j+1}^{n}-t_{j}^{n}}\text{, }t\in \left[ t_{j}^{n},t_{j+1}^{n}\right] 
\text{.}  \notag
\end{eqnarray}%
Then there exists $C_{p}$ such that, ($y_{s,t}^{n}:=\left( y_{s}^{n}\right)
^{-1}\otimes y_{t}^{n}$) 
\begin{equation}
\left\vert \!\left\vert \!\left\vert y_{s,t}^{n}\right\vert \!\right\vert
\!\right\vert \leq C_{p}\omega \left( s,t\right) ^{\frac{1}{p}}\text{, }%
\forall \left\{ s,t\right\} \subseteq \Lambda \left( n\right) \,\text{, }%
\omega \left( s,t\right) \leq 1\text{, }\forall n\geq 1\text{.}
\label{estimate uniform bound on dyadic solution on small interval}
\end{equation}%
Moreover, for $n\geq 1$, $\left\{ s_{0},t_{0},s,t\right\} \subseteq \Lambda
\left( n\right) $ satisfying $\left[ s,t\right] \subseteq \left[ s_{0},t_{0}%
\right] $ and $\omega \left( s_{0},t_{0}\right) \leq 1$, let $y^{n,s,t}:%
\left[ 0,1\right] \rightarrow G^{\left[ p\right] }\left( \mathcal{U}\right) $
denote the solution of the ordinary differential equation%
\begin{equation}
dy_{u}^{n,s,t}=\sum_{k=1}^{\left[ p\right] }\left( F\left( f\left( \cdot
+\pi _{1}\left( y_{s_{0}}^{n}\right) \right) \right) \right) ^{\circ k}\pi
_{k}\left( \log X_{s,t}\right) \left( Id_{L^{\left[ p\right] }\left( 
\mathcal{U}\right) }\right) \left( y_{u}^{n,s,t}\right) du,u\in \left[ 0,1%
\right] ,y_{0}^{n,s,t}=y_{s_{0},s}^{n}\text{.}
\label{inner definition of y(s,t)}
\end{equation}%
Then there exists $C_{p}$ such that, 
\begin{equation}
\left\Vert y_{s_{0},t}^{n}-y_{1}^{n,s,t}\right\Vert \leq C_{p}\omega \left(
s,t\right) ^{\frac{\left[ p\right] +1}{p}}\text{, }\forall \left\{
s_{0},t_{0},s,t\right\} \subseteq \Lambda \left( n\right) \text{, }\left[ s,t%
\right] \subseteq \left[ s_{0},t_{0}\right] \text{, }\omega \left(
s_{0},t_{0}\right) \leq 1\text{, }\forall n\geq 1\text{.}
\label{estimate uniform bound on difference on small interval}
\end{equation}
\end{lemma}

\begin{proof}
We assume $\left\vert f\right\vert _{\gamma }=1$. Otherwise, we replace $f$
and $X$ by $\left\vert f\right\vert _{\gamma }^{-1}f$ and $\delta
_{\left\vert f\right\vert _{\gamma }}X$ respectively. In that case, both $%
y^{n}$ and $y^{n,s,t}$ will stay unchanged.

Fix $n\geq 1$ and $s_{0},t_{0}\in \Lambda \left( n\right) $ satisfying $%
\omega \left( s_{0},t_{0}\right) \leq 1$.

For $\left\{ s,u,t\right\} \subseteq \Lambda \left( n\right) $, $s_{0}\leq
s\leq u\leq t\leq t_{0}$, let $y^{n,s,u,t}:\left[ 0,2\right] \rightarrow G^{%
\left[ p\right] }\left( \mathcal{U}\right) $ be the solution of the ordinary
differential equation%
\begin{equation}
dy_{r}^{n,s,u,t}=\left\{ 
\begin{array}{cc}
\sum_{k=1}^{\left[ p\right] }F\left( f\left( \cdot +\pi _{1}\left(
y_{s_{0}}^{n}\right) \right) \right) ^{\circ k}\pi _{k}\left( \log
X_{s,u}\right) \left( Id_{L^{\left[ p\right] }\left( \mathcal{U}\right)
}\right) \left( y_{r}^{n,s,u,t}\right) dr, & r\in \left[ 0,1\right] \\ 
\sum_{k=1}^{\left[ p\right] }F\left( f\left( \cdot +\pi _{1}\left(
y_{s_{0}}^{n}\right) \right) \right) ^{\circ k}\pi _{k}\left( \log
X_{u,t}\right) \left( Id_{L^{\left[ p\right] }\left( \mathcal{U}\right)
}\right) \left( y_{r}^{n,s,u,t}\right) dr, & r\in \left[ 1,2\right]%
\end{array}%
\right. \text{, }y_{0}^{n,s,u,t}=y_{s_{0},s}^{n}\text{.}
\label{inner definition of y(s,u,t) in Lemma of uniform bound of dyadic ODEs}
\end{equation}

To simply the notation, we omit $n$ and denote%
\begin{equation*}
y^{s,t}:=y^{n,s,t}\text{ and }y^{s,u,t}:=y^{n,s,u,t}.
\end{equation*}%
Yet the coefficients below are all independent of $n$. For $\left\{
s,t\right\} \subseteq \Lambda \left( n\right) $, $\left[ s,t\right]
\subseteq \left[ s_{0},t_{0}\right] $, denote%
\begin{equation}
\Gamma
^{s,t}:=y_{s_{0},t}^{n}-y_{1}^{s,t}=y_{s_{0},t}^{n}-y_{s_{0},s}^{n}-\left(
y_{1}^{s,t}-y_{s_{0},s}^{n}\right) \text{.}
\label{inner definition of Gamma(s,t)}
\end{equation}%
Based on the definition of $y^{n}$ at $\left( \ref{Definition of yn}\right) $
and the definition of $y^{s,t}$ at $\left( \ref{inner definition of y(s,t)}%
\right) $, it can be checked that, on any level-$n$ dyadic interval $\left[
t_{j}^{n},t_{j+1}^{n}\right] \subseteq \left[ s_{0},t_{0}\right] $, 
\begin{equation}
\Gamma
^{t_{j}^{n},t_{j+1}^{n}}=y_{s_{0},t_{j+1}^{n}}^{n}-y_{1}^{t_{j}^{n},t_{j+1}^{n}}=0%
\text{.}  \label{inner Gamma vanishes on small dyadic interval}
\end{equation}

Suppose $u\in \left( s,t\right) $, $u\in \Lambda \left( n\right) $, then%
\begin{eqnarray}
\Gamma ^{s,u}+\Gamma ^{u,t}-\Gamma ^{s,t}
&=&y_{1}^{s,t}-y_{s_{0},s}^{n}-\left( y_{1}^{s,u}-y_{s_{0},s}^{n}\right)
-\left( y_{1}^{u,t}-y_{s_{0},u}^{n}\right)
\label{inner relationship of Gamma (s,t) after bisection} \\
&=&\left( y_{1}^{s,t}-y_{2}^{s,u,t}\right) +\left(
y_{2}^{s,u,t}-y_{1}^{s,u}-\left( y_{1}^{u,t}-y_{s_{0},u}^{n}\right) \right) 
\notag
\end{eqnarray}

For $k=0,1,\dots ,\left[ p\right] $, denote%
\begin{equation}
M_{k}:=\max_{j=0,1,\dots ,k}\,\sup_{\omega \left( s_{0},s\right) \leq 1,s\in
\Lambda \left( n\right) }\,\sup_{n\geq 1}\left\Vert \pi _{j}\left(
y_{s_{0},s}^{n}\right) \right\Vert \text{.}  \label{inner definition of Mk}
\end{equation}

We first prove that $M_{\left[ p\right] }\leq C_{p}$. It is clear that $%
M_{0}=1$. Based on Lemma \ref{Lemma Euler expansion of solution of ODE and
difference between one step and two steps} on p\pageref{Lemma Euler
expansion of solution of ODE and difference between one step and two steps}
(error between two-steps ODE and one-step ODE), we have,%
\begin{equation}
\left\Vert \pi _{k}\left( y_{2}^{s,u,t}-y_{1}^{s,t}\right) \right\Vert \leq
C_{p}M_{k-1}\omega \left( s,t\right) ^{\frac{\left[ p\right] +1}{p}}.
\label{inner lemma A in Lemma uniform bound on dyadic intervals}
\end{equation}%
On the other hand,%
\begin{equation}
y_{2}^{s,u,t}-y_{1}^{s,u}-\left( y_{1}^{u,t}-y_{s_{0},u}^{n}\right)
=y_{1}^{s,u}\otimes \left( \left( y_{1}^{s,u}\right) ^{-1}\otimes
y_{2}^{s,u,t}-1\right) -y_{s_{0},u}^{n}\otimes \left( \left(
y_{s_{0},u}^{n}\right) ^{-1}\otimes y_{1}^{u,t}-1\right) \text{.}
\label{inner re-rewriting in Lemma uniform bound on dyadic intervals}
\end{equation}%
Based on the explicit expression of $F\left( f\right) ^{\circ k}$ in Lemma %
\ref{Lemma explicit form of Fcirck} on p\pageref{Lemma explicit form of
Fcirck}, $\left( y_{1}^{s,u}\right) ^{-1}\otimes y_{r+1}^{s,u,t}-1$ and $%
\left( y_{s_{0},u}^{n}\right) ^{-1}\otimes y_{r}^{u,t}-1$, $r\in \left[ 0,1%
\right] $, are respectively the solution of the ODE 
\begin{eqnarray*}
dy_{r} &=&\tsum\nolimits_{k=1}^{\left[ p\right] }F\left( f\left( \cdot +\pi
_{1}\left( y_{1}^{s,u}\right) +\pi _{1}\left( y_{s_{0}}^{n}\right) \right)
\right) ^{\circ k}\left( \log X_{u,t}\right) \left( Id_{L^{\left[ p\right]
}\left( \mathcal{U}\right) }\right) \left( y_{r}\right) dr\text{, }r\in %
\left[ 0,1\right] \text{, }y_{0}=0\text{,} \\
dy_{r} &=&\tsum\nolimits_{k=1}^{\left[ p\right] }F\left( f\left( \cdot +\pi
_{1}\left( y_{s_{0},u}^{n}\right) +\pi _{1}\left( y_{s_{0}}^{n}\right)
\right) \right) ^{\circ k}\left( \log X_{u,t}\right) \left( Id_{L^{\left[ p%
\right] }\left( \mathcal{U}\right) }\right) \left( y_{r}\right) dr\text{, }%
r\in \left[ 0,1\right] \text{, }y_{0}=0\text{.}
\end{eqnarray*}%
It is clear that 
\begin{equation*}
\pi _{0}\left( \left( y_{1}^{s,u}\right) ^{-1}\otimes y_{2}^{s,u,t}-1\right)
=\pi _{0}\left( \left( y_{s_{0},u}^{n}\right) ^{-1}\otimes
y_{1}^{u,t}-1\right) =0\text{.}
\end{equation*}%
Based on Lemma \ref{Lemma two ODE with different vector fields} on p\pageref%
{Lemma two ODE with different vector fields}, we have, for $k=1,\dots ,\left[
p\right] $,%
\begin{eqnarray}
&&\left\Vert \pi _{k}\left( \left( y_{1}^{s,u}\right) ^{-1}\otimes
y_{2}^{s,u,t}-1\right) -\pi _{k}\left( \left( y_{s_{0},u}^{n}\right)
^{-1}\otimes y_{1}^{u,t}-1\right) \right\Vert
\label{inner estimate lemma B in Lemma of uniform bound on dyadic intervals 2}
\\
&\leq &C_{p}\omega \left( u,t\right) ^{\frac{k}{p}}\left\vert f\left( \cdot
+\pi _{1}\left( y_{1}^{s,u}\right) +\pi _{1}\left( y_{s_{0}}^{n}\right)
\right) -f\left( \cdot +\pi _{1}\left( y_{s_{0},u}^{n}\right) +\pi
_{1}\left( y_{s_{0}}^{n}\right) \right) \right\vert _{\left[ p\right] -1} 
\notag \\
&\leq &C_{p}\omega \left( u,t\right) ^{\frac{k}{p}}\left\Vert \pi _{1}\left(
y_{1}^{s,u}\right) -\pi _{1}\left( y_{s_{0},u}^{n}\right) \right\Vert
=C_{p}\omega \left( u,t\right) ^{\frac{k}{p}}\left\Vert \pi _{1}\left(
\Gamma ^{s,u}\right) \right\Vert \text{,}  \notag
\end{eqnarray}%
and based on $\left( \ref{bound on y on small interval}\right) $ on p\pageref%
{bound on y on small interval} in\ Lemma \ref{Lemma simple property of
solution of ODE},%
\begin{equation}
\left\Vert \pi _{k}\left( \left( y_{1}^{s,u}\right) ^{-1}\otimes
y_{2}^{s,u,t}-1\right) \right\Vert \leq C_{p}\omega \left( u,t\right) ^{%
\frac{k}{p}}\text{.}
\label{inner estimate lemma B in Lemma of uniform bound on dyadic intervals 1}
\end{equation}%
Then by combining $\left( \ref{inner definition of Gamma(s,t)}\right) $, $%
\left( \ref{inner definition of Mk}\right) $, $\left( \ref{inner
re-rewriting in Lemma uniform bound on dyadic intervals}\right) $, $\left( %
\ref{inner estimate lemma B in Lemma of uniform bound on dyadic intervals 2}%
\right) $ and $\left( \ref{inner estimate lemma B in Lemma of uniform bound
on dyadic intervals 1}\right) $, we have%
\begin{eqnarray}
&&\left\Vert \pi _{k}\left( y_{2}^{s,u,t}-y_{1}^{s,u}-\left(
y_{1}^{u,t}-y_{s_{0},u}^{n}\right) \right) \right\Vert
\label{inner estimate lemma B in Lemma of uniform bound on dyadic intervals}
\\
&\leq &\sum_{j=1}^{k-1}\left\Vert \pi _{j}\left( y_{1}^{s,u}\right) -\pi
_{j}\left( y_{s_{0},u}^{n}\right) \right\Vert \left\Vert \pi _{k-j}\left(
\left( y_{1}^{s,u}\right) ^{-1}\otimes y_{2}^{s,u,t}-1\right) \right\Vert 
\notag \\
&&+\sum_{j=0}^{k-1}\left\Vert \pi _{j}\left( y_{s_{0},u}^{n}\right)
\right\Vert \left\Vert \pi _{k-j}\left( \left( y_{1}^{s,u}\right)
^{-1}\otimes y_{2}^{s,u,t}-1\right) -\pi _{k-j}\left( \left(
y_{s_{0},u}^{n}\right) ^{-1}\otimes y_{1}^{u,t}-1\right) \right\Vert  \notag
\\
&\leq &C_{p}\sum_{j=1}^{k-1}\left\Vert \pi _{j}\left( \Gamma ^{s,u}\right)
\right\Vert \omega \left( u,t\right) ^{\frac{k-j}{p}}+C_{p}%
\sum_{j=0}^{k-1}M_{j}\omega \left( u,t\right) ^{\frac{k-j}{p}}\left\Vert \pi
_{1}\left( \Gamma ^{s,u}\right) \right\Vert  \notag \\
&\leq &C_{p}\sum_{j=1}^{k-1}\omega \left( u,t\right) ^{\frac{k-j}{p}%
}\left\Vert \pi _{j}\left( \Gamma ^{s,u}\right) \right\Vert
+C_{p}M_{k-1}\omega \left( u,t\right) ^{\frac{1}{p}}\left\Vert \pi
_{1}\left( \Gamma ^{s,u}\right) \right\Vert  \notag
\end{eqnarray}%
As a result, combining $\left( \ref{inner relationship of Gamma (s,t) after
bisection}\right) $, $\left( \ref{inner lemma A in Lemma uniform bound on
dyadic intervals}\right) $ and $\left( \ref{inner estimate lemma B in Lemma
of uniform bound on dyadic intervals}\right) $, we have 
\begin{eqnarray}
&&\left\Vert \pi _{k}\left( \Gamma ^{s,t}\right) -\pi _{k}\left( \Gamma
^{s,u}\right) -\pi _{k}\left( \Gamma ^{u,t}\right) \right\Vert
\label{inner induction of Gamma on level k} \\
&\leq &C_{p}\left( M_{k-1}\omega \left( s,t\right) ^{\frac{\left[ p\right] +1%
}{p}}+M_{k-1}\omega \left( u,t\right) ^{\frac{1}{p}}\left\Vert \pi
_{1}\left( \Gamma ^{s,u}\right) \right\Vert +\sum_{j=1}^{k-1}\omega \left(
s,t\right) ^{\frac{k-j}{p}}\left\Vert \pi _{j}\left( \Gamma ^{s,u}\right)
\right\Vert \right) \text{.}  \notag
\end{eqnarray}

In particular, since $M_{0}=1$, we have%
\begin{equation*}
\left\Vert \pi _{1}\left( \Gamma ^{s,t}\right) -\pi _{1}\left( \Gamma
^{s,u}\right) -\pi _{1}\left( \Gamma ^{u,t}\right) \right\Vert \leq
C_{p}\left( \omega \left( s,t\right) ^{\frac{\left[ p\right] +1}{p}}+\omega
\left( u,t\right) ^{\frac{1}{p}}\left\Vert \pi _{1}\left( \Gamma
^{s,u}\right) \right\Vert \right) \text{.}
\end{equation*}%
When $\left[ s,t\right] \subseteq \left[ s_{0},t_{0}\right] $ is dyadic, by
recursively bisecting $\left[ s,t\right] $, we have%
\begin{eqnarray*}
\left\Vert \pi _{1}\left( \Gamma ^{s,t}\right) \right\Vert &\leq &\left(
\sum_{k=0}^{n}\left( \prod\limits_{j=0}^{k}\left( 1+2^{-\frac{j}{p}%
}C_{p}\omega \left( s,t\right) ^{\frac{1}{p}}\right) \right) \left( \frac{1}{%
2}\right) ^{\left( \frac{\left[ p\right] +1}{p}-1\right) k}\right) \omega
\left( s,t\right) ^{\frac{\left[ p\right] +1}{p}} \\
&&+\left( \prod\limits_{j=0}^{n}\left( 1+2^{-\frac{j}{p}}C_{p}\omega \left(
s,t\right) ^{\frac{1}{p}}\right) \right) \left( \sum_{\left[
t_{j}^{n},t_{j+1}^{n}\right] \subseteq \left[ s,t\right] }\left\Vert \pi
_{1}\left( \Gamma ^{t_{j}^{n},t_{j+1}^{n}}\right) \right\Vert \right) \text{.%
}
\end{eqnarray*}%
By using $\Gamma ^{t_{j}^{n},t_{j+1}^{n}}=0$ as at $\left( \ref{inner Gamma
vanishes on small dyadic interval}\right) $ and that $\omega \left(
s,t\right) \leq \omega \left( s_{0},t_{0}\right) \leq 1$, we have%
\begin{eqnarray}
\left\Vert \pi _{1}\left( \Gamma ^{s,t}\right) \right\Vert &\leq &\left(
\sum_{k=0}^{n}\left( \prod\limits_{j=0}^{k}\left( 1+2^{-\frac{j}{p}}\omega
\left( s,t\right) ^{\frac{1}{p}}\right) \right) \left( \frac{1}{2}\right)
^{\left( \frac{\left[ p\right] +1}{p}-1\right) k}\right) \omega \left(
s,t\right) ^{\frac{\left[ p\right] +1}{p}}
\label{inner estimate on dyadic interval} \\
&\leq &\exp \left( \frac{2^{\frac{1}{p}}}{2^{\frac{1}{p}}-1}\right) \frac{2^{%
\frac{\left[ p\right] +1}{p}-1}}{2^{\frac{\left[ p\right] +1}{p}-1}-1}\omega
\left( s,t\right) ^{\frac{\left[ p\right] +1}{p}}=C_{p}\omega \left(
s,t\right) ^{\frac{\left[ p\right] +1}{p}}\text{.}  \notag
\end{eqnarray}%
Thus, for any dyadic interval $\left[ s,t\right] \subseteq \left[ s_{0},t_{0}%
\right] $, we have%
\begin{equation*}
\left\Vert \pi _{1}\left( \Gamma ^{s,t}\right) \right\Vert \leq C_{p}\omega
\left( s,t\right) ^{\frac{\left[ p\right] +1}{p}}\text{.}
\end{equation*}

Then we prove $M_{1}\leq C_{p}$. Based on the explicit expression of $%
F\left( f\right) ^{\circ k}$ as in Lemma \ref{Lemma explicit form of Fcirck}%
, $\left( y_{s_{0},s}^{n}\right) ^{-1}\otimes y_{1}^{s,t}-1$ is the solution
of the ODE 
\begin{equation*}
dy_{r}=\sum_{k=1}^{\left[ p\right] }F\left( f\left( \cdot +\pi _{1}\left(
y_{s_{0},s}^{n}\right) +\pi _{1}\left( y_{s_{0}}^{n}\right) \right) \right)
^{\circ k}\pi _{k}\left( \log X_{s,t}\right) \left( Id_{L^{\left[ p\right]
}\left( \mathcal{U}\right) }\right) \left( y_{r}\right) dr,r\in \left[ 0,1%
\right] ,y_{0}=0\text{.}
\end{equation*}%
Based on $\left( \ref{bound on y on small interval}\right) $ in Lemma \ref%
{Lemma simple property of solution of ODE} on p\pageref{bound on y on small
interval}, 
\begin{equation}
\pi _{0}\left( \left( y_{s_{0},s}^{n}\right) ^{-1}\otimes
y_{1}^{s,t}-1\right) =0\text{ and }\left\Vert \pi _{k}\left( \left(
y_{s_{0},s}^{n}\right) ^{-1}\otimes y_{1}^{s,t}-1\right) \right\Vert \leq
C_{p}\omega \left( s,t\right) ^{\frac{k}{p}}\text{, }k=1,2,\dots ,\left[ p%
\right] \text{.}  \label{inner estimate of ode part}
\end{equation}%
As a result, we have, for any dyadic interval $\left[ s,t\right] \subseteq %
\left[ s_{0},t_{0}\right] $,%
\begin{eqnarray}
\left\Vert \pi _{1}\left( y_{s_{0},t}^{n}\right) -\pi _{1}\left(
y_{s_{0},s}^{n}\right) \right\Vert &\leq &\left\Vert \pi _{1}\left( \Gamma
^{s,t}\right) \right\Vert +\left\Vert \pi _{1}\left(
y_{1}^{s,t}-y_{s_{0},s}^{n}\right) \right\Vert
\label{inner estimate of the first level} \\
&=&\left\Vert \pi _{1}\left( \Gamma ^{s,t}\right) \right\Vert +\left\Vert
\pi _{1}\left( \left( y_{s_{0},s}^{n}\right) ^{-1}\otimes
y_{1}^{s,t}-1\right) \right\Vert  \notag \\
&\leq &C_{p}\omega \left( s,t\right) ^{\frac{\left[ p\right] +1}{p}%
}+C_{p}\omega \left( s,t\right) ^{\frac{1}{p}}\leq C_{p}\omega \left(
s,t\right) ^{\frac{1}{p}}\text{.}  \notag
\end{eqnarray}%
Then similar estimate holds for non-dyadic intervals based on Lemma \ref%
{Lemma dyadic to non-dyadic} on p\pageref{Lemma dyadic to non-dyadic}, and
we have, for any $\left\{ s,t\right\} \in \Lambda \left( n\right) $, $\left[
s,t\right] \subseteq \left[ s_{0},t_{0}\right] $, 
\begin{equation}
\left\Vert \pi _{1}\left( y_{s_{0},t}^{n}\right) -\pi _{1}\left(
y_{s_{0},s}^{n}\right) \right\Vert \leq C_{p}\omega \left( s,t\right) ^{%
\frac{1}{p}}\text{.}  \label{inner obtain non-dyadic by dyadic}
\end{equation}%
As a result,%
\begin{equation*}
M_{1}=\,\sup_{\omega \left( s_{0},s\right) \leq 1,s\in \Lambda \left(
n\right) }\,\sup_{n\geq 1}\left\Vert \pi _{1}\left( y_{s_{0},s}^{n}\right)
\right\Vert \leq C_{p}\text{.}
\end{equation*}

For $k=2,\dots ,\left[ p\right] $, assume%
\begin{equation*}
M_{k-1}\leq C_{p}\text{,}
\end{equation*}%
and assume that, for any dyadic $\left[ s,t\right] \subseteq \left[
s_{0},t_{0}\right] $,%
\begin{equation*}
\left\Vert \pi _{j}\left( \Gamma ^{s,t}\right) \right\Vert \leq C_{p}\omega
\left( s,t\right) ^{\frac{\left[ p\right] +1}{p}}\text{, }j=1,2,\dots ,k-1%
\text{.}
\end{equation*}%
Then by using the inductive relationship of $\pi _{k}\left( \Gamma
^{s,t}\right) $ as at $\left( \ref{inner induction of Gamma on level k}%
\right) $, we have%
\begin{equation*}
\left\Vert \pi _{k}\left( \Gamma ^{s,t}\right) -\pi _{k}\left( \Gamma
^{s,u}\right) -\pi _{k}\left( \Gamma ^{u,t}\right) \right\Vert \leq
C_{p}M_{k-1}\omega \left( s,t\right) ^{\frac{\left[ p\right] +1}{p}}\leq
C_{p}\omega \left( s,t\right) ^{\frac{\left[ p\right] +1}{p}}\text{.}
\end{equation*}%
Since $\left[ s,t\right] $ is dyadic, by recursively bisecting $\left[ s,t%
\right] $ and using $\Gamma ^{t_{j}^{n},t_{j+1}^{n}}=0$ as at $\left( \ref%
{inner Gamma vanishes on small dyadic interval}\right) $, we have that, for
any dyadic $\left[ s,t\right] \subseteq \left[ s_{0},t_{0}\right] $,%
\begin{equation}
\left\Vert \pi _{k}\left( \Gamma ^{s,t}\right) \right\Vert \leq C_{p}\left(
\sum_{n=0}^{\infty }2^{\left( 1-\frac{\left[ p\right] +1}{p}\right)
n}\right) \omega \left( s,t\right) ^{\frac{\left[ p\right] +1}{p}}\leq
C_{p}\omega \left( s,t\right) ^{\frac{\left[ p\right] +1}{p}}\text{.}
\label{inner Gamma}
\end{equation}%
Moreover, based on $\left( \ref{inner estimate of ode part}\right) $ and $%
\left( \ref{inner Gamma}\right) $, for any dyadic interval $\left[ s,t\right]
\subseteq \left[ s_{0},t_{0}\right] $, \ we have 
\begin{eqnarray}
\left\Vert \pi _{k}\left( y_{s_{0}t}^{n}-y_{s_{0},s}^{n}\right) \right\Vert
&\leq &\left\Vert \pi _{k}\left( y_{1}^{s,t}-y_{s_{0},s}^{n}\right)
\right\Vert +\left\Vert \pi _{k}\left( \Gamma ^{s,t}\right) \right\Vert
\label{inner estimate of higher levels} \\
&\leq &\sum_{j=0}^{k-1}\left\Vert \pi _{j}\left( y_{s_{0},s}^{n}\right)
\right\Vert \left\Vert \pi _{k-j}\left( \left( y_{s_{0},s}^{n}\right)
^{-1}\otimes y_{1}^{s,t}-1\right) \right\Vert +C_{p}\omega \left( s,t\right)
^{\frac{\left[ p\right] +1}{p}}  \notag \\
&\leq &C_{p}M_{k-1}\omega \left( s,t\right) ^{\frac{1}{p}}+C_{p}\omega
\left( s,t\right) ^{\frac{\left[ p\right] +1}{p}}\leq C_{p}\omega \left(
s,t\right) ^{\frac{1}{p}}\text{.}  \notag
\end{eqnarray}%
Then similar estimate holds for non-dyadic intervals based on Lemma \ref%
{Lemma dyadic to non-dyadic}, and we have, for any $\left\{ s,t\right\} \in
\Lambda \left( n\right) $, $\left[ s,t\right] \subseteq \left[ s_{0},t_{0}%
\right] $,%
\begin{equation*}
\left\Vert \pi _{k}\left( y_{s_{0}t}^{n}-y_{s_{0},s}^{n}\right) \right\Vert
\leq C_{p}\omega \left( s,t\right) ^{\frac{1}{p}}\text{.}
\end{equation*}%
As a result, 
\begin{equation}
M_{k}\leq M_{k-1}+\sup_{\omega \left( s_{0},s\right) \leq 1,s\in \Lambda
\left( n\right) }\sup_{n\geq 1}\left\Vert \pi _{k}\left(
y_{s_{0},s}^{n}\right) \right\Vert \leq C_{p}\text{.}
\label{inner y(n) are uniformly bounded}
\end{equation}

Based on $\left( \ref{inner Gamma}\right) $, for any dyadic $\left[ s,t%
\right] \subseteq \left[ s_{0},t_{0}\right] $, $\left\Vert \Gamma
^{s,t}\right\Vert \leq C_{p}\omega \left( s,t\right) ^{\frac{\left[ p\right]
+1}{p}}$. Similar estimate holds for non-dyadic intervals based on Lemma \ref%
{Lemma from dyadic difference to non-dyadic differene} on p\pageref{Lemma
from dyadic difference to non-dyadic differene}, and we have, for any $%
\left\{ s,t\right\} \in \Lambda \left( n\right) $, $\left[ s,t\right]
\subseteq \left[ s_{0},t_{0}\right] $, 
\begin{equation}
\left\Vert \Gamma ^{s,t}\right\Vert \leq C_{p}\omega \left( s,t\right) ^{%
\frac{\left[ p\right] +1}{p}}\text{.}
\label{inner estimate of Gamma on general small intervals}
\end{equation}

In particular, when $\left[ s,t\right] =\left[ s_{0},t_{0}\right] $, $%
y^{s,t}:\left[ 0,1\right] \rightarrow G^{\left[ p\right] }\left( \mathcal{V}%
\right) $ (defined at $\left( \ref{inner definition of y(s,t)}\right) $) is
the solution of the ODE%
\begin{equation*}
dy_{r}^{s,t}=\sum_{k=1}^{\left[ p\right] }F\left( f\left( \cdot +\pi
_{1}\left( y_{s}^{n}\right) \right) \right) ^{\circ k}\pi _{k}\left( \log
X_{s,t}\right) \left( Id_{L^{\left[ p\right] }\left( \mathcal{U}\right)
}\right) \left( y_{r}^{s,t}\right) dr\text{, }r\in \left[ 0,1\right] \text{, 
}y_{0}^{s,t}=1\text{.}
\end{equation*}%
Then based on $\left( \ref{bound on y on small interval}\right) $ in Lemma %
\ref{Lemma simple property of solution of ODE} on p\pageref{bound on y on
small interval}, we have, for $k=1,2,\dots ,\left[ p\right] $,%
\begin{equation}
\left\Vert \pi _{k}\left( y_{1}^{s,t}\right) \right\Vert \leq C_{p}\omega
\left( s,t\right) ^{\frac{k}{p}}\text{.}  \label{inner estimate of y(s,t)}
\end{equation}%
Then, based on $\left( \ref{inner estimate of Gamma on general small
intervals}\right) $ and $\left( \ref{inner estimate of y(s,t)}\right) $,
there exists constant $C_{p}$, such that, for any $n\geq 1$, any $\left\{
s,t\right\} \subseteq \Lambda \left( n\right) $, $\omega \left( s,t\right)
\leq 1$, and $k=1,2,\dots ,\left[ p\right] $, 
\begin{equation*}
\left\Vert \pi _{k}\left( y_{s,t}^{n}\right) \right\Vert \leq \left\Vert \pi
_{k}\left( \Gamma ^{s,t}\right) \right\Vert +\left\Vert \pi _{k}\left(
y_{1}^{s,t}\right) \right\Vert \leq C_{p}\omega \left( s,t\right) ^{\frac{k}{%
p}}\text{.}
\end{equation*}
\end{proof}

\begin{lemma}
\label{Lemma continuity in initial value of dyadic solutions}For $\left[ p%
\right] +1\geq \gamma >p\geq 1$, suppose $X\in C^{p-var}\left( \left[ 0,T%
\right] ,G^{\left[ p\right] }\left( \mathcal{V}\right) \right) $, $f\in
L\left( \mathcal{V},C^{\gamma }\left( \mathcal{U},\mathcal{U}\right) \right) 
$ and $\xi ^{i}\in G^{\left[ p\right] }\left( \mathcal{U}\right) $, $i=1,2$.
Define control $\omega $ by%
\begin{equation*}
\omega \left( s,t\right) :=\left\vert f\right\vert _{\gamma }^{p}\left\Vert
X\right\Vert _{p-var,\left[ s,t\right] }^{p}\text{, }\forall 0\leq s\leq
t\leq T\text{.}
\end{equation*}%
Based on $\omega $, define dyadic partition $\Lambda \left( n\right)
=\left\{ t_{j}^{n}\right\} $ as in Notation \ref{Notation dyadic partition
of a control}. For $i=1,2$, let $y^{i}:\left[ 0,T\right] \rightarrow G^{%
\left[ p\right] }\left( \mathcal{U}\right) $ be the solution of the ODE
(with different initial value)%
\begin{eqnarray}
y_{0}^{i} &=&\xi ^{i},  \label{definition of yi} \\
dy_{u}^{i} &=&\sum_{k=1}^{\left[ p\right] }F\left( f\right) ^{\circ k}\pi
_{k}\left( \log X_{t_{j}^{n},t_{j+1}^{n}}\right) \left( Id\right) \left(
y_{u}^{i}\right) \frac{du}{t_{j+1}^{n}-t_{j}^{n}},u\in \left[
t_{j}^{n},t_{j+1}^{n}\right] ,j=0,1,\dots ,2^{n}-1\text{.}  \notag
\end{eqnarray}%
Then, there exist $C_{p,\gamma }>0$ and $\delta _{p,\gamma }\in (0,1]$, such
that, for any $\left\{ s,t\right\} \in \Lambda \left( n\right) $ satisfying $%
\omega \left( s,t\right) \leq \delta _{p,\gamma }$,%
\begin{equation*}
\left\Vert y_{t}^{1}-y_{t}^{2}\right\Vert \leq C_{p,\gamma }\left\Vert
y_{s}^{1}\right\Vert \left\Vert y_{s}^{1}-y_{s}^{2}\right\Vert \text{.}
\end{equation*}
\end{lemma}

\begin{proof}
We assume $\left\vert f\right\vert _{\gamma }=1$. Otherwise, we replace $f$
and $X$ by $\left\vert f\right\vert _{\gamma }^{-1}f$ and $\delta
_{\left\vert f\right\vert _{\gamma }}X$ respectively. Fix $s_{0},t_{0}\in
\Lambda \left( n\right) $ satisfying $\omega \left( s_{0},t_{0}\right) \leq
1 $. Then based on Lemma \ref{Lemma dyadic paths are uniformly bounded},
there exists constant $C_{p}$,%
\begin{equation}
\max_{i=1,2}\sup_{s\in \wedge \left( n\right) \cap \left[ s_{0},t_{0}\right]
}\left\Vert y_{s_{0},s}^{i}\right\Vert \leq C_{p}<\infty \text{.}
\label{inner uniform bound on yni}
\end{equation}%
For $i=1,2$, $\left\{ s,u,t\right\} \subset \Lambda \left( n\right) $
satisfying $s_{0}\leq s\leq u\leq t\leq t_{0}$, let $y^{i,s,t}:\left[ 0,1%
\right] \rightarrow G^{\left[ p\right] }\left( \mathcal{U}\right) $ and $%
y^{i,s,u,t}:\left[ 0,2\right] \rightarrow G^{\left[ p\right] }\left( 
\mathcal{U}\right) $ be the solution to the ODE%
\begin{eqnarray*}
dy_{r}^{i,s,t} &=&\tsum\nolimits_{k=1}^{\left[ p\right] }F\left( f^{i}\left(
\cdot +\pi _{1}\left( y_{s_{0}}^{i}\right) \right) \right) ^{\circ k}\pi
_{k}\left( \log X_{s,t}\right) \left( Id_{L^{\left[ p\right] }\left( 
\mathcal{U}\right) }\right) \left( y_{r}^{i,s,t}\right) dr\text{, }r\in %
\left[ 0,1\right] \text{, }y_{0}^{i,s,t}=y_{s_{0},s}^{i}\text{,} \\
dy_{r}^{i,s,u,t} &=&\left\{ 
\begin{array}{cc}
\sum_{k=1}^{\left[ p\right] }F\left( f^{i}\left( \cdot +\pi _{1}\left(
y_{s_{0}}^{i}\right) \right) \right) ^{\circ k}\pi _{k}\left( \log
X_{s,u}\right) \left( Id_{L^{\left[ p\right] }\left( \mathcal{U}\right)
}\right) \left( y_{r}^{i,s,u,t}\right) dr\text{,} & r\in \left[ 0,1\right]
\\ 
\sum_{k=1}^{\left[ p\right] }F\left( f^{i}\left( \cdot +\pi _{1}\left(
y_{s_{0}}^{i}\right) \right) \right) ^{\circ k}\pi _{k}\left( \log
X_{s,u}\right) \left( Id_{L^{\left[ p\right] }\left( \mathcal{U}\right)
}\right) \left( y_{r}^{i,s,u,t}\right) dr\text{,} & r\in \left[ 1,2\right]%
\end{array}%
\right. \text{, }y_{0}^{i,s,u,t}=y_{s_{0},s}^{i}\text{.}
\end{eqnarray*}%
For $\left\{ s,t\right\} \subseteq \Lambda \left( n\right) $, $s_{0}\leq
s\leq t\leq t_{0}$, denote%
\begin{equation*}
y_{t}:=y_{s_{0},t}^{1}-y_{s_{0},t}^{2}\text{, }%
y_{1}^{s,t}:=y_{1}^{1,s,t}-y_{1}^{2,s,t}\text{ and }\Gamma
^{s,t}:=y_{t}-y_{1}^{s,t}\text{.}
\end{equation*}%
Then it can be checked that, 
\begin{equation}
\Gamma ^{t_{j}^{n},t_{j+1}^{n}}=0\text{, }j=0,1,\dots ,2^{n}-1\text{.}
\label{inner Gamma vanishes on dyadic interval in Lemma of continuity in initial value}
\end{equation}%
For $u\in \Lambda \left( n\right) $, $u\in \left( s,t\right) $, denote 
\begin{equation*}
y^{s,u,t}:=y^{1,s,u,t}-y^{2,s,u,t}\text{.}
\end{equation*}%
Then, it can be computed that,%
\begin{eqnarray}
\left\Vert \Gamma ^{s,u}+\Gamma ^{u,t}-\Gamma ^{s,t}\right\Vert
&=&\left\Vert y_{1}^{s,u}-y_{s}+y_{1}^{u,t}-y_{u}-\left(
y_{1}^{s,t}-y_{s}\right) \right\Vert
\label{inner bisecting Gamma st in Lemma of continuity in initial value} \\
&\leq &\left\Vert y_{1}^{u,t}-y_{u}-\left( y_{2}^{s,u,t}-y_{1}^{s,u}\right)
\right\Vert +\left\Vert y_{2}^{s,u,t}-y_{1}^{s,t}\right\Vert  \notag
\end{eqnarray}

Denote $\delta :=\omega \left( s,t\right) ^{\frac{1}{p}}$, then $\delta \leq
\omega \left( s_{0},t_{0}\right) ^{\frac{1}{p}}\leq 1$. Based on Lemma \ref%
{Lemma B} (on p\pageref{Lemma B}), we have%
\begin{equation}
\left\Vert y_{2}^{s,u,t}-y_{1}^{s,t}\right\Vert \leq C_{p}\delta ^{\gamma
}\left\Vert y_{s}\right\Vert \text{.}
\label{inner lemma a in Lemma of continuity in initial value}
\end{equation}%
On the other hand, we want to prove%
\begin{equation*}
\left\Vert y_{1}^{u,t}-y_{u}-\left( y_{2}^{s,u,t}-y_{1}^{s,u}\right)
\right\Vert \leq C_{p}\delta \left\Vert \Gamma _{s,u}\right\Vert
+C_{p}\delta ^{\gamma }\left\Vert y_{u}\right\Vert \text{.}
\end{equation*}%
Based on $\left( \ref{inner uniform bound on yni}\right) $, we have $%
\max_{i=1,2}\left\Vert y_{s_{0},u}^{i}\right\Vert \vee \left\Vert
y_{1}^{i,s,u}\right\Vert \leq C_{p}$. Then by using Lemma \ref{Lemma simple
but important} (on p\pageref{Lemma simple but important}) and estimating the
two cases $k=1,\dots ,\left[ p\right] -1$ and $k=\left[ p\right] $
separately, we get%
\begin{eqnarray}
&&\left\Vert y_{r}^{u,t}-y_{u}-\left( y_{r+1}^{s,u,t}-y_{1}^{s,u}\right)
\right\Vert  \label{inner estimate 1 in Lemma of continuity in initial value}
\\
&\leq &\sum\nolimits_{k=1}^{\left[ p\right] }\int_{0}^{r}\left\Vert F\left(
f\right) ^{\circ k}\pi _{k}\left( \log X_{u,t}\right) \left( Id_{L^{\left[ p%
\right] }\left( \mathcal{U}\right) }\right) \left( y_{v}^{1,u,t}\right)
-F\left( f\right) ^{\circ k}\pi _{k}\left( \log X_{u,t}\right) \left( Id_{L^{%
\left[ p\right] }\left( \mathcal{U}\right) }\right) \left(
y_{v+1}^{1,s,u,t}\right) \right.  \notag \\
&&\left. -F\left( f\right) ^{\circ k}\pi _{k}\left( \log X_{u,t}\right)
\left( Id_{L^{\left[ p\right] }\left( \mathcal{U}\right) }\right) \left(
y_{v}^{2,u,t}\right) +F\left( f\right) ^{\circ k}\pi _{k}\left( \log
X_{u,t}\right) \left( Id_{L^{\left[ p\right] }\left( \mathcal{U}\right)
}\right) \left( y_{v+1}^{2,s,u,t}\right) \right\Vert dv  \notag \\
&\leq &C_{p}\delta \int_{0}^{r}\left\Vert
y_{v}^{u,t}-y_{v+1}^{s,u,t}\right\Vert dv  \notag \\
&&+C_{p}\sup_{r\in \left[ 0,1\right] }\left( \left\Vert
y_{r}^{1,u,t}-y_{r+1}^{1,s,u,t}\right\Vert +\left\Vert
y_{r}^{2,u,t}-y_{r+1}^{2,s,u,t}\right\Vert \right) \left( \delta \sup_{r\in %
\left[ 0,1\right] }\left\Vert y_{r}^{1,u,t}-y_{r}^{2,u,t}\right\Vert \right)
\notag \\
&&+C_{p}\sup_{r\in \left[ 0,1\right] }\left( \left\Vert
y_{r}^{1,u,t}-y_{r+1}^{1,s,u,t}\right\Vert +\left\Vert
y_{r}^{2,u,t}-y_{r+1}^{2,s,u,t}\right\Vert \right) ^{\left\{ \gamma \right\}
}\left( \delta ^{\left[ p\right] }\sup_{r\in \left[ 0,1\right] }\left\Vert
y_{r}^{1,u,t}-y_{r}^{2,u,t}\right\Vert \right) \text{.}  \notag
\end{eqnarray}%
By using $\left\Vert y_{s_{0},s}^{i}\right\Vert \leq C_{p}$ as at $\left( %
\ref{inner uniform bound on yni}\right) $, and based on Lemma \ref{Lemma two
ODE with different vector fields} on p\pageref{Lemma two ODE with different
vector fields} (continuity in initial value) and $\left( \ref{estimate
uniform bound on difference on small interval}\right) $ in Lemma \ref{Lemma
dyadic paths are uniformly bounded} on p\pageref{estimate uniform bound on
difference on small interval}, we have, 
\begin{equation}
\sup_{r\in \left[ 0,1\right] }\left\Vert
y_{r}^{i,u,t}-y_{r+1}^{i,s,u,t}\right\Vert \leq C_{p}\left\Vert
y_{s_{0},u}^{i}-y_{1}^{i,s,u}\right\Vert \leq C_{p}\omega \left( s,u\right)
^{\frac{\left[ p\right] +1}{p}}\leq C_{p}\delta ^{\left[ p\right] +1}\text{.}
\label{inner estimate 2 in Lemma of continuity in initial value}
\end{equation}%
By using $\left\Vert y_{s_{0},s}^{i}\right\Vert \leq C_{p}$, and based on
Lemma \ref{Lemma two ODE with different vector fields} (continuity in
initial value),%
\begin{equation}
\sup_{r\in \left[ 0,1\right] }\left\Vert
y_{r}^{1,u,t}-y_{r}^{2,u,t}\right\Vert \leq C_{p}\left\Vert
y_{s_{0},u}^{1}-y_{s_{0},u}^{2}\right\Vert =C_{p}\left\Vert y_{u}\right\Vert 
\text{.}  \label{inner estimate 3 in Lemma of continuity in initial value}
\end{equation}%
As a result, combining $\left( \ref{inner estimate 1 in Lemma of continuity
in initial value}\right) $, $\left( \ref{inner estimate 2 in Lemma of
continuity in initial value}\right) $ and $\left( \ref{inner estimate 3 in
Lemma of continuity in initial value}\right) $, we have ($\delta \leq 1$)%
\begin{equation*}
\left\Vert y_{r}^{u,t}-y_{u}-\left( y_{r+1}^{s,u,t}-y_{1}^{s,u}\right)
\right\Vert \leq C_{p}\delta \int_{0}^{r}\left\Vert
y_{v}^{u,t}-y_{v+1}^{s,u,t}\right\Vert dv+C_{p}\delta ^{\gamma }\left\Vert
y_{u}\right\Vert \text{.}
\end{equation*}%
Then by using Gronwall's inequality ($\delta \leq 1$), we have%
\begin{equation*}
\sup_{r\in \left[ 0,1\right] }\left\Vert
y_{r}^{u,t}-y_{r+1}^{s,u,t}\right\Vert \leq C_{p}\left\Vert
y_{u}-y_{1}^{s,u}\right\Vert +C_{p}\delta ^{\gamma }\left\Vert
y_{u}\right\Vert =C_{p}\left\Vert \Gamma _{s,u}\right\Vert +C_{p}\delta
^{\gamma }\left\Vert y_{u}\right\Vert \text{,}
\end{equation*}%
and%
\begin{equation}
\sup_{r\in \left[ 0,1\right] }\left\Vert y_{r}^{u,t}-y_{u}-\left(
y_{r+1}^{s,u,t}-y_{1}^{s,u}\right) \right\Vert \leq C_{p}\delta \left\Vert
\Gamma _{s,u}\right\Vert +C_{p}\delta ^{\gamma }\left\Vert y_{u}\right\Vert 
\text{.}  \label{inner lemma b in Lemma of continuity in initial value}
\end{equation}%
As a result, combining $\left( \ref{inner bisecting Gamma st in Lemma of
continuity in initial value}\right) $, $\left( \ref{inner lemma a in Lemma
of continuity in initial value}\right) $ and $\left( \ref{inner lemma b in
Lemma of continuity in initial value}\right) $, we have, for $\left\{
s,u,t\right\} \in \Lambda \left( n\right) $, $\left[ s,t\right] \subseteq %
\left[ s_{0},t_{0}\right] $, $\omega \left( s_{0},t_{0}\right) \leq 1$,%
\begin{equation*}
\left\Vert \Gamma ^{s,u}+\Gamma ^{u,t}-\Gamma ^{s,t}\right\Vert \leq
C_{p}\left( \omega \left( s,t\right) ^{\frac{1}{p}}\left\Vert \Gamma
^{s,u}\right\Vert +\omega \left( s,t\right) ^{\frac{\gamma }{p}}\left(
\sup_{r\in \left[ s,t\right] }\left\Vert y_{r}\right\Vert \right) \right) 
\text{.}
\end{equation*}

If $\left[ s,t\right] \subseteq \left[ s_{0},t_{0}\right] $ is a dyadic
interval, then after a sequence of bisections, we have ($\Gamma
^{t_{j}^{n},t_{j+1}^{n}}=0$ as at $\left( \ref{inner Gamma vanishes on
dyadic interval in Lemma of continuity in initial value}\right) $)%
\begin{eqnarray*}
\left\Vert \Gamma ^{s,t}\right\Vert &\leq &\left( \sum_{k=0}^{n}\left(
\prod\nolimits_{j=0}^{k}\left( 1+2^{-\frac{j}{p}}\omega \left( s,t\right) ^{%
\frac{1}{p}}\right) \left( \frac{1}{2}\right) ^{\left( \frac{\gamma }{p}%
-1\right) k}\right) \right) \omega \left( s,t\right) ^{\frac{\gamma }{p}%
}\left( \sup_{r\in \left[ s,t\right] }\left\Vert y_{r}\right\Vert \right) \\
&&+\prod\nolimits_{j=0}^{n}\left( 1+2^{-\frac{j}{p}}\omega \left( s,t\right)
^{\frac{1}{p}}\right) \left( \sum_{j=0}^{2^{n}-1}\left\Vert \Gamma
^{t_{j}^{n},t_{j+1}^{n}}\right\Vert \right) \\
&\leq &\exp \left( \frac{2^{\frac{1}{p}}}{2^{\frac{1}{p}}-1}\right) \frac{2^{%
\frac{\gamma }{p}-1}}{2^{\frac{\gamma }{p}-1}-1}\omega \left( s,t\right) ^{%
\frac{\gamma }{p}}\left( \sup_{r\in \left[ s,t\right] }\left\Vert
y_{r}\right\Vert \right) \text{.}
\end{eqnarray*}%
As a result, when $\left[ s,t\right] \subseteq \left[ s_{0},t_{0}\right] $
is a dyadic interval, we have%
\begin{equation*}
\left\Vert y_{t}-y_{1}^{s,t}\right\Vert \leq C_{p,\gamma }\omega \left(
s,t\right) ^{\frac{\gamma }{p}}\sup_{r\in \left[ s,t\right] }\left\Vert
y_{r}\right\Vert \text{.}
\end{equation*}

Based on Lemma \ref{Lemma from dyadic difference to non-dyadic differene} on
p\pageref{Lemma from dyadic difference to non-dyadic differene}, we can pass
similar estimate to non-dyadic intervals and get, for any $\left\{
s,t\right\} \in \Lambda \left( n\right) $, $\left[ s,t\right] \subseteq %
\left[ s_{0},t_{0}\right] $, 
\begin{equation}
\left\Vert y_{t}-y_{1}^{s,t}\right\Vert \leq C_{p,\gamma }\omega \left(
s,t\right) ^{\frac{\gamma }{p}}\sup_{r\in \left[ s,t\right] }\left\Vert
y_{r}\right\Vert \text{.}  \label{inner estimate Gamma on dyadic interval}
\end{equation}

Let $\left[ s,t\right] =\left[ s_{0},t_{0}\right] $, then $y^{i,s,t}$ is the
solution to the ODE:%
\begin{equation*}
dy_{r}^{i,s,t}=\sum_{k=1}^{\left[ p\right] }F\left( f\left( \cdot +\pi
_{1}\left( y_{s}^{i}\right) \right) \right) ^{\circ k}\pi _{k}\left( \log
X_{s,t}\right) \left( Id_{L^{\left[ p\right] }\left( \mathcal{U}\right)
}\right) \left( y_{r}^{i,s,t}\right) dr\text{, }r\in \left[ 0,1\right] \text{%
, }y_{0}^{i,s,t}=1\text{.}
\end{equation*}%
Based on Lemma \ref{Lemma two ODE with different vector fields} (on p\pageref%
{Lemma two ODE with different vector fields}), for $k=1,2,\dots ,\left[ p%
\right] $,%
\begin{equation}
\left\Vert \pi _{k}\left( y_{1}^{s,t}\right) \right\Vert =\left\Vert \pi
_{k}\left( y_{1}^{1,s,t}\right) -\pi _{k}\left( y_{1}^{2,s,t}\right)
\right\Vert \leq C_{p}\omega \left( s,t\right) ^{\frac{k}{p}}\left\Vert \pi
_{1}\left( y_{s}^{1}\right) -\pi _{1}\left( y_{s}^{2}\right) \right\Vert 
\text{.}  \label{inner estimate of ys,t on dyadic interval}
\end{equation}%
Combining $\left( \ref{inner estimate Gamma on dyadic interval}\right) $ and 
$\left( \ref{inner estimate of ys,t on dyadic interval}\right) $, we get,
for $k=1,2,\dots ,\left[ p\right] $,%
\begin{eqnarray}
\left\Vert \pi _{k}\left( y_{s,t}^{1}\right) -\pi _{k}\left(
y_{s,t}^{2}\right) \right\Vert &=&\left\Vert \pi _{k}\left( y_{t}\right)
\right\Vert
\label{inner estimate on difference of solutions with different initial value}
\\
&\leq &C_{p}\omega \left( s,t\right) ^{\frac{k}{p}}\left\Vert \pi _{1}\left(
y_{s}^{1}\right) -\pi _{1}\left( y_{s}^{2}\right) \right\Vert +C_{p,\gamma
}\omega \left( s,t\right) ^{\frac{\gamma }{p}}\left( \sup_{u\in \left[ s,t%
\right] }\left\Vert y_{s,u}^{1}-y_{s,u}^{2}\right\Vert \right) \text{.} 
\notag
\end{eqnarray}%
In particular, we have 
\begin{equation*}
\sup_{u\in \left[ s,t\right] }\left\Vert y_{s,u}^{1}-y_{s,u}^{2}\right\Vert
\leq C_{p}\omega \left( s,t\right) ^{\frac{1}{p}}\left\Vert \pi _{1}\left(
y_{s}^{1}\right) -\pi _{1}\left( y_{s}^{2}\right) \right\Vert +C_{p,\gamma
}\omega \left( s,t\right) ^{\frac{\gamma }{p}}\left( \sup_{u\in \left[ s,t%
\right] }\left\Vert y_{s,u}^{1}-y_{s,u}^{2}\right\Vert \right) \text{.}
\end{equation*}%
Choose $\delta _{p,\gamma }\in (0,1]$ satisfying $C_{p,\gamma }\delta
_{p,\gamma }^{\frac{\gamma }{p}}\leq 2^{-1}$. Then for $\left[ s,t\right] $
satisfying $\omega \left( s,t\right) \leq \delta _{p,\gamma }$, we have 
\begin{equation*}
\sup_{u\in \left[ s,t\right] }\left\Vert y_{s,u}^{1}-y_{s,u}^{2}\right\Vert
\leq C_{p}\omega \left( s,t\right) ^{\frac{1}{p}}\left\Vert \pi _{1}\left(
y_{s}^{1}\right) -\pi _{1}\left( y_{s}^{2}\right) \right\Vert \text{.}
\end{equation*}%
Combined with $\left( \ref{inner estimate on difference of solutions with
different initial value}\right) $,%
\begin{equation}
\left\Vert \pi _{k}\left( y_{s,t}^{1}\right) -\pi _{k}\left(
y_{s,t}^{2}\right) \right\Vert \leq C_{p,\gamma }\omega \left( s,t\right) ^{%
\frac{k}{p}}\left\Vert \pi _{1}\left( y_{s}^{1}\right) -\pi _{1}\left(
y_{s}^{2}\right) \right\Vert \text{.}
\label{inner estimate  interval is less than 1/2}
\end{equation}

As a result, by combining $\left( \ref{inner estimate interval is less than
1/2}\right) $ with $\left\Vert \pi _{j}\left( y_{s,t}^{i}\right) \right\Vert
\leq C_{p}\omega \left( s,t\right) ^{\frac{j}{p}}\leq C_{p}$ as at $\left( %
\ref{estimate uniform bound on dyadic solution on small interval}\right) $
(on p\pageref{estimate uniform bound on dyadic solution on small interval}),
there exists $\delta _{p,\gamma }\in (0,1]$ such that for any $\left[ s,t%
\right] $ satisfying $\omega \left( s,t\right) \leq \delta _{p,\gamma }$, we
have ($\left\Vert y_{s}^{1}\right\Vert \geq \left\Vert \pi _{0}\left(
y_{s}^{1}\right) \right\Vert =1$) 
\begin{eqnarray*}
&&\left\Vert \pi _{k}\left( y_{t}^{1}\right) -\pi _{k}\left(
y_{t}^{2}\right) \right\Vert \\
&\leq &\sum_{j=1}^{k}\left\Vert \pi _{k-j}\left( y_{s}^{1}\right)
\right\Vert \left\Vert \pi _{j}\left( y_{s,t}^{1}\right) -\pi _{j}\left(
y_{s,t}^{2}\right) \right\Vert +\sum_{j=1}^{k}\left\Vert \pi _{j}\left(
y_{s}^{1}\right) -\pi _{j}\left( y_{s}^{2}\right) \right\Vert \left\Vert \pi
_{k-j}\left( y_{s,t}^{2}\right) \right\Vert \\
&\leq &C_{p,\gamma }\left\Vert y_{s}^{1}\right\Vert \left\Vert
y_{s}^{1}-y_{s}^{2}\right\Vert \text{.}
\end{eqnarray*}
\end{proof}

Then we re-state Theorem \ref{Theorem existence and uniqueness}, and give a
proof.

\bigskip

\noindent \textbf{Theorem\ \ref{Theorem existence and uniqueness}} \ \textit{%
For }$\gamma >p\geq 1$\textit{, suppose }$X\in C^{p-var}\left( \left[ 0,T%
\right] ,G^{\left[ p\right] }\left( \mathcal{V}\right) \right) $\textit{, }$%
f\in L\left( \mathcal{V},C^{\gamma }\left( \mathcal{U},\mathcal{U}\right)
\right) $\textit{\ and }$\xi \in G^{\left[ p\right] }\left( \mathcal{U}%
\right) $\textit{. Then the rough differential equation}%
\begin{equation}
dY=f\left( Y\right) dX\text{, }Y_{0}=\xi \text{,}
\label{RDE in proof of main theorem}
\end{equation}%
\textit{has a unique solution (denoted as }$Y$\textit{) in the sense of
Definition \ref{Definition of solution of RDE}, which is a continuous path
taking values in }$G^{\left[ p\right] }\left( \mathcal{U}\right) $\textit{.
If define control }$\omega $\textit{\ by}%
\begin{equation*}
\omega \left( s,t\right) :=\left\vert f\right\vert _{\gamma }^{p}\left\Vert
X\right\Vert _{p-var,\left[ s,t\right] }^{p}\text{, }\forall 0\leq s\leq
t\leq T\text{,}
\end{equation*}%
\textit{then there exists a constant }$C_{p}$\textit{\ such that, for any }$%
0\leq s\leq t\leq T$\textit{,}%
\begin{equation}
\left\Vert Y\right\Vert _{p-var,\left[ s,t\right] }\leq C_{p}\left( \omega
\left( s,t\right) ^{\frac{1}{p}}\vee \omega \left( s,t\right) \right) \text{.%
}  \label{bound of Y in proof of main theorem}
\end{equation}%
\textit{Moreover, for }$0\leq s\leq t\leq T$\textit{, if let }$y^{s,t}:\left[
0,1\right] \rightarrow L^{\left[ p\right] }\left( \mathcal{U}\right) $%
\textit{\ denote the solution of the ordinary differential equation}%
\begin{equation}
dy_{u}^{s,t}=\tsum\nolimits_{k=1}^{\left[ p\right] }F\left( f\left( \cdot
+\pi _{1}\left( Y_{s}\right) \right) \right) ^{\circ k}\pi _{k}\left( \log
X_{s,t}\right) \left( Id_{L^{\left[ p\right] }\left( \mathcal{U}\right)
}\right) \left( y_{u}^{s,t}\right) du\text{, }u\in \left[ 0,1\right] \text{, 
}y_{0}^{s,t}=1\text{,}  \label{definition of ys,t in proof of main theorem}
\end{equation}%
\textit{then }$y^{s,t}$\textit{\ takes value in }$G^{\left[ p\right] }\left( 
\mathcal{U}\right) $\textit{, and there exists a constant }$C_{p}$\textit{,
such that,}%
\begin{gather}
\left\Vert Y_{s,t}-y_{1}^{s,t}\right\Vert \leq C_{p}\left( \omega \left(
s,t\right) ^{\frac{\left[ p\right] +1}{p}}\vee \omega \left( s,t\right) ^{%
\left[ p\right] }\right) \text{,}
\label{estimate of difference in proof of main theorem} \\
\left\Vert Y_{s,t}-\tsum\nolimits_{k=1}^{\left[ p\right] }F\left( f\left(
\cdot +\pi _{1}\left( Y_{s}\right) \right) \right) ^{\circ k}\pi _{k}\left(
X_{s,t}\right) \left( Id_{L^{\left[ p\right] }\left( \mathcal{U}\right)
}\right) \left( 1\right) \right\Vert \leq C_{p}\left( \omega \left(
s,t\right) ^{\frac{\left[ p\right] +1}{p}}\vee \omega \left( s,t\right) ^{%
\left[ p\right] }\right) \text{.}  \notag
\end{gather}

\bigskip

\medskip

\begin{proof}
\label{Proof of Theorem existence and uniqueness}Firstly, with $\delta
_{p,\gamma }\in (0,1]$ selected in Lemma \ref{Lemma continuity in initial
value of dyadic solutions}, we assume $\omega \left( 0,T\right) \leq \delta
_{p,\gamma }$ and prove existence and uniqueness. Denote dyadic partitions $%
\Lambda \left( n\right) =\left\{ t_{j}^{n}\right\} $ of $\omega $ as in
Notation \ref{Notation dyadic partition of a control} on p\pageref{Notation
dyadic partition of a control}. Let $y^{n}:\left[ 0,T\right] \rightarrow G^{%
\left[ p\right] }\left( \mathcal{U}\right) $ be the solution to the ODE%
\begin{equation*}
dy_{u}^{n}=\sum_{k=1}^{\left[ p\right] }F\left( f\right) ^{\circ k}\pi
_{k}\left( \log X_{t_{j}^{n},t_{j+1}^{n}}\right) \left( Id_{L^{\left[ p%
\right] }\left( \mathcal{U}\right) }\right) \left( y_{u}^{n}\right) \frac{dt%
}{t_{j+1}^{n}-t_{j}^{n}}\text{, }t\in \left[ t_{j}^{n},t_{j+1}^{n}\right] 
\text{, }y_{0}^{n}=\xi \text{.}
\end{equation*}%
For $\left[ s,t\right] \subseteq \left[ 0,T\right] $, let $y^{n,s,t}:\left[
0,1\right] \rightarrow G^{\left[ p\right] }\left( \mathcal{U}\right) $ be
the solution of the ODE%
\begin{equation*}
dy_{u}^{n,s,t}=\sum_{k=1}^{\left[ p\right] }F\left( f\left( \cdot +\pi
_{1}\left( y_{s}^{n}\right) \right) \right) ^{\circ k}\pi _{k}\left( \log
X_{s,t}\right) \left( Id_{L^{\left[ p\right] }\left( \mathcal{U}\right)
}\right) \left( y_{u}^{n,s,t}\right) du\text{, }u\in \left[ 0,1\right] \text{%
, }y_{0}^{n,s,t}=1\text{.}
\end{equation*}%
Based Lemma \ref{Lemma dyadic paths are uniformly bounded} (uniform bound on
concatenated dyadic ODEs), we have, ($\omega \left( 0,T\right) \leq \delta
_{p,\gamma }$) 
\begin{equation}
\sup_{n\geq 1}\sup_{u\in \left[ 0,T\right] }\left\Vert y_{u}^{n}\right\Vert
\leq C_{p}\text{. }  \label{inner uniform bound on ynu}
\end{equation}

Suppose $m\geq n\geq 1$. For $j=0,1,\dots ,2^{n}$, as in the proof of Thm
2.3\ by Davie \cite{AMDavie}, we let $Z^{j}$ be the solution of the ODE (the
ODE approximation w.r.t. $\Lambda \left( m\right) $ starting at time $%
t_{j}^{n}$ from point $y_{t_{j}^{n}}^{n}$)%
\begin{equation*}
dZ_{u}^{j}=\sum_{k=1}^{\left[ p\right] }F\left( f\right) ^{\circ k}\pi
_{k}\left( \log X_{t_{l}^{m},t_{l+1}^{m}}\right) \left( Id_{L^{\left[ p%
\right] }\left( \mathcal{U}\right) }\right) \left( Z_{u}^{j}\right) \frac{dt%
}{t_{l+1}^{m}-t_{l}^{m}}\text{, }t\in \left[ t_{l}^{m},t_{l+1}^{m}\right] 
\text{, }l\geq j\text{, }Z_{t_{j}^{n}}^{j}=y_{t_{j}^{n}}^{n}\text{.}
\end{equation*}%
Then $Z_{t_{j}^{n}}^{0}=y_{t_{j}^{n}}^{m}$ and $%
Z_{t_{j}^{n}}^{j}=y_{t_{j}^{n}}^{n}$. Moreover, based on $\left( \ref%
{estimate uniform bound on difference on small interval}\right) $ in Lemma %
\ref{Lemma dyadic paths are uniformly bounded} on p\pageref{estimate uniform
bound on difference on small interval}, for $j=0,1,\dots ,2^{n}-1$,%
\begin{equation*}
\left\Vert
Z_{t_{j}^{n},t_{j+1}^{n}}^{j}-y_{1}^{n,t_{j}^{n},t_{j+1}^{n}}\right\Vert
\leq C_{p}\omega \left( t_{j}^{n},t_{j+1}^{n}\right) ^{\frac{\left[ p\right]
+1}{p}}\text{.}
\end{equation*}%
Combined with Lemma \ref{Lemma continuity in initial value of dyadic
solutions} (continuity in initial value of dyadic ODE approximations) and
using $\left( \ref{inner uniform bound on ynu}\right) $, we have, for $%
k=0,1,\dots ,2^{n}$, 
\begin{eqnarray*}
\left\Vert y_{t_{k}^{n}}^{m}-y_{t_{k}^{n}}^{n}\right\Vert &\leq
&\sum_{j=0}^{k-1}\left\Vert Z_{t_{k}^{n}}^{j+1}-Z_{t_{k}^{n}}^{j}\right\Vert
\leq C_{p,\gamma }\sum_{j=0}^{k-1}\left\Vert
Z_{t_{j+1}^{n}}^{j+1}-Z_{t_{j+1}^{n}}^{j}\right\Vert =C_{p,\gamma
}\sum_{j=0}^{k-1}\left\Vert y_{t_{j}^{n}}^{n}\otimes \left(
y_{1}^{n,t_{j}^{n},t_{j+1}^{n}}-Z_{t_{j}^{n},t_{j+1}^{n}}^{j}\right)
\right\Vert \\
&\leq &C_{p,\gamma }\sum_{j=0}^{2^{n}-1}\omega \left(
t_{j}^{n},t_{j+1}^{n}\right) ^{\frac{\left[ p\right] +1}{p}}\leq C_{p,\gamma
}\left( 2^{n\left( 1-\frac{\left[ p\right] +1}{p}\right) }\right) \omega
\left( 0,T\right) ^{\frac{\left[ p\right] +1}{p}}\rightarrow 0\text{ as }%
n\rightarrow \infty \text{.}
\end{eqnarray*}%
Then for any fixed $n$, $y^{m}$ converge uniformly on $\Lambda \left(
n\right) $ as $m\rightarrow \infty $. Denote the limit as $Y$, which is
densely defined. Combined with Lemma \ref{Lemma dyadic paths are uniformly
bounded} on p\pageref{Lemma dyadic paths are uniformly bounded}, $Y$ extends
to a continuous path, and for any $\left[ s,t\right] \subseteq \left[ 0,T%
\right] $, $\omega \left( s,t\right) \leq 1$,%
\begin{equation*}
\left\vert \!\left\vert \!\left\vert Y_{s,t}\right\vert \!\right\vert
\!\right\vert \leq C_{p}\omega \left( s,t\right) ^{\frac{1}{p}}\text{ and }%
\left\Vert Y_{s,t}-y_{1}^{s,t}\right\Vert \leq C_{p}\omega \left( s,t\right)
^{\frac{\left[ p\right] +1}{p}}\text{,}
\end{equation*}%
where $y^{s,t}:\left[ 0,1\right] \rightarrow G^{\left[ p\right] }\left( 
\mathcal{U}\right) $ is the solution of the ODE%
\begin{equation*}
dy_{u}^{s,t}=\sum_{k=1}^{\left[ p\right] }F\left( f\left( \cdot +\pi
_{1}\left( Y_{s}\right) \right) \right) ^{\circ k}\pi _{k}\left( \log
X_{s,t}\right) \left( Id_{L^{\left[ p\right] }\left( \mathcal{U}\right)
}\right) \left( y_{u}^{s,t}\right) du\text{, }u\in \left[ 0,1\right] \text{, 
}y_{0}^{s,t}=1\text{.}
\end{equation*}%
Combined with the Euler expansion of $y_{1}^{s,t}$ in Lemma \ref{Lemma Euler
expansion of solution of ODE and difference between one step and two steps}
on p\pageref{Lemma Euler expansion of solution of ODE and difference between
one step and two steps}, we have%
\begin{equation*}
\left\Vert Y_{s,t}-\tsum\nolimits_{k=1}^{\left[ p\right] }F\left( f\left(
\cdot +\pi _{1}\left( Y_{s}\right) \right) \right) ^{\circ k}\pi _{k}\left(
X_{s,t}\right) \left( Id_{L^{\left[ p\right] }\left( \mathcal{U}\right)
}\right) \left( 1\right) \right\Vert \leq C_{p}\omega \left( s,t\right) ^{%
\frac{\left[ p\right] +1}{p}}\text{.}
\end{equation*}%
Thus, based on the definition of solution (Definition \ref{Definition of
solution of RDE}), $Y$ is a solution to the RDE\ $\left( \ref{RDE in proof
of main theorem}\right) $. Based on\ Lemma \ref{Lemma simple property of
solution of ODE} on p\pageref{Lemma simple property of solution of ODE}, $%
y^{n}$ takes value in $G^{\left[ p\right] }\left( \mathcal{U}\right) $, so $%
Y $ takes value in $G^{\left[ p\right] }\left( \mathcal{U}\right) $.

Then we prove that the solution is unique. Suppose $\widetilde{Y}$ is
another solution. Then, by combining the definition of solution and the
Euler expansion of ODE as in Lemma \ref{Lemma Euler expansion of solution of
ODE and difference between one step and two steps} on p\pageref{Lemma Euler
expansion of solution of ODE and difference between one step and two steps},
for some $\theta :\left\{ 0\leq s\leq t\leq T\right\} \rightarrow \overline{%
\mathbb{R}
^{+}}$ satisfying%
\begin{equation*}
\lim_{\left\vert D\right\vert \rightarrow 0}\tsum\nolimits_{j,t_{j}\in
D}\theta \left( t_{j},t_{j+1}\right) =0\text{,}
\end{equation*}%
we have, for all sufficiently small $\left[ s,t\right] \subseteq \left[ 0,T%
\right] $,%
\begin{equation}
\left\Vert \widetilde{Y}_{s,t}-\widetilde{y}_{1}^{s,t}\right\Vert \leq
C_{p}\omega \left( s,t\right) ^{\frac{\left[ p\right] +1}{p}}+\theta \left(
s,t\right) \text{,}  \label{inner bound on increment of Ytilde}
\end{equation}%
where $\widetilde{y}^{s,t}:\left[ 0,1\right] \rightarrow L^{\left[ p\right]
}\left( \mathcal{U}\right) $ is the solution of the ODE%
\begin{equation*}
d\widetilde{y}_{u}^{s,t}=\sum_{k=1}^{\left[ p\right] }F\left( f\left( \cdot
+\pi _{1}\left( \widetilde{Y}_{s}\right) \right) \right) ^{\circ k}\pi
_{k}\left( \log X_{s,t}\right) \left( Id_{L^{\left[ p\right] }\left( 
\mathcal{U}\right) }\right) \left( \widetilde{y}_{u}^{s,t}\right) du\text{, }%
u\in \left[ 0,1\right] \text{, }\widetilde{y}_{0}^{s,t}=1\text{.}
\end{equation*}%
For integer $n\geq 1$ and $j=0,1,\dots ,2^{n}$, denote $\widetilde{Z}^{j}$
as the solution of the ODE (the ODE approximation w.r.t. $\Lambda \left(
n\right) $ starting at time $t_{j}^{n}$ from point $\widetilde{Y}%
_{t_{j}^{n}} $),%
\begin{equation*}
d\widetilde{Z}_{u}^{j}=\sum_{k=1}^{\left[ p\right] }F\left( f\right) ^{\circ
k}\pi _{k}\left( \log X_{t_{l}^{n},t_{l+1}^{n}}\right) \left( Id_{L^{\left[ p%
\right] }\left( \mathcal{U}\right) }\right) \left( \widetilde{Z}%
_{u}^{j}\right) \frac{dt}{t_{l+1}^{n}-t_{l}^{n}}\text{, }t\in \left[
t_{l}^{n},t_{l+1}^{n}\right] \text{, }l\geq j\text{, }\widetilde{Z}%
_{t_{j}^{n}}^{j}=\widetilde{Y}_{t_{j}^{n}}\text{.}
\end{equation*}%
Then $\widetilde{Z}_{t_{j}^{n}}^{0}=y_{t_{j}^{n}}^{n}$ and $\widetilde{Z}%
_{t_{j}^{n}}^{j}=\widetilde{Y}_{t_{j}^{n}}$. Based on $\left( \ref{inner
bound on increment of Ytilde}\right) $ and that $\omega \left( 0,T\right)
\leq \delta _{p,\gamma }\leq 1$, $\widetilde{Y}$ is bounded. Then by using
Lemma \ref{Lemma continuity in initial value of dyadic solutions}
(continuity in initial value of dyadic ODE approximations) and $\left( \ref%
{inner bound on increment of Ytilde}\right) $, we have, for $k=0,1,\dots
,2^{n}$, 
\begin{eqnarray*}
\left\Vert \widetilde{Y}_{t_{k}^{n}}-y_{t_{k}^{n}}^{n}\right\Vert &\leq
&\sum_{j=0}^{k-1}\left\Vert \widetilde{Z}_{t_{k}^{n}}^{j+1}-\widetilde{Z}%
_{t_{k}^{n}}^{j}\right\Vert \leq C_{p,\gamma }\sum_{j=0}^{k-1}\left\Vert 
\widetilde{Y}_{t_{j}^{n}}\otimes \left( \widetilde{Y}%
_{t_{j}^{n},t_{j+1}^{n}}-\widetilde{y}_{1}^{t_{j}^{n},t_{j+1}^{n}}\right)
\right\Vert \\
&\leq &C_{p,\gamma ,\omega ,\theta }\sum_{j=0}^{2^{n}-1}\left( \omega \left(
t_{j}^{n},t_{j+1}^{n}\right) ^{\frac{\left[ p\right] +1}{p}}+\theta \left(
t_{j}^{n},t_{j+1}^{n}\right) \right) \rightarrow 0\text{ as }n\rightarrow
\infty \text{.}
\end{eqnarray*}%
As a result, $\widetilde{Y}$ is the uniform limit of $y^{n}$, so coincides
with $Y$.

Then we prove $\left( \ref{bound of Y in proof of main theorem}\right) $.
Since $Y$ is the uniform limit of $y^{n}$, based on $\left( \ref{estimate
uniform bound on dyadic solution on small interval}\right) $ in Lemma \ref%
{Lemma dyadic paths are uniformly bounded} on p\pageref{estimate uniform
bound on dyadic solution on small interval}, for $\left[ s,t\right]
\subseteq \left[ 0,T\right] $ satisfying $\omega \left( s,t\right) \leq 1$,
we have%
\begin{equation*}
\left\vert \!\left\vert \!\left\vert Y_{s,t}\right\vert \!\right\vert
\!\right\vert \leq \omega \left( s,t\right) ^{\frac{1}{p}}\text{, so by
using sub-additivity of }\omega \text{, }\left\Vert Y\right\Vert _{p-var,%
\left[ s,t\right] }\leq \omega \left( s,t\right) ^{\frac{1}{p}}\text{.}
\end{equation*}%
When $\omega \left( s,t\right) \geq 1$, as in Prop 5.10 \cite%
{FrizVictoirbook}, we decompose $\left[ s,t\right] =\cup _{j=0}^{n-1}\left[
t_{j},t_{j+1}\right] $ with $\omega \left( t_{j},t_{j+1}\right) =1$, $%
j=0,\dots ,n-2$, $\omega \left( t_{n-1},t_{n}\right) \leq 1$. Then%
\begin{equation*}
n=1+\sum_{j=0}^{n-2}\omega \left( t_{j},t_{j+1}\right) \leq 2\omega \left(
s,t\right) \text{,}
\end{equation*}%
and (since $\left\vert \!\left\vert \!\left\vert \cdot \right\vert
\!\right\vert \!\right\vert $ is equivalent to an additive norm upto a
constant depending on $p$, Exer 7.38 \cite{FrizVictoirbook}) we have 
\begin{equation}
\left\Vert Y\right\Vert _{p-var,\left[ s,t\right] }\leq
C_{p}\sum_{j=0}^{n-1}\left\Vert Y\right\Vert _{p-var,\left[ t_{j},t_{j+1}%
\right] }\leq C_{p}n\leq C_{p}\omega \left( s,t\right) \text{.}
\label{inner estimate on big interval based on small intervals}
\end{equation}%
As a result, we have%
\begin{equation*}
\left\Vert Y\right\Vert _{p-var,\left[ s,t\right] }\leq C_{p}\left( \omega
\left( s,t\right) ^{\frac{1}{p}}\vee \omega \left( s,t\right) \right) \text{.%
}
\end{equation*}

Then we prove $\left( \ref{estimate of difference in proof of main theorem}%
\right) $. When $\omega \left( s,t\right) \leq 1$, by using Lemma \ref{Lemma
dyadic paths are uniformly bounded} (uniform estimate of dyadic
approximations) and that $y^{n}$ converge uniformly to $Y$, we have, (with $%
y^{s,t}$ defined at $\left( \ref{definition of ys,t in proof of main theorem}%
\right) $) 
\begin{equation*}
\left\Vert Y_{s,t}-y_{1}^{s,t}\right\Vert \leq C_{p}\omega \left( s,t\right)
^{\frac{\left[ p\right] +1}{p}}\text{.}
\end{equation*}%
Combined with Lemma \ref{Lemma Euler expansion of solution of ODE and
difference between one step and two steps} (high order Euler expansion of
solution of ODE), we have%
\begin{equation*}
\left\Vert Y_{s,t}-\tsum\nolimits_{k=1}^{\left[ p\right] }F\left( f\left(
\cdot +\pi _{1}\left( Y_{s}\right) \right) \right) ^{\circ k}\pi _{k}\left(
X_{s,t}\right) \left( Id_{L^{\left[ p\right] }\left( \mathcal{U}\right)
}\right) \left( 1\right) \right\Vert \leq C_{p}\omega \left( s,t\right) ^{%
\frac{\left[ p\right] +1}{p}}\text{.}
\end{equation*}%
When $\omega \left( s,t\right) \geq 1$, based on $\left( \ref{bound of Y in
proof of main theorem}\right) $, 
\begin{equation*}
\left\Vert Y_{s,t}\right\Vert =\tsum\nolimits_{k=1}^{\left[ p\right]
}\left\Vert \pi _{k}\left( Y_{s,t}\right) \right\Vert \leq \omega \left(
s,t\right) ^{\left[ p\right] }\text{.}
\end{equation*}%
On the other hand, based on Lemma \ref{Lemma explicit form of Fcirck}
(explicit expression of $F\left( f\right) ^{\circ k}$), it can be proved
inductively that, 
\begin{equation*}
\sup_{u\in \left[ 0,1\right] }\left\Vert \pi _{k}\left( y_{u}^{s,t}\right)
\right\Vert \leq C_{p}\omega \left( s,t\right) ^{\frac{k\left[ p\right] }{p}}%
\text{, }k=1,2,\dots ,\left[ p\right] \text{, so }\left\Vert
y_{1}^{s,t}\right\Vert \leq C_{p}\omega \left( s,t\right) ^{\frac{\left[ p%
\right] ^{2}}{p}}\leq C_{p}\omega \left( s,t\right) ^{\left[ p\right] }\text{%
.}
\end{equation*}%
Then%
\begin{equation*}
\left\Vert Y_{s,t}-y_{1}^{s,t}\right\Vert \leq \left\Vert Y_{s,t}\right\Vert
+\left\Vert y_{1}^{s,t}\right\Vert \leq C_{p}\omega \left( s,t\right) ^{ 
\left[ p\right] }\text{.}
\end{equation*}%
For high order Euler expansion, when $\omega \left( s,t\right) \geq 1$, 
\begin{equation*}
\left\Vert Y_{s,t}-\tsum\nolimits_{k=1}^{\left[ p\right] }F\left( f\left(
\cdot +\pi _{1}\left( Y_{s}\right) \right) \right) ^{\circ k}\pi _{k}\left(
X_{s,t}\right) \left( Id_{L^{\left[ p\right] }\left( \mathcal{U}\right)
}\right) \left( 1\right) \right\Vert \leq C_{p}\omega \left( s,t\right) ^{%
\left[ p\right] }+C_{p}\omega \left( s,t\right) ^{\frac{\left[ p\right] }{p}%
}\leq C_{p}\omega \left( s,t\right) ^{\left[ p\right] }\text{.}
\end{equation*}
\end{proof}

\subsection{Continuity of solution in initial value, vector field and
driving noise}

\begin{proof}[Proof of Theorem \protect\ref{Theorem continuity}]
\label{Proof of Theorem continuity}We assume $\left\vert f^{i}\right\vert
_{\gamma }=1$, $i=1,2$. Otherwise, we replace $f^{i}$ and $X^{i}$ by $%
\left\vert f^{i}\right\vert _{\gamma }^{-1}f^{i}$ and $\delta _{\left\vert
f^{i}\right\vert _{\gamma }}X^{i}$, and the solution $Y^{i}$ will stay
unchanged based on the definition of solution of RDE (in Definition \ref%
{Definition of solution of RDE}). Replace $\gamma $ by $\gamma \wedge \left( %
\left[ p\right] +1\right) $, so $\left[ p\right] +1\geq \gamma >p$.

Fix interval $\left[ s_{0},t_{0}\right] $ satisfying $\omega \left(
s_{0},t_{0}\right) \leq 1$. For $i=1,2$ and $\left[ s,t\right] \subseteq %
\left[ s_{0},t_{0}\right] $, let $y^{i,s,t}:\left[ 0,1\right] \rightarrow G^{%
\left[ p\right] }\left( \mathcal{U}\right) $ be the solution of the ODE%
\begin{equation}
dy_{r}^{i,s,t}=\tsum\nolimits_{k=1}^{\left[ p\right] }F\left( f^{i}\left(
\cdot +\pi _{1}\left( Y_{s_{0}}^{i}\right) \right) \right) ^{\circ k}\pi
_{k}\left( \log \left( X_{s,t}\right) \right) \left( Id_{L^{\left[ p\right]
}\left( \mathcal{U}\right) }\right) \left( y_{r}^{i,s,t}\right) dr\text{, }%
r\in \left[ 0,1\right] \text{, }y_{0}^{i,s,t}=Y_{s_{0},s}^{i}\text{.}
\label{inner definition of y(i,s,t)}
\end{equation}%
Denote%
\begin{gather}
Y_{t}:=Y_{s_{0},t}^{1}-Y_{s_{0},t}^{2}\text{, }%
y_{1}^{s,t}:=y_{1}^{1,s,t}-y_{1}^{2,s,t}\text{,}  \notag \\
\Gamma
^{i,s,t}:=Y_{s_{0},t}^{i}-y_{1}^{i,s,t}=Y_{s_{0},t}^{i}-Y_{s_{0},s}^{i}-%
\left( y_{1}^{i,s,t}-Y_{s_{0},s}^{i}\right) \text{,}
\label{inner notation of Gamma(i,s,t)} \\
\Gamma ^{s,t}:=\Gamma ^{1,s,t}-\Gamma ^{2,s,t}=Y_{t}-y_{1}^{s,t}\text{.} 
\notag
\end{gather}

For $s_{0}\leq s\leq u\leq t\leq t_{0}$, let $y^{i,s,u,t}:\left[ 0,2\right]
\rightarrow G^{\left[ p\right] }\left( \mathcal{U}\right) $ be the solution
of the ODE%
\begin{equation}
dy_{r}^{i,s,u,t}=\left\{ 
\begin{array}{cc}
\sum_{k=1}^{\left[ p\right] }F\left( f^{i}\left( \cdot +\pi _{1}\left(
Y_{s_{0}}^{i}\right) \right) \right) ^{\circ k}\pi _{k}\left( \log
X_{s,u}^{i}\right) \left( Id_{L^{\left[ p\right] }\left( \mathcal{U}\right)
}\right) \left( y_{r}^{i,s,u,t}\right) dr, & r\in \left[ 0,1\right] \\ 
\sum_{k=1}^{\left[ p\right] }F\left( f^{i}\left( \cdot +\pi _{1}\left(
Y_{s_{0}}^{i}\right) \right) \right) ^{\circ k}\pi _{k}\left( \log
X_{u,t}^{i}\right) \left( Id_{L^{\left[ p\right] }\left( \mathcal{U}\right)
}\right) \left( y_{r}^{i,s,u,t}\right) dr, & r\in \left[ 1,2\right]%
\end{array}%
\right. \text{, }y_{0}^{i,s,u,t}=Y_{s_{0},s}^{i}\text{,}
\label{inner definition of y(i,s,u,t)}
\end{equation}%
and denote%
\begin{equation*}
y^{s,u,t}:=y^{1,s,u,t}-y^{2,s,u,t}\text{.}
\end{equation*}

Then, we have%
\begin{eqnarray}
\left\Vert \Gamma ^{s,u}+\Gamma ^{u,t}-\Gamma ^{s,t}\right\Vert
&=&\left\Vert y_{1}^{s,u}-Y_{s}+y_{1}^{u,t}-Y_{u}-\left(
y_{1}^{s,t}-Y_{s}\right) \right\Vert
\label{inner important bisect Gamma (s,t)} \\
&\leq &\left\Vert y_{1}^{u,t}-Y_{u}-\left( y_{2}^{s,u,t}-y_{1}^{s,u}\right)
\right\Vert +\left\Vert y_{2}^{s,u,t}-y_{1}^{s,t}\right\Vert  \notag
\end{eqnarray}

Based on Lemma \ref{Lemma B} on p\pageref{Lemma B}, we have%
\begin{equation}
\left\Vert y_{2}^{s,u,t}-y_{1}^{s,t}\right\Vert \leq C_{p}\left( \omega
\left( s,t\right) ^{\frac{\gamma }{p}}\left( \left\Vert Y_{s}\right\Vert
+\left\vert f^{1}-f^{2}\right\vert _{\gamma -1}\right)
+\tsum\nolimits_{n=1}^{\left[ p\right] }\omega \left( s,t\right) ^{\frac{%
\gamma -n}{p}}d_{p,\left[ s,t\right] }^{n}\left( X^{1},X^{2}\right) \right) 
\text{.}  \label{inner important Lemma A}
\end{equation}

On the other hand, since $y^{u,t}=y^{1,u,t}-y^{2,u,t}$ and $%
y^{s,u,t}=y^{1,s,u,t}-y^{2,s,u,t}$, based on the definition of $y^{i,s,u}$
(at $\left( \ref{inner definition of y(i,s,t)}\right) $) and $y^{i,s,u,t}$
(at $\left( \ref{inner definition of y(i,s,u,t)}\right) $), we have, for $%
r\in \left[ 0,1\right] $,%
\begin{eqnarray*}
&&\left\Vert y_{r}^{u,t}-Y_{u}-\left( y_{r+1}^{s,u,t}-y_{1}^{s,u}\right)
\right\Vert \\
&\leq &\sum_{k=1}^{\left[ p\right] }\int_{0}^{r}\left\Vert F\left(
f^{1}\right) ^{\circ k}\pi _{k}\left( \log X_{u,t}^{1}\right) \left( Id_{L^{%
\left[ p\right] }\left( \mathcal{U}\right) }\right) \left(
y_{v}^{1,u,t}\right) -F\left( f^{1}\right) ^{\circ k}\pi _{k}\left( \log
X_{u,t}^{1}\right) \left( Id_{L^{\left[ p\right] }\left( \mathcal{U}\right)
}\right) \left( y_{v+1}^{1,s,u,t}\right) \right. \\
&&\left. -F\left( f^{2}\right) ^{\circ k}\pi _{k}\left( \log
X_{u,t}^{2}\right) \left( Id_{L^{\left[ p\right] }\left( \mathcal{U}\right)
}\right) \left( y_{v}^{2,u,t}\right) +F\left( f^{2}\right) ^{\circ k}\pi
_{k}\left( \log X_{u,t}^{2}\right) \left( Id_{L^{\left[ p\right] }\left( 
\mathcal{U}\right) }\right) \left( Id_{L^{\left[ p\right] }\left( \mathcal{U}%
\right) }\right) \left( y_{v+1}^{2,s,u,t}\right) \right\Vert dv\text{.}
\end{eqnarray*}%
Since $\omega \left( s_{0},t_{0}\right) \leq 1$, we have $%
\max_{i=1,2}\sup_{v\in \left[ 0,1\right] }\left\Vert
y_{v}^{i,u,t}\right\Vert \vee \left\Vert y_{v+1}^{i,s,u,t}\right\Vert \leq
C_{p}$. Then by using Lemma \ref{Lemma simple but important} on p\pageref%
{Lemma simple but important} and estimating the two cases $k=1,\dots ,\left[
p\right] -1$ and $k=\left[ p\right] $ separately, we have 
\begin{eqnarray}
&&\left\Vert y_{r}^{u,t}-Y_{u}-\left( y_{r+1}^{s,u,t}-y_{1}^{s,u}\right)
\right\Vert  \label{inner important 1} \\
&\leq &C_{p}\omega \left( s,t\right) ^{\frac{1}{p}}\int_{0}^{r}\left\Vert
y_{v}^{u,t}-y_{v+1}^{s,u,t}\right\Vert dv  \notag \\
&&+C_{p}\sup_{r\in \left[ 0,1\right] }\left( \left\Vert
y_{r}^{1,u,t}-y_{r+1}^{1,s,u,t}\right\Vert +\left\Vert
y_{r}^{2,u,t}-y_{r+1}^{2,s,u,t}\right\Vert \right)  \notag \\
&&\times \left( \omega \left( s,t\right) ^{\frac{1}{p}}\left( \left\Vert
y_{r}^{1,u,t}-y_{r}^{2,u,t}\right\Vert +\left\vert f^{1}-f^{2}\right\vert
_{\gamma -1}\right) +\tsum\nolimits_{n=1}^{\left[ p\right] -1}d_{p,\left[ s,t%
\right] }^{n}\left( X^{1},X^{2}\right) \right)  \notag \\
&&+C_{p}\sup_{r\in \left[ 0,1\right] }\left( \left\Vert
y_{r}^{1,u,t}-y_{r+1}^{1,s,u,t}\right\Vert +\left\Vert
y_{r}^{2,u,t}-y_{r+1}^{2,s,u,t}\right\Vert \right) ^{\left\{ \gamma \right\}
}  \notag \\
&&\times \left( \omega \left( s,t\right) ^{\frac{\left[ p\right] }{p}}\left(
\left\Vert y_{r}^{1,u,t}-y_{r}^{2,u,t}\right\Vert +\left\vert
f^{1}-f^{2}\right\vert _{\gamma -1}\right) +d_{p,\left[ s,t\right] }^{\left[
p\right] }\left( X^{1},X^{2}\right) \right) \text{.}  \notag
\end{eqnarray}%
Since $\left\Vert Y_{s_{0},u}^{i}\right\Vert \leq C_{p}$, based on\ Lemma %
\ref{Lemma two ODE with different vector fields} on p\pageref{Lemma two ODE
with different vector fields} (continuity of ODE solutions in initial value)
and $\left( \ref{estimate of difference in main theorem}\right) $ in Theorem %
\ref{Theorem existence and uniqueness} on p\pageref{estimate of difference
in main theorem} (difference between RDE solution and ODE solution), we
have, for $i=1,2$,%
\begin{equation}
\sup_{r\in \left[ 0,1\right] }\left\Vert
y_{r}^{i,u,t}-y_{r+1}^{i,s,u,t}\right\Vert \leq C_{p}\left\Vert
Y_{s_{0},u}^{i}-y_{1}^{i,s,u}\right\Vert \leq C_{p}\omega \left( s,u\right)
^{\frac{\left[ p\right] +1}{p}}\text{.}  \label{inner important 2}
\end{equation}%
Based on Lemma \ref{Lemma two ODE with different vector fields} on p\pageref%
{Lemma two ODE with different vector fields} (continuous dependence of ODE
solution in term of initial value and vector field), ($%
Y_{u}=Y_{s_{0},u}^{1}-Y_{s_{0},u}^{2}$)%
\begin{equation}
\sup_{r\in \left[ 0,1\right] }\left\Vert
y_{r}^{1,u,t}-y_{r}^{2,u,t}\right\Vert \leq C_{p}\left( \left\Vert
Y_{u}\right\Vert +\omega \left( s,t\right) ^{\frac{1}{p}}\left\vert
f^{1}-f^{2}\right\vert _{\left[ p\right] -1}+\tsum\nolimits_{n=1}^{\left[ p%
\right] }d_{p,\left[ s,t\right] }^{n}\left( X^{1},X^{2}\right) \right) \text{%
.}  \label{inner important 3}
\end{equation}%
Then combining $\left( \ref{inner important 1}\right) $, $\left( \ref{inner
important 2}\right) $ and $\left( \ref{inner important 3}\right) $, we get%
\begin{eqnarray}
&&\left\Vert y_{r}^{u,t}-Y_{u}-\left( y_{r+1}^{s,u,t}-y_{1}^{s,u}\right)
\right\Vert  \label{inner important 4} \\
&\leq &C_{p}\omega \left( s,t\right) ^{\frac{1}{p}}\int_{0}^{r}\left\Vert
y_{v}^{u,t}-y_{v+1}^{s,u,t}\right\Vert dv  \notag \\
&&+C_{p}\omega \left( s,t\right) ^{\frac{\gamma }{p}}\left( \left\Vert
Y_{u}\right\Vert +\left\vert f^{1}-f^{2}\right\vert _{\gamma -1}\right)
+C_{p}\tsum\nolimits_{n=1}^{\left[ p\right] }\omega \left( s,t\right) ^{%
\frac{\gamma -n}{p}}d_{p,\left[ s,t\right] }^{n}\left( X^{1},X^{2}\right) 
\text{.}  \notag
\end{eqnarray}%
Then by using Gronwall's inequality, we have%
\begin{eqnarray*}
&&\sup_{r\in \left[ 0,1\right] }\left\Vert
y_{r}^{u,t}-y_{r+1}^{s,u,t}\right\Vert \\
&\leq &C_{p}\left( \left\Vert Y_{u}-y_{1}^{s,u}\right\Vert +\omega \left(
s,t\right) ^{\frac{\gamma }{p}}\left( \left\Vert Y_{u}\right\Vert
+\left\vert f^{1}-f^{2}\right\vert _{\gamma -1}\right)
+\tsum\nolimits_{n=1}^{\left[ p\right] }\omega \left( s,t\right) ^{\frac{%
\gamma -n}{p}}d_{p,\left[ s,t\right] }^{n}\left( X^{1},X^{2}\right) \right) 
\text{,}
\end{eqnarray*}%
and combined with $\left( \ref{inner important 4}\right) $, we have ($\Gamma
^{s,u}:=Y_{u}-y_{1}^{s,u}$)%
\begin{eqnarray}
&&\left\Vert y_{1}^{u,t}-Y_{u}-\left( y_{2}^{s,u,t}-y_{1}^{s,u}\right)
\right\Vert  \label{inner important Lemma B} \\
&\leq &C_{p}\left( \omega \left( s,t\right) ^{\frac{1}{p}}\left\Vert \Gamma
^{s,u}\right\Vert +\omega \left( s,t\right) ^{\frac{\gamma }{p}}\left(
\left\Vert Y_{u}\right\Vert +\left\vert f^{1}-f^{2}\right\vert _{\gamma
-1}\right) +\tsum\nolimits_{n=1}^{\left[ p\right] }\omega \left( s,t\right)
^{\frac{\gamma -n}{p}}d_{p,\left[ s,t\right] }^{n}\left( X^{1},X^{2}\right)
\right) \text{.}  \notag
\end{eqnarray}

As a result, combining $\left( \ref{inner important bisect Gamma (s,t)}%
\right) $, $\left( \ref{inner important Lemma A}\right) $ and $\left( \ref%
{inner important Lemma B}\right) $, we have%
\begin{eqnarray*}
&&\left\Vert \Gamma ^{s,u}+\Gamma ^{u,t}-\Gamma ^{s,t}\right\Vert \\
&\leq &C_{p}\left( \omega \left( s,t\right) ^{\frac{1}{p}}\left\Vert \Gamma
^{s,u}\right\Vert +\omega \left( s,t\right) ^{\frac{\gamma }{p}}\left(
\sup_{r\in \left[ s,t\right] }\left\Vert Y_{r}\right\Vert +\left\vert
f^{1}-f^{2}\right\vert _{\gamma -1}\right) +\tsum\nolimits_{n=1}^{\left[ p%
\right] }\omega \left( s,t\right) ^{\frac{\gamma -n}{p}}d_{p,\left[ s,t%
\right] }^{n}\left( X^{1},X^{2}\right) \right) \text{.}
\end{eqnarray*}

Therefore, if we denote%
\begin{equation*}
\widetilde{\omega }\left( s,t\right) :=\omega \left( s,t\right) \left(
\sup_{r\in \left[ s,t\right] }\left\Vert Y_{r}\right\Vert +\left\vert
f^{1}-f^{2}\right\vert _{\gamma -1}\right) ^{\frac{p}{\gamma }%
}+\tsum\nolimits_{n=1}^{\left[ p\right] }\omega \left( s,t\right) ^{\frac{%
\gamma -n}{\gamma }}\left( d_{p,\left[ s,t\right] }^{n}\left(
X^{1},X^{2}\right) ^{\frac{p}{n}}\right) ^{\frac{n}{\gamma }}\text{,}
\end{equation*}%
then $\widetilde{\omega }$ is a control, i.e. $\widetilde{\omega }$ is
continuous, vanishes on the diagonal and 
\begin{equation*}
\widetilde{\omega }\left( s,u\right) +\widetilde{\omega }\left( u,t\right)
\leq \widetilde{\omega }\left( s,t\right) \text{, }\forall s\leq u\leq t%
\text{.}
\end{equation*}%
Indeed, since 
\begin{equation*}
d_{p,\left[ s,t\right] }^{n}\left( X^{1},X^{2}\right) ^{\frac{p}{n}%
}=\sup_{D\subset \left[ s,t\right] }\sum_{j,t_{j}\in D}\left\Vert \pi
_{n}\left( X_{t_{j},t_{j+1}}^{1}\right) -\pi _{n}\left(
X_{t_{j},t_{j+1}}^{2}\right) \right\Vert ^{\frac{p}{n}}\text{ is a control,}
\end{equation*}%
and $\frac{\gamma -n}{\gamma }+\frac{n}{\gamma }=1$, by using H\"{o}lder
inequality (or based on $\left( iii\right) $ in Exer 1.9 \cite{Friz Victoir
book}), 
\begin{equation*}
\omega \left( s,t\right) ^{\frac{\gamma -n}{\gamma }}\left( d_{p,\left[ s,t%
\right] }^{n}\left( X^{1},X^{2}\right) ^{\frac{p}{n}}\right) ^{\frac{n}{%
\gamma }}
\end{equation*}%
is another control.

Then we have%
\begin{equation*}
\left\Vert \Gamma ^{s,u}+\Gamma ^{u,t}-\Gamma ^{s,t}\right\Vert \leq
C_{p}\omega \left( s,t\right) ^{\frac{1}{p}}\left\Vert \Gamma
^{s,u}\right\Vert +C_{p}\widetilde{\omega }\left( s,t\right) ^{\frac{\gamma 
}{p}}\text{.}
\end{equation*}%
Since $\omega $ and $\widetilde{\omega }$ are two different controls, we
select the dyadic partition in such a way that, for some integers $1\leq
m\leq M$ satisfying%
\begin{equation*}
1>\frac{m}{M}>\frac{p}{\gamma }\text{,}
\end{equation*}%
we let 
\begin{eqnarray*}
I_{0}^{0} &:&=\left[ s,t\right] \text{ and }I_{k}^{n}=I_{2k}^{n+1}\cup
I_{2k+1}^{n+1}\text{ where} \\
\widetilde{\omega }\left( I_{2k}^{n+1}\right) &=&\widetilde{\omega }\left(
I_{2k+1}^{n+1}\right) \leq \frac{1}{2}\widetilde{\omega }\left(
I_{k}^{n}\right) \text{ for }k=0,1,\dots ,2^{n}-1\text{ and }n=lM+s\text{, }%
s=1,\dots ,m\text{, }l\in 
\mathbb{N}
\text{,} \\
\omega \left( I_{2k}^{n+1}\right) &=&\omega \left( I_{2k+1}^{n+1}\right)
\leq \frac{1}{2}\omega \left( I_{k}^{n}\right) \text{ for }k=0,1,\dots
,2^{n}-1\text{ and }n=lM+s\text{, }s=m+1,\dots ,M\text{, }l\in 
\mathbb{N}
\text{.}
\end{eqnarray*}%
Therefore, we have%
\begin{eqnarray*}
\sup_{k=0,\dots ,2^{n}-1}\widetilde{\omega }\left( I_{k}^{n}\right) &\leq
&2^{-\left[ \frac{m}{M}n\right] }\widetilde{\omega }\left( s,t\right) \leq
2^{-\frac{m}{M}n+1}\widetilde{\omega }\left( s,t\right) \text{,} \\
\sup_{k=0,\dots ,2^{n}-1}\omega \left( I_{k}^{n}\right) &\leq &2^{-\left[ 
\frac{M-m}{M}n\right] }\omega \left( s,t\right) \leq 2^{-\frac{M-m}{M}%
n+1}\omega \left( s,t\right) \text{.}
\end{eqnarray*}

By recursively bisecting $\left[ s,t\right] $ in this way, we have, 
\begin{eqnarray}
\left\Vert \Gamma ^{s,t}\right\Vert &\leq &\overline{\lim }_{n\rightarrow
\infty }\sum_{k=0}^{n}\left( \prod\nolimits_{j=0}^{k}\left(
1+\max_{i=0,1,\dots ,2^{j}-1}\omega \left( I_{i}^{j}\right) ^{\frac{1}{p}%
}\right) \left( 2^{k}\max_{i=0,1,\dots ,2^{k}-1}\widetilde{\omega }\left(
I_{i}^{k}\right) ^{\frac{\gamma }{p}}\right) \right)
\label{inner estimate on dyadic interval2} \\
&&+\overline{\lim }_{n\rightarrow \infty }\prod\nolimits_{j=0}^{n}\left(
1+\max_{i=0,1,\dots ,2^{j}-1}\omega \left( I_{i}^{j}\right) ^{\frac{1}{p}%
}\right) \left( \sum_{j=0}^{2^{n}-1}\left\Vert \Gamma
^{t_{j}^{n},t_{j+1}^{n}}\right\Vert \right)  \notag \\
&\leq &\overline{\lim }_{n\rightarrow \infty }\sum_{k=0}^{n}\left(
\prod\nolimits_{j=0}^{k}\left( 1+2^{\frac{1}{p}}2^{-\frac{M-m}{Mp}j}\omega
\left( s,t\right) ^{\frac{1}{p}}\right) 2^{\left( 1-\frac{\gamma }{p}\frac{m%
}{M}\right) k+\frac{\gamma }{p}}\right) \widetilde{\omega }\left( s,t\right)
^{\frac{\gamma }{p}}  \notag \\
&&+\overline{\lim }_{n\rightarrow \infty }\prod\nolimits_{j=0}^{n}\left(
1+2^{\frac{1}{p}}2^{-\frac{M-m}{Mp}j}\omega \left( s,t\right) ^{\frac{1}{p}%
}\right) \left( \sum_{j=0}^{2^{n}-1}\left\Vert \Gamma
^{t_{j}^{n},t_{j+1}^{n}}\right\Vert \right) \text{.}  \notag
\end{eqnarray}%
Then since $1>\frac{m}{M}>\frac{p}{\gamma }$ and $||\Gamma
^{t_{j}^{n},t_{j+1}^{n}}||\leq C_{p}\omega \left(
t_{j}^{n},t_{j+1}^{n}\right) ^{\frac{\left[ p\right] +1}{p}}$($\left( \ref%
{estimate of difference in main theorem}\right) $ in Theorem \ref{Theorem
existence and uniqueness} on p\pageref{Theorem existence and uniqueness}),
we have 
\begin{eqnarray}
\left\Vert \Gamma ^{s,t}\right\Vert &\leq &\exp \left( \frac{2^{\frac{1}{p}}%
}{1-2^{-\frac{M-m}{Mp}}}\omega \left( s,t\right) ^{\frac{1}{p}}\right) \frac{%
2^{\frac{\gamma }{p}}}{1-2^{-\left( \frac{\gamma m}{pM}-1\right) }}%
\widetilde{\omega }\left( s,t\right) ^{\frac{\gamma }{p}}  \notag \\
&\leq &\exp \left( \frac{2^{\frac{1}{p}}}{1-2^{-\frac{M-m}{Mp}}}\right) 
\frac{2^{\frac{\gamma }{p}}}{1-2^{-\left( \frac{\gamma m}{pM}-1\right) }}%
\widetilde{\omega }\left( s,t\right) ^{\frac{\gamma }{p}}  \notag
\end{eqnarray}%
As a result, for any $\left[ s,t\right] \subseteq \left[ s_{0},t_{0}\right] $%
, $\omega \left( s_{0},t_{0}\right) \leq 1$,%
\begin{equation}
\left\Vert Y_{t}-y_{1}^{s,t}\right\Vert \leq C_{p,\gamma }\widetilde{\omega }%
\left( s,t\right) ^{\frac{\gamma }{p}}\text{.}
\label{inner estimate of difference between Y and y}
\end{equation}

In particular, if we let 
\begin{equation*}
\left[ s,t\right] =\left[ s_{0},t_{0}\right] \text{,}
\end{equation*}%
then $y^{i,s,t}$ is the solution of the ODE%
\begin{equation*}
dy_{r}^{i,s,t}=\sum_{k=1}^{\left[ p\right] }F\left( f^{i}\left( \cdot +\pi
_{1}\left( Y_{s}^{i}\right) \right) \right) ^{\circ k}\pi _{k}\left( \log
\left( X_{s,t}\right) \right) \left( Id_{L^{\left[ p\right] }\left( \mathcal{%
U}\right) }\right) \left( y_{r}^{i,s,t}\right) dr\text{, }r\in \left[ 0,1%
\right] \text{, }y_{0}^{i,s,t}=1\text{.}
\end{equation*}%
Based on Lemma \ref{Lemma two ODE with different vector fields} on p\pageref%
{Lemma two ODE with different vector fields}, for $k=1,2,\dots ,\left[ p%
\right] $,%
\begin{eqnarray*}
\left\Vert \pi _{k}\left( y_{1}^{s,t}\right) \right\Vert &=&\left\Vert \pi
_{k}\left( y_{1}^{1,s,t}\right) -\pi _{k}\left( y_{1}^{2,s,t}\right)
\right\Vert \\
&\leq &C_{p}\left( \omega \left( s,t\right) ^{\frac{k}{p}}\left( \left\vert
f^{1}-f^{2}\right\vert _{\gamma -1}+\left\Vert \pi _{1}\left(
Y_{s}^{1}\right) -\pi _{1}\left( Y_{s}^{2}\right) \right\Vert \right)
+\sum_{n=1}^{\left[ p\right] }\omega \left( s,t\right) ^{\frac{\left(
k-n\right) \vee 0}{p}}d_{p,\left[ s,t\right] }^{n}\left( X^{1},X^{2}\right)
\right) \text{.}
\end{eqnarray*}%
Combined with $\left( \ref{inner estimate of difference between Y and y}%
\right) $, we get ($Y_{t}=Y_{s,t}^{1}-Y_{s,t}^{2}$)%
\begin{eqnarray}
\left\Vert \pi _{k}\left( Y_{s,t}^{1}\right) -\pi _{k}\left(
Y_{s,t}^{2}\right) \right\Vert &\leq &C_{p,\gamma }\omega \left( s,t\right)
^{\frac{k}{p}}\left( \left\vert f^{1}-f^{2}\right\vert _{\gamma
-1}+\left\Vert \pi _{1}\left( Y_{s}^{1}\right) -\pi _{1}\left(
Y_{s}^{2}\right) \right\Vert \right)
\label{inner difference between two RDE solutions} \\
&&+C_{p,\gamma }\left( \sum_{n=1}^{\left[ p\right] }\omega \left( s,t\right)
^{\frac{\left( k-n\right) \vee 0}{p}}d_{p,\left[ s,t\right] }^{n}\left(
X^{1},X^{2}\right) +\omega \left( s,t\right) ^{\frac{\gamma }{p}}\left(
\sup_{u\in \left[ s,t\right] }\left\Vert Y_{s,u}^{1}-Y_{s,u}^{2}\right\Vert
\right) \right) \text{.}  \notag
\end{eqnarray}%
Recall control $\omega \left( s,t\right) :=\left\vert f^{1}\right\vert
_{\gamma }^{p}\left\Vert X^{1}\right\Vert _{p-var,\left[ s,t\right]
}^{p}+\left\vert f^{2}\right\vert _{\gamma }^{p}\left\Vert X^{2}\right\Vert
_{p-var,\left[ s,t\right] }^{p}$. For $\alpha \in (0,1]$, denote $\omega
^{\alpha }\left( s,t\right) $ as at $\left( \ref{Definition of omega alpha}%
\right) $ on p\pageref{Definition of omega alpha} and denote $d_{p,\left[ s,t%
\right] }^{n,\alpha }\left( X^{1},X^{2}\right) $ as at $\left( \ref%
{Definition of dpn^alpha}\right) $ on p\pageref{Definition of dpn^alpha}.
Then based on $\left( \ref{inner difference between two RDE solutions}%
\right) $, by using sub-additivity of control, we have ($\omega ^{\alpha
}\left( s,t\right) \leq \omega \left( s,t\right) \leq 1$)%
\begin{eqnarray}
\left\Vert \pi _{k}\left( Y_{s,t}^{1}\right) -\pi _{k}\left(
Y_{s,t}^{2}\right) \right\Vert &\leq &d_{p,\left[ s,t\right] }^{k}\left(
Y^{1},Y^{2}\right)  \label{inner difference between two RDE solutions 1} \\
&\leq &C_{p,\gamma }\omega ^{\alpha }\left( s,t\right) ^{\frac{k}{p}}\left(
\left\vert f^{1}-f^{2}\right\vert _{\gamma -1}+\sup_{u\in \left[ s,t\right]
}\left\Vert \pi _{1}\left( Y_{u}^{1}\right) -\pi _{1}\left( Y_{u}^{2}\right)
\right\Vert \right)  \notag \\
&&+C_{p,\gamma }\left( \sum_{n=1}^{\left[ p\right] }\omega ^{\alpha }\left(
s,t\right) ^{\frac{\left( k-n\right) \vee 0}{p}}d_{p,\left[ s,t\right]
}^{n,\alpha }\left( X^{1},X^{2}\right) +\omega ^{\alpha }\left( s,t\right) ^{%
\frac{\gamma }{p}}\left( \sup_{u\in \left[ s,t\right] }\left\Vert
Y_{s,u}^{1}-Y_{s,u}^{2}\right\Vert \right) \right)  \notag \\
&\leq &C_{p,\gamma }\omega ^{\alpha }\left( s,t\right) ^{\frac{k}{p}}\left(
\left\vert f^{1}-f^{2}\right\vert _{\gamma -1}+\left\Vert \pi _{1}\left(
Y_{s}^{1}\right) -\pi _{1}\left( Y_{s}^{2}\right) \right\Vert \right)  \notag
\\
&&+C_{p,\gamma }\left( \sum_{n=1}^{\left[ p\right] }\omega ^{\alpha }\left(
s,t\right) ^{\frac{\left( k-n\right) \vee 0}{p}}d_{p,\left[ s,t\right]
}^{n,\alpha }\left( X^{1},X^{2}\right) +\omega ^{\alpha }\left( s,t\right) ^{%
\frac{k}{p}}\left( \sup_{u\in \left[ s,t\right] }\left\Vert
Y_{s,u}^{1}-Y_{s,u}^{2}\right\Vert \right) \right) \text{.}  \notag
\end{eqnarray}%
As a result,%
\begin{eqnarray*}
\sup_{u\in \left[ s,t\right] }\left\Vert Y_{s,u}^{1}-Y_{s,u}^{2}\right\Vert
&\leq &C_{p,\gamma }\omega ^{\alpha }\left( s,t\right) ^{\frac{1}{p}}\left(
\left\vert f^{1}-f^{2}\right\vert _{\gamma -1}+\left\Vert \pi _{1}\left(
Y_{s}^{1}\right) -\pi _{1}\left( Y_{s}^{2}\right) \right\Vert \right) \\
&&+C_{p,\gamma }\left( \sum_{n=1}^{\left[ p\right] }d_{p,\left[ s,t\right]
}^{n,\alpha }\left( X^{1},X^{2}\right) +\omega ^{\alpha }\left( s,t\right) ^{%
\frac{1}{p}}\left( \sup_{u\in \left[ s,t\right] }\left\Vert
Y_{s,u}^{1}-Y_{s,u}^{2}\right\Vert \right) \right) \text{.}
\end{eqnarray*}%
Then there exists $\delta _{p,\gamma }>0$ such that for $\left[ s,t\right] $
satisfying $\omega \left( s,t\right) \leq \delta _{p,\gamma }$, we have%
\begin{equation*}
\sup_{u\in \left[ s,t\right] }\left\Vert Y_{s,u}^{1}-Y_{s,u}^{2}\right\Vert
\leq C_{p,\gamma }\omega ^{\alpha }\left( s,t\right) ^{\frac{1}{p}}\left(
\left\vert f^{1}-f^{2}\right\vert _{\gamma -1}+\left\Vert \pi _{1}\left(
Y_{s}^{1}\right) -\pi _{1}\left( Y_{s}^{2}\right) \right\Vert \right)
+C_{p,\gamma }\sum_{n=1}^{\left[ p\right] }d_{p,\left[ s,t\right]
}^{n,\alpha }\left( X^{1},X^{2}\right) \text{,}
\end{equation*}%
and combined with $\left( \ref{inner difference between two RDE solutions 1}%
\right) $, we have, for any $\left[ s,t\right] $ satisfying $\omega \left(
s,t\right) \leq \delta _{p,\gamma }$ and any $\alpha \in (0,\delta
_{p,\gamma }]$,%
\begin{eqnarray}
&&\left\Vert \pi _{k}\left( Y_{s,t}^{1}\right) -\pi _{k}\left(
Y_{s,t}^{2}\right) \right\Vert
\label{inner estimate when interval is less than delta} \\
&\leq &C_{p,\gamma }\omega ^{\alpha }\left( s,t\right) ^{\frac{k}{p}}\left(
\left\vert f^{1}-f^{2}\right\vert _{\gamma -1}+\left\Vert \pi _{1}\left(
Y_{s}^{1}\right) -\pi _{1}\left( Y_{s}^{2}\right) \right\Vert \right)
+C_{p,\gamma }\sum_{n=1}^{\left[ p\right] }\omega ^{\alpha }\left(
s,t\right) ^{\frac{\left( k-n\right) \vee 0}{p}}d_{p,\left[ s,t\right]
}^{n,\alpha }\left( X^{1},X^{2}\right) \text{.}  \notag
\end{eqnarray}

For $\left[ s_{0},t_{0}\right] $ satisfying $\omega \left(
s_{0},t_{0}\right) \geq \delta _{p,\gamma }$ and $\alpha \in (0,\delta
_{p,\gamma }]$, we decompose $\left[ s_{0},t_{0}\right] =\cup _{j=0}^{m-1}%
\left[ t_{j},t_{j+1}\right] $, $m\geq 2$, with $\omega \left(
t_{j},t_{j+1}\right) =\alpha $, $j=0,1,\dots ,m-2$, and $\omega \left(
t_{m-1,}t_{m}\right) \leq \alpha $. Then $\left( m-1\right) \alpha
=\sum_{j=0}^{m-2}\omega \left( t_{j},t_{j+1}\right) $, so%
\begin{equation}
m=\alpha ^{-1}\left( \sum_{j=0}^{m-2}\omega \left( t_{j},t_{j+1}\right)
\right) +1\leq 2\alpha ^{-1}\left( \sum_{j=0}^{m-1}\omega \left(
t_{j},t_{j+1}\right) \right) \leq 2\alpha ^{-1}\omega ^{\alpha }\left(
s,t\right) .  \label{inner bound on m}
\end{equation}%
Firstly, we estimate $\sum_{j=0}^{m-1}\left\Vert \pi _{k}\left(
Y_{t_{j},t_{j+1}}^{1}\right) -\pi _{k}\left( Y_{t_{j},t_{j+1}}^{2}\right)
\right\Vert $ for $k=1,2,\dots ,\left[ p\right] $. By using $\left( \ref%
{inner estimate when interval is less than delta}\right) $ and denote $%
\gamma _{k}:=\alpha ^{\frac{k-1}{p}}$, $\beta :=C_{p,\gamma }\alpha ^{\frac{1%
}{p}}$ and $\lambda _{k,n}:=C_{p,\gamma }\alpha ^{\frac{\left( k-n\right)
\vee 0}{p}}$, we have, ($\omega ^{\alpha }\left( t_{j},t_{j+1}\right)
=\omega \left( t_{j},t_{j+1}\right) \leq \alpha $, $\forall j$) 
\begin{eqnarray*}
&&\sum_{j=0}^{m-1}\left\Vert \pi _{k}\left( Y_{t_{j},t_{j+1}}^{1}\right)
-\pi _{1}\left( Y_{t_{j},t_{j+1}}^{2}\right) \right\Vert \\
&\leq &\alpha ^{\frac{k-1}{p}}\sum_{j=0}^{m-1}C_{p,\gamma }\alpha ^{\frac{1}{%
p}}\left( \left\vert f^{1}-f^{2}\right\vert _{\gamma -1}+\left\Vert \pi
_{1}\left( Y_{t_{j}}^{1}\right) -\pi _{1}\left( Y_{t_{j}}^{2}\right)
\right\Vert \right) +\sum_{j=0}^{m-1}\sum_{n=1}^{\left[ p\right]
}C_{p,\gamma }\alpha ^{\frac{\left( k-n\right) \vee 0}{p}}d_{p,\left[
t_{j},t_{j+1}\right] }^{n,\alpha }\left( X^{1},X^{2}\right) \\
&=&:\gamma _{k}\sum_{j=0}^{m-1}\beta \left( \left\vert
f^{1}-f^{2}\right\vert _{\gamma -1}+\left\Vert \pi _{1}\left(
Y_{t_{j}}^{1}\right) -\pi _{1}\left( Y_{t_{j}}^{2}\right) \right\Vert
\right) +\sum_{j=0}^{m-1}\sum_{n=1}^{\left[ p\right] }\lambda _{k,n}d_{p,%
\left[ t_{j},t_{j+1}\right] }^{n,\alpha }\left( X^{1},X^{2}\right) \text{.}
\end{eqnarray*}%
Then by repeatedly using $\left( \ref{inner estimate when interval is less
than delta}\right) $, we have ($\gamma _{k}\lambda _{1,n}\leq \lambda _{k,n}$%
) 
\begin{eqnarray*}
&&\sum_{j=0}^{m-1}\left\Vert \pi _{k}\left( Y_{t_{j},t_{j+1}}^{1}\right)
-\pi _{1}\left( Y_{t_{j},t_{j+1}}^{2}\right) \right\Vert \\
&\leq &\gamma _{k}\left( m\beta \left\vert f^{1}-f^{2}\right\vert _{\gamma
-1}+\beta \left\Vert Y_{s_{0}}^{1}-Y_{s_{0}}^{2}\right\Vert \right)
+\sum_{j=0}^{m-1}\sum_{n=1}^{\left[ p\right] }\lambda _{k,n}d_{p,\left[
t_{j},t_{j+1}\right] }^{n,\alpha }\left( X^{1},X^{2}\right) \\
&&+\gamma _{k}\beta \sum_{j=0}^{m-2}\left( \left\Vert \pi _{1}\left(
Y_{t_{j}}^{1}\right) -\pi _{1}\left( Y_{t_{j}}^{2}\right) \right\Vert +\beta
\left( \left\vert f^{1}-f^{2}\right\vert _{\gamma -1}+\left\Vert \pi
_{1}\left( Y_{t_{j}}^{1}\right) -\pi _{1}\left( Y_{t_{j}}^{2}\right)
\right\Vert \right) +\sum_{n=1}^{\left[ p\right] }\lambda _{1,n}d_{p,\left[
t_{j},t_{j+1}\right] }^{n,\alpha }\left( X^{1},X^{2}\right) \right) \\
&\leq &\cdots \leq \gamma _{k}\left( m\beta +\left( m-1\right) \beta
^{2}+\left( m-2\right) \beta ^{2}\left( 1+\beta \right) +\cdots +\beta
^{2}\left( 1+\beta \right) ^{m-2}\right) \left\vert f^{1}-f^{2}\right\vert
_{\gamma -1} \\
&&+\gamma _{k}\left( \beta +\beta \left( 1+\beta \right) +\beta \left(
1+\beta \right) ^{2}+\cdots +\beta \left( 1+\beta \right) ^{m-1}\right)
\left\Vert \pi _{1}\left( Y_{s_{0}}^{1}\right) -\pi _{1}\left(
Y_{s_{0}}^{2}\right) \right\Vert \\
&&+\sum_{j=0}^{m-2}\left( 1+\beta +\beta \left( 1+\beta \right) \cdots
+\beta \left( 1+\beta \right) ^{m-2-j}\right) \sum_{n=1}^{\left[ p\right]
}\lambda _{k,n}d_{p,\left[ t_{j},t_{j+1}\right] }^{n,\alpha }\left(
X^{1},X^{2}\right) +\sum_{n=1}^{\left[ p\right] }\lambda _{k,n}d_{p,\left[
t_{m-1},t_{m}\right] }^{n,\alpha }\left( X^{1},X^{2}\right) \\
&\leq &\gamma _{k}\left( \left( 1+\beta \right) ^{m}-1\right) \left(
\left\vert f^{1}-f^{2}\right\vert _{\gamma -1}+\left\Vert \pi _{1}\left(
Y_{s_{0}}^{1}\right) -\pi _{1}\left( Y_{s_{0}}^{2}\right) \right\Vert
\right) +\sum_{n=1}^{\left[ p\right] }\lambda _{k,n}\left(
\sum_{j=0}^{m-1}\left( 1+\beta \right) ^{m-1-j}d_{p,\left[ t_{j},t_{j+1}%
\right] }^{n,\alpha }\left( X^{1},X^{2}\right) \right) \text{.}
\end{eqnarray*}%
When $p$ is not an integer, for $n=1,\dots ,\left[ p\right] $, by using H%
\"{o}lder inequality,%
\begin{align}
\sum_{j=0}^{m-1}\left( 1+\beta \right) ^{m-1-j}d_{p,\left[ t_{j},t_{j+1}%
\right] }^{n}\left( X^{1},X^{2}\right) & \leq \left( \sum_{j=0}^{m-1}\left(
1+\beta \right) ^{\frac{p}{p-n}\left( m-1-j\right) }\right) ^{1-\frac{n}{p}%
}\left( \sum_{j=0}^{m-1}d_{p,\left[ t_{j},t_{j+1}\right] }^{n,\alpha }\left(
X^{1},X^{2}\right) ^{\frac{p}{n}}\right) ^{\frac{n}{p}}
\label{inner estimate when n is small} \\
& \leq C_{p}\left( \frac{\left( 1+\beta \right) ^{\frac{p}{p-n}m}-1}{\beta }%
\right) d_{p,\left[ s,t\right] }^{n,\alpha }\left( X^{1},X^{2}\right) . 
\notag
\end{align}%
When $p$ is an integer, for integer $n<p$, $\left( \ref{inner estimate when
n is small}\right) $ holds; for $n=p$, 
\begin{equation*}
\sum_{j=0}^{m-1}\left( 1+\beta \right) ^{m-1-j}d_{p,\left[ t_{j},t_{j+1}%
\right] }^{n}\left( X^{1},X^{2}\right) \leq \frac{\left( 1+\beta \right)
^{m}-1}{\beta }d_{p,\left[ s,t\right] }^{n,\alpha }\left( X^{1},X^{2}\right) 
\text{.}
\end{equation*}%
Then by using $\left( 1+\beta \right) ^{m}-1\leq m\left( 1+\beta \right)
^{m-1}\beta $, we have ($\lfloor p\rfloor $ denotes the largest integer
which is strictly less than $p$)%
\begin{eqnarray*}
&&\sum_{j=0}^{m-1}\left\Vert \pi _{k}\left( Y_{t_{j},t_{j+1}}^{1}\right)
-\pi _{1}\left( Y_{t_{j},t_{j+1}}^{2}\right) \right\Vert \\
&\leq &C_{p}m\left( 1+\beta \right) ^{\frac{p}{p-\lfloor p\rfloor }m}\left(
\gamma _{k}\beta \left( \left\vert f^{1}-f^{2}\right\vert _{\gamma
-1}+\left\Vert \pi _{1}\left( Y_{s_{0}}^{1}\right) -\pi _{1}\left(
Y_{s_{0}}^{2}\right) \right\Vert \right) +\sum_{n=1}^{\left[ p\right]
}\lambda _{k,n}d_{p,\left[ s,t\right] }^{n,\alpha }\left( X^{1},X^{2}\right)
\right) \text{.}
\end{eqnarray*}%
Combined with $\left( \ref{inner bound on m}\right) $, we have ($\gamma
_{k}:=\alpha ^{\frac{k-1}{p}}$, $\beta :=C_{p,\gamma }\alpha ^{\frac{1}{p}}$%
, $\lambda _{k,n}:=C_{p,\gamma }\alpha ^{\frac{\left( k-n\right) \vee 0}{p}}$%
, $\alpha \in (0,\delta _{p,\gamma }]$ with $\delta _{p,\gamma }\leq 1$) \ 
\begin{eqnarray}
&&\sum_{j=0}^{m-1}\left\Vert \pi _{k}\left( Y_{t_{j},t_{j+1}}^{1}\right)
-\pi _{1}\left( Y_{t_{j},t_{j+1}}^{2}\right) \right\Vert
\label{inner estimate of accumulated level k} \\
&\leq &C_{p,\gamma }\exp \left( C_{p,\gamma }\alpha ^{-1}\omega ^{\alpha
}\left( s,t\right) \right) \left( \alpha ^{\frac{k}{p}}\left( \left\vert
f^{1}-f^{2}\right\vert _{\gamma -1}+\left\Vert \pi _{1}\left(
Y_{s_{0}}^{1}\right) -\pi _{1}\left( Y_{s_{0}}^{2}\right) \right\Vert
\right) +\sum_{n=1}^{\left[ p\right] }\alpha ^{\frac{\left( k-n\right) \vee 0%
}{p}}d_{p,\left[ s,t\right] }^{n,\alpha }\left( X^{1},X^{2}\right) \right) 
\text{.}  \notag
\end{eqnarray}

Then, it can be checked that,%
\begin{eqnarray*}
&&\left\Vert \pi _{k}\left( Y_{s,t}^{1}\right) -\pi _{k}\left(
Y_{s,t}^{2}\right) \right\Vert \\
&\leq &\sum_{j=0}^{m-1}\left\Vert \pi _{k}\left(
Y_{t_{j},t_{j+1}}^{1}\right) -\pi _{k}\left( Y_{t_{j},t_{j+1}}^{2}\right)
\right\Vert \\
&&+\sum_{i_{0}+\cdots +i_{m-1}=k\text{, }i_{s}=0,1,\dots ,k-1}\left\Vert \pi
_{i_{0}}\left( Y_{s,t_{1}}^{1}\right) \otimes \cdots \otimes \pi
_{i_{m-1}}\left( Y_{t_{m-1},t_{m}}^{1}\right) -\pi _{i_{0}}\left(
Y_{s,t_{1}}^{2}\right) \otimes \cdots \otimes \pi _{i_{m-1}}\left(
Y_{t_{m-1},t_{m}}^{2}\right) \right\Vert \text{.}
\end{eqnarray*}%
Based on $\left( \ref{bound on RDE solution in main theorem}\right) $ in
Theorem \ref{Theorem existence and uniqueness} on p\pageref{Theorem
existence and uniqueness}, for $j=0,1,\dots ,m-1$ and $s=1,2,\dots ,\left[ p%
\right] $, $|\hspace{-0.01in}|\pi _{s}(Y_{t_{j},t_{j+1}}^{i})|\hspace{-0.01in%
}|\leq C_{p}\omega \left( t_{j},t_{j+1}\right) ^{\frac{s}{p}}$. Then by
replacing $\omega $ by $C_{p}\omega $, we have 
\begin{equation*}
\left\Vert \pi _{s}\left( Y_{t_{j},t_{j+1}}^{i}\right) \right\Vert \leq 
\frac{\omega \left( t_{j},t_{j+1}\right) ^{\frac{s}{p}}}{\left( \frac{s}{p}%
\right) !}\text{, }s=1,\dots ,\left[ p\right] \text{, }j=0,1,\dots ,m-1\text{%
, }i=1,2\text{.}
\end{equation*}%
Then%
\begin{eqnarray*}
&&\sum_{i_{0}+\cdots +i_{m-1}=k\text{, }i_{s}\leq k-1}\left\Vert \pi
_{i_{0}}\left( Y_{s,t_{1}}^{1}\right) \otimes \cdots \otimes \pi
_{i_{m-1}}\left( Y_{t_{m-1},t_{m}}^{1}\right) -\pi _{i_{0}}\left(
Y_{s,t_{1}}^{2}\right) \otimes \cdots \otimes \pi _{i_{m-1}}\left(
Y_{t_{m-1},t_{m}}^{2}\right) \right\Vert \\
&\leq &\sum_{j=0}^{m-1}\sum_{s=1}^{k-1}\left\Vert \pi _{s}\left(
Y_{t_{j},t_{j+1}}^{1}\right) -\pi _{s}\left( Y_{t_{j},t_{j+1}}^{2}\right)
\right\Vert \left( \sum_{\tsum\nolimits_{i\neq j}s_{i}=k-s,s_{i}\geq
0}\prod\nolimits_{i\neq j}\frac{\omega \left( t_{i},t_{i+1}\right) ^{\frac{%
s_{i}}{p}}}{\left( \frac{s_{i}}{p}\right) !}\right) \text{.}
\end{eqnarray*}%
By using neo-classical inequality (Exer 3.9 \cite{Lyonsnotes}, proved by
Hara and Hino \cite{HaraHino}): for $x_{1},x_{2},\dots ,x_{m}\geq 0$,%
\begin{equation*}
\sum_{s_{1}+\cdots +s_{m}=k,s_{i}\geq 0}\frac{x_{1}^{\frac{s_{1}}{p}}}{%
\left( \frac{s_{1}}{p}\right) !}\cdots \frac{x_{m}^{\frac{s_{m}}{p}}}{\left( 
\frac{s_{m}}{p}\right) !}\leq p^{2m-2}\frac{\left( x_{1}+\cdots
+x_{m}\right) ^{\frac{k}{p}}}{\left( \frac{k}{p}\right) !}\text{,}
\end{equation*}%
we have%
\begin{eqnarray*}
&&\sum_{i_{0}+\cdots +i_{m-1}=k\text{, }i_{s}\leq k-1}\left\Vert \pi
_{i_{0}}\left( Y_{s,t_{1}}^{1}\right) \otimes \cdots \otimes \pi
_{i_{m-1}}\left( Y_{t_{m-1},t_{m}}^{1}\right) -\pi _{i_{0}}\left(
Y_{s,t_{1}}^{2}\right) \otimes \cdots \otimes \pi _{i_{m-1}}\left(
Y_{t_{m-1},t_{m}}^{2}\right) \right\Vert \\
&\leq &C_{p}p^{2m}\sum_{s=1}^{k-1}\left( \omega ^{\alpha }\left( s,t\right)
^{\frac{k-s}{p}}\left( \sum_{j=0}^{m-1}\left\Vert \pi _{s}\left(
Y_{t_{j},t_{j+1}}^{1}\right) -\pi _{s}\left( Y_{t_{j},t_{j+1}}^{2}\right)
\right\Vert \right) \right)
\end{eqnarray*}%
Then, by using $\left( \ref{inner bound on m}\right) $ and $\left( \ref%
{inner estimate of accumulated level k}\right) $, we have%
\begin{eqnarray*}
&&\left\Vert \pi _{k}\left( Y_{s,t}^{1}\right) -\pi _{k}\left(
Y_{s,t}^{2}\right) \right\Vert \\
&\leq &C_{p}p^{2m}\sum_{s=1}^{k}\left( \omega ^{\alpha }\left( s,t\right) ^{%
\frac{k-s}{p}}\left( \sum_{j=0}^{m-1}\left\Vert \pi _{s}\left(
Y_{t_{j},t_{j+1}}^{1}\right) -\pi _{s}\left( Y_{t_{j},t_{j+1}}^{2}\right)
\right\Vert \right) \right) \\
&\leq &C_{p,\gamma }\exp \left( C_{p,\gamma }\alpha ^{-1}\omega ^{\alpha
}\left( s,t\right) \right) \left( \alpha ^{\frac{k}{p}}\left( \left\vert
f^{1}-f^{2}\right\vert _{\gamma -1}+\left\Vert \pi _{1}\left(
Y_{s_{0}}^{1}\right) -\pi _{1}\left( Y_{s_{0}}^{2}\right) \right\Vert
\right) +\sum_{n=1}^{\left[ p\right] }\alpha ^{\frac{\left( k-n\right) \vee 0%
}{p}}d_{p,\left[ s,t\right] }^{n,\alpha }\left( X^{1},X^{2}\right) \right)
\end{eqnarray*}%
Thus, for any $\alpha \in (0,\delta _{p,\gamma }]$ and any $\left[ s,t\right]
$ satisfying $\omega ^{\alpha }\left( s,t\right) >\alpha $, we have%
\begin{eqnarray}
&&\left\Vert \pi _{k}\left( Y_{s,t}^{1}\right) -\pi _{k}\left(
Y_{s,t}^{2}\right) \right\Vert  \label{main continuity result} \\
&\leq &C_{p,\gamma }\exp \left( C_{p,\gamma }\alpha ^{-1}\omega ^{\alpha
}\left( s,t\right) \right)  \notag \\
&&\times \left( \omega ^{\alpha }\left( s,t\right) ^{\frac{k}{p}}\left(
\left\vert f^{1}-f^{2}\right\vert _{\gamma -1}+\left\Vert \pi _{1}\left(
Y_{s}^{1}\right) -\pi _{1}\left( Y_{s}^{2}\right) \right\Vert \right)
+\sum_{n=1}^{\left[ p\right] }\omega ^{\alpha }\left( s,t\right) ^{\frac{%
\left( k-n\right) \vee 0}{p}}d_{p,\left[ s,t\right] }^{n,\alpha }\left(
X^{1},X^{2}\right) \right) \text{,}  \notag
\end{eqnarray}%
On the other hand, based on $\left( \ref{inner estimate when interval is
less than delta}\right) $, $\left( \ref{main continuity result}\right) $
also holds when $\omega ^{\alpha }\left( s,t\right) \leq \alpha $ for $%
\alpha \in (0,\delta _{p,\gamma }]$. To extend $\left( \ref{main continuity
result}\right) $ to all $\alpha \in (0,1]$, by letting $\alpha =\delta
_{p,\gamma }$ in $\left( \ref{main continuity result}\right) $ and using $%
\omega \left( s,u\right) +\omega \left( u,t\right) \leq \omega \left(
s,t\right) $, $\forall s\leq u\leq t$, we have, for $\left[ s,t\right] $
satisfying $\omega \left( s,t\right) \leq 1$,%
\begin{eqnarray*}
&&\left\Vert \pi _{k}\left( Y_{s,t}^{1}\right) -\pi _{k}\left(
Y_{s,t}^{2}\right) \right\Vert \\
&\leq &C_{p,\gamma }\left( \omega \left( s,t\right) ^{\frac{k}{p}}\left(
\left\vert f^{1}-f^{2}\right\vert _{\gamma -1}+\left\Vert \pi _{1}\left(
Y_{s}^{1}\right) -\pi _{1}\left( Y_{s}^{2}\right) \right\Vert \right)
+\sum_{n=1}^{\left[ p\right] }\omega \left( s,t\right) ^{\frac{\left(
k-n\right) \vee 0}{p}}d_{p,\left[ s,t\right] }^{n}\left( X^{1},X^{2}\right)
\right) \text{.}
\end{eqnarray*}%
Then by using definition of control $\omega ^{\alpha }$, it can be proved
that $\left( \ref{inner estimate when interval is less than delta}\right) $
holds for any $\left[ s,t\right] $ satisfying $\omega \left( s,t\right) \leq
1$ and for any $\alpha \in (0,1]$. Then by following the same argument after 
$\left( \ref{inner estimate when interval is less than delta}\right) $, we
have $\left( \ref{main continuity result}\right) $ holds for any $\alpha \in
(0,1]$ and any $\left[ s,t\right] \subseteq \left[ 0,T\right] $. \ By using
sub-additivity of a control, $\left( \ref{main continuity result}\right) $
holds with $\left\Vert \pi _{k}\left( Y_{s,t}^{1}\right) -\pi _{k}\left(
Y_{s,t}^{2}\right) \right\Vert $ replaced by $d_{p,\left[ s,t\right]
}^{k}\left( Y^{1},Y^{2}\right) $.
\end{proof}

\begin{proof}[Proof of Corollary \protect\ref{Corollary error of ODE
approximation}]
\label{Proof of Corollary error of ODE approximation}Fix $D=\left\{
t_{j}\right\} _{j=0}^{n}\subset \left[ 0,T\right] $. That $y^{D}$ takes
value in $G^{\left[ p\right] }\left( \mathcal{U}\right) $ follows from Lemma %
\ref{Lemma simple property of solution of ODE} on p\pageref{Lemma simple
property of solution of ODE}. We assume $\xi =1$. Otherwise, we replace $Y$
and $y^{D}$ by $\xi ^{-1}\otimes Y$ and $\xi ^{-1}\otimes y^{D}$, and
replace $f$ by $f\left( \cdot +\pi _{1}\left( \xi \right) \right) $.

For $j=0,1,\dots ,n$, let $Z^{j}:\left[ t_{j},T\right] \rightarrow G^{\left[
p\right] }\left( \mathcal{V}\right) $ be the solution to the RDE%
\begin{equation*}
dZ^{j}=f\left( Z^{j}\right) dX\text{, }Z_{t_{j}}^{j}=y_{t_{j}}^{D}\text{.}
\end{equation*}%
Then $Z_{t_{j}}^{0}=Y_{t_{j}}$, $Z_{t_{j}}^{j}=y_{t_{j}}^{D}$, and%
\begin{eqnarray}
Y_{t_{j}}-y_{t_{j}}^{D} &=&\tsum\nolimits_{i=0}^{j-1}\left(
Z_{t_{j}}^{i}-Z_{t_{j}}^{i+1}\right)
\label{inner Corollary error of ODE approximation 1} \\
&=&\tsum\nolimits_{i=0}^{j-1}\left( \left(
Z_{t_{i+1}}^{i}-Z_{t_{i+1}}^{i+1}\right) \otimes
Z_{t_{i+1},t_{j}}^{i}+Z_{t_{i+1}}^{i+1}\otimes \left(
Z_{t_{i+1},t_{j}}^{i}-Z_{t_{i+1},t_{j}}^{i+1}\right) \right)  \notag \\
&=&\tsum\nolimits_{i=0}^{j-1}\left( y_{t_{i}}^{D}\otimes \left(
Z_{t_{i},t_{i+1}}^{i}-y_{t_{i},t_{i+1}}^{D}\right) \otimes
Z_{t_{i+1},t_{j}}^{i}+y_{t_{i+1}}^{D}\otimes \left(
Z_{t_{i+1},t_{j}}^{i}-Z_{t_{i+1},t_{j}}^{i+1}\right) \right) \text{.}  \notag
\end{eqnarray}%
Based on $\left( \ref{estimate of difference in main theorem}\right) $ in
Theorem \ref{Theorem existence and uniqueness} on p\pageref{estimate of
difference in main theorem} (difference between RDE solution and ODE
solution), we have 
\begin{equation}
\left\Vert Z_{t_{i},t_{i+1}}^{i}-y_{t_{i},t_{i+1}}^{D}\right\Vert \leq
C_{p}\left( \omega \left( t_{i},t_{i+1}\right) ^{\frac{\left[ p\right] +1}{p}%
}\vee \omega \left( t_{i},t_{i+1}\right) ^{\left[ p\right] }\right) \text{,}
\label{inner Corollary error of ODE approximation 2}
\end{equation}%
Based again on Theorem \ref{Theorem existence and uniqueness} for the bound
on RDE solution (by decomposing big interval as the union of small
intervals, similar as at $\left( \ref{inner estimate on big interval based
on small intervals}\right) $ on p\pageref{inner estimate on big interval
based on small intervals}), for any $\alpha \in (0,1]$, we have%
\begin{equation}
\left\Vert Z_{t_{i+1},t_{j}}^{i}\right\Vert \leq C_{p}\omega ^{\alpha
}\left( 0,T\right) ^{\frac{1}{p}}\vee \left( \alpha ^{-1}\omega ^{\alpha
}\left( 0,T\right) \right) ^{\left[ p\right] }\text{.}
\label{inner Corollary error of ODE approximation 3}
\end{equation}%
According to Theorem \ref{Theorem continuity} on p\pageref{Theorem
continuity} (continuous dependence of RDE solution on initial value) and $%
\left( \ref{inner Corollary error of ODE approximation 2}\right) $,%
\begin{eqnarray}
\left\Vert Z_{t_{i+1},t_{j}}^{i}-Z_{t_{i+1},t_{j}}^{i+1}\right\Vert &\leq
&C_{p,\gamma }\exp \left( C_{p,\gamma }\alpha ^{-1}\omega ^{\alpha }\left(
0,T\right) \right) \left\Vert \pi _{1}\left( Z_{t_{i+1}}^{i}\right) -\pi
_{1}\left( Z_{t_{i+1}}^{i+1}\right) \right\Vert
\label{inner Corollary error of ODE approximation 4} \\
&=&C_{p,\gamma }\exp \left( C_{p,\gamma }\alpha ^{-1}\omega ^{\alpha }\left(
0,T\right) \right) \left\Vert \pi _{1}\left( Z_{t_{i},t_{i+1}}^{i}\right)
-\pi _{1}\left( y_{t_{i},t_{i+1}}^{D}\right) \right\Vert  \notag \\
&\leq &C_{p,\gamma }\exp \left( C_{p,\gamma }\alpha ^{-1}\omega ^{\alpha
}\left( 0,T\right) \right) \left( \omega \left( t_{i},t_{i+1}\right) ^{\frac{%
\left[ p\right] +1}{p}}\vee \omega \left( t_{i},t_{i+1}\right) ^{\left[ p%
\right] }\right) \text{.}  \notag
\end{eqnarray}%
For $k=1,2,\dots ,\left[ p\right] $, denote%
\begin{equation*}
M_{k}:=\max_{l=0,1,\dots ,k}\max_{j=0,1,\dots ,n}\left\Vert \pi _{l}\left(
y_{t_{j}}^{D}\right) \right\Vert \text{.}
\end{equation*}%
Combining $\left( \ref{inner Corollary error of ODE approximation 1}\right) $%
, $\left( \ref{inner Corollary error of ODE approximation 2}\right) $, $%
\left( \ref{inner Corollary error of ODE approximation 3}\right) $ and $%
\left( \ref{inner Corollary error of ODE approximation 4}\right) $, we have,
for $k=1,2,\dots \left[ p\right] $,%
\begin{eqnarray}
&&\max_{j=0,1,\dots ,n}\left\Vert \pi _{k}\left( Y_{t_{j}}\right) -\pi
_{k}\left( y_{t_{j}}^{D}\right) \right\Vert
\label{inner estimate of difference between ODE and RDE} \\
&\leq &C_{p,\gamma }M_{k-1}\exp \left( C_{p,\gamma }\alpha ^{-1}\omega
^{\alpha }\left( 0,T\right) \right) \left( \tsum\nolimits_{j=0}^{n-1}\omega
\left( t_{j},t_{j+1}\right) ^{\frac{\left[ p\right] +1}{p}}\vee \omega
\left( t_{i},t_{i+1}\right) ^{\left[ p\right] }\right) \text{.}  \notag
\end{eqnarray}

Then by mathematical induction, we prove 
\begin{equation}
M_{\left[ p\right] }\leq C_{p,\gamma }\exp \left( C_{p,\gamma }\left( \omega
^{\alpha _{0}}\left( 0,T\right) +\alpha ^{-1}\omega ^{\alpha }\left(
0,T\right) \right) \right) \text{.}  \label{inner bound on M[p]}
\end{equation}%
It is clear that $M_{0}=1$. For $k=1,\dots ,\left[ p\right] $, suppose 
\begin{equation}
M_{k-1}\leq C_{p,\gamma }\exp \left( C_{p,\gamma }\left( \omega ^{\alpha
_{0}}\left( 0,T\right) +\alpha ^{-1}\omega ^{\alpha }\left( 0,T\right)
\right) \right) \text{.}  \label{inner inductive hypothesis in Corollary}
\end{equation}%
Based on similar estimates as that leads to $\left( \ref{inner Corollary
error of ODE approximation 3}\right) $, we have, for $\alpha \in (0,1]$ and $%
k=1,2,\dots ,\left[ p\right] $, ($\xi =1$) 
\begin{equation}
\max_{j=0,1,\dots ,n}\left\Vert \pi _{k}\left( Y_{t_{j}}\right) \right\Vert
\leq C_{p}\left( \omega ^{\alpha }\left( 0,T\right) ^{\frac{k}{p}}\vee
\left( \alpha ^{-1}\omega ^{\alpha }\left( 0,T\right) \right) ^{k}\right) 
\text{,}  \label{inner bound on Y in Corollary}
\end{equation}%
Then by combining $\left( \ref{inner estimate of difference between ODE and
RDE}\right) $, $\left( \ref{inner inductive hypothesis in Corollary}\right) $
with $\left( \ref{inner bound on Y in Corollary}\right) $, we have, with $%
\alpha _{0}:=\max_{t_{j}\in D}\omega \left( t_{j},t_{j+1}\right) $, ($%
M_{k-1}\geq M_{0}=1$) 
\begin{eqnarray*}
\max_{j=0,1,\dots ,n}\left\Vert \pi _{k}\left( y_{t_{j}}^{D}\right)
\right\Vert &\leq &C_{p,\gamma }M_{k-1}\exp \left( C_{p,\gamma }\alpha
^{-1}\omega ^{\alpha }\left( 0,T\right) \right) \left(
1+\tsum\nolimits_{j=0}^{n-1}\omega \left( t_{j},t_{j+1}\right) ^{\frac{\left[
p\right] +1}{p}}\vee \omega \left( t_{j},t_{j+1}\right) ^{\left[ p\right]
}\right) \\
&\leq &C_{p,\gamma }\exp \left( C_{p,\gamma }\left( \omega ^{\alpha
_{0}}\left( 0,T\right) +\alpha ^{-1}\omega ^{\alpha }\left( 0,T\right)
\right) \right) \tprod\nolimits_{j=0}^{n-1}\exp \left( C_{p}\omega \left(
t_{j},t_{j+1}\right) \right) \\
&\leq &C_{p,\gamma }\exp \left( C_{p,\gamma }\left( \omega ^{\alpha
_{0}}\left( 0,T\right) +\alpha ^{-1}\omega ^{\alpha }\left( 0,T\right)
\right) \right) \text{.}
\end{eqnarray*}%
Then%
\begin{equation*}
M_{k}\leq M_{k-1}+\max_{j=0,1,\dots ,n}\left\Vert \pi _{k}\left(
y_{t_{j}}^{D}\right) \right\Vert \leq C_{p,\gamma }\exp \left( C_{p,\gamma
}\left( \omega ^{\alpha _{0}}\left( 0,T\right) +\alpha ^{-1}\omega ^{\alpha
}\left( 0,T\right) \right) \right) \text{.}
\end{equation*}

Combining $\left( \ref{inner estimate of difference between ODE and RDE}%
\right) $ with $\left( \ref{inner bound on M[p]}\right) $, we have%
\begin{equation*}
\left\Vert Y_{T}-y_{T}^{D}\right\Vert \leq C_{p,\gamma }\exp \left(
C_{p,\gamma }\left( \omega ^{\alpha _{0}}\left( 0,T\right) +\alpha
^{-1}\omega ^{\alpha }\left( 0,T\right) \right) \right) \left(
\tsum\nolimits_{j=0}^{n-1}\omega \left( t_{j},t_{j+1}\right) ^{\frac{\left[ p%
\right] +1}{p}}\vee \omega \left( t_{i},t_{i+1}\right) ^{\left[ p\right]
}\right) \text{.}
\end{equation*}
\end{proof}

\section{Acknowledgments}

The research of both authors are supported by European Research Council
under the European Union's Seventh Framework Programme (FP7-IDEAS-ERC) / ERC
grant agreement nr. 291244. The authors acknowledge the support of the
Oxford-Man Institute.

\end{document}